\documentclass[a4paper]{amsart}
\pdfoutput=1 %

\usepackage{latexsym}
\usepackage{amssymb}
\usepackage{amsmath}
\usepackage{amsfonts}
\usepackage{amsthm}
\usepackage[all,knot,poly,2cell]{xy}
   \UseAllTwocells
\usepackage{xspace}
\usepackage{mathtools}
\usepackage{stmaryrd} %
\usepackage[french,english]{babel}

\usepackage[T1]{fontenc}
\usepackage{eucal}

\usepackage{aliascnt} %
\usepackage[colorlinks=true,linkcolor=blue,citecolor=blue,urlcolor=blue,citebordercolor={0 0 1},urlbordercolor={0 0 1},linkbordercolor={0 0 1},linktoc=all]{hyperref}
\usepackage[nameinlink]{cleveref}
\crefname{enumi}{part}{parts}

\usepackage[inline]{enumitem}
\setlist[enumerate]{font=\normalfont}

\makeatletter
\newcommand{\customitem}[2]{%
  \expandafter\let\expandafter\CI@old@label\csname label\@listctr\endcsname
  \expandafter\let\expandafter\CI@old@ref\csname the\@listctr\endcsname
  \expandafter\def\csname label\@listctr\endcsname{#1}%
  \expandafter\def\csname the\@listctr\endcsname{#2}%
  \item
  \expandafter\let\csname label\@listctr\endcsname\CI@old@label
  \expandafter\let\csname the\@listctr\endcsname\CI@old@ref
}%
\makeatother

\newcommand{\myitem}[1]{\customitem{(#1)}{\normalfont (#1)}}

\theoremstyle{plain}
  \newtheorem{theorem}{Theorem}[section]
  \newtheorem{corollary}[theorem]{Corollary}
  \newtheorem{lemma}[theorem]{Lemma}
  \newtheorem{proposition}[theorem]{Proposition}

\theoremstyle{definition}
  \newtheorem{definition}[theorem]{Definition}
  \newtheorem{example}[theorem]{Example}
  
  \newtheorem{setup}[theorem]{Setup}
\theoremstyle{remark}
  \newtheorem{remark}[theorem]{Remark}

\numberwithin{equation}{section}

\usepackage{macros}

\date{Mar 28, 2025}

\begin{document}

\title{The \'etale local structure of algebraic stacks}
\author[J.~Alper]{Jarod Alper}
\address{Department of Mathematics, University of Washington, Box 354350, Seattle, WA
98195-4350, USA}
\email{jarod@uw.edu}
\author[J.~Hall]{Jack Hall}
\address{School of Mathematics \& Statistics, The University of Melbourne, Parkville, VIC, 3010, Australia}
\email{jack.hall@unimelb.edu.au}
\author[D.~Rydh]{David Rydh}
\address{Department of Mathematics, KTH Royal Institute of Technology, SE-100 44 Stockholm, Sweden}
\email{dary@math.kth.se}
\begin{abstract}
We prove that an algebraic stack with affine stabilizers over an arbitrary base is \'etale-locally a quotient stack around
any point with a linearly reductive stabilizer. This generalizes earlier work
by the authors of this article (stacks over algebraically closed fields) and by Abramovich,
Olsson and Vistoli (stacks with finite inertia). In addition, we prove a number
of foundational results, which are new even over a field. 
These include various coherent completeness and effectivity results for
adic sequences of algebraic stacks. Finally, we give several applications of
our results and methods, such as structure theorems for linearly reductive
group schemes and generalizations to the relative setting of Sumihiro's theorem
on torus actions and Luna's \'etale slice theorem.
\end{abstract}
\maketitle
\setcounter{tocdepth}{1}
\tableofcontents

\section{Introduction} \label{S:intro}

This paper offers a broad generalization and extension of our previous work \cite{luna-field}, which provided a local structure theorem for algebraic stacks of finite type over an algebraically closed field.  In addition to establishing a local structure theorem for algebraic stacks defined over an arbitrary base (\Cref{T:base}), we prove a number of  foundational results that are new even over an algebraically closed field. %
This includes a general coherent completeness result for algebraic stacks (\Cref{T:complete}), which becomes particularly powerful when coupled with Tannaka duality (see \S \ref{SS:tannaka}).  We also prove an effectivity theorem for adic sequences of noetherian algebraic stacks (\Cref{T:effectivity}), analogous to Grothendieck's result on algebraization of formal schemes \cite[III.5.4.5]{EGA}. While of independent interest, this is one of the key ingredients for the other main theorems in this paper---including the local structure theorem.

We prove several other foundational results and provide numerous applications to equivariant geometry and moduli theory. To highlight a few, we prove that adequate moduli spaces are universal for maps to algebraic spaces (\Cref{T:universal}); this implies that GIT quotients in positive characteristic are categorical quotients in algebraic spaces, which was formerly unknown---even over an algebraically closed field.  We establish that various properties (e.g., the resolution property\footnote{That is, every quasi-coherent sheaf is a quotient of a direct sum of vector bundles. If the algebraic stack is quasi-compact and quasi-separated with affine stabilizers, then this is equivalent to being expressible as $[U/\GL_n]$, where $U$ is a quasi-affine scheme \cite{MR2108211,gross-resolution}.}) and objects (e.g., morphisms, finite \'etale covers, vector bundles), defined on a closed substack $\cX_0$ of $\cX$, extend to an \'etale neighborhood of $\cX_0$ if there is an affine good moduli space $\cX \to X$ (see \Cref{P:extension-resprop,P:extension-morphisms,P:extension-nicefund,P:extension-linfund}).
These extension results are technical but extremely useful in local-to-global arguments.
We also prove a generalization of Sumihiro's theorem on torus actions (\Cref{C:sumihiro-torus}): an algebraic space $X$ 
over a base
$S$
with an action of a torus $T \to S$ has $T$-equivariant affine \'etale neighborhoods.

\subsection{A  local structure theorem}
We prove that if $\cX$ is an algebraic stack satisfying some mild assumptions, then a point $x \in |\cX|$ with linearly reductive stabilizer has an \'etale neighborhood $([\Spec A / \GL_n], w) \to (\cX, x)$ inducing an isomorphism on residual gerbes.  This is the conclusion of the following theorem in the special case that $\cW_0 = \cG_x$.

\begin{theorem}[Local structure]\label{T:base}  
  Let $S$ be a quasi-separated algebraic space and 
  $\cX$ be an algebraic stack, locally of finite presentation, and quasi-separated
  over $S$, with affine stabilizers.  
  Let $x\in |\cX|$ be a point with
  residual gerbe $\cG_x$ and image $s\in |S|$ such that the residue field extension $\kappa(x)/\kappa(s)$ is finite.  
  Let $h_0 \colon \stk{W}_0 \to \cG_x$ be a smooth (resp.\ \'etale) morphism where $\stk{W}_0$ is a gerbe over the spectrum of a field and has linearly reductive stabilizer.  
  Then there exist an algebraic stack $\cW = [\Spec A / \GL_n]$ and a point $w \in |\cW|$ with an identification $\cG_w = \cW_0$ together with a cartesian diagram 
\[
\xymatrix{
\mathllap{\cG_w = \;}\cW_0 \ar[r]^{\smash{h_0}} \ar[d]  & \cG_x \ar[d] \\
\mathllap{[\Spec A/\GL_n] = \;}\cW \ar[r]^h   & \cX,
}
\]
where $h\colon (\stk{W},w) \to (\cX,x)$ is a 
smooth (resp.\ \'etale) pointed morphism and $w$ is closed in its fiber over $s$.
   Moreover, if $\cX$ has separated (resp.\ affine) diagonal and $h_0$ is representable, then $h$
  can be arranged to be representable (resp.\ affine). 
\end{theorem}

The flexibility of allowing $\cW_0 \to \cG_x$ to be a smooth or \'etale morphism (rather than an isomorphism) is useful when the stabilizer $G_x$ of $x$ is not linearly reductive, which is a particularly restrictive condition in positive characteristic.  Assuming that the residual gerbe $\cG_x = B G_x$ is trivial, one can choose any linearly reductive subgroup $H \subset G_x$ such that $G_x/H$ is smooth (resp. \'etale) and apply the theorem to the smooth (resp. \'etale) morphism $\cW_0 = BH \to BG_x$. 

\begin{remark}[Known results] 
  In the case that $\cX$ has finite inertia and $h_0$ is an isomorphism, this theorem had been established in \cite[Thm.~3.2]{aov}. When $S=\Spec k$ and $k$ is algebraically closed, this theorem follows from \cite[Thm.~1.1]{luna-field}, which established the stronger result that $\cW$ can be written as $[\Spec A/H]$.
\end{remark}

\begin{remark}[Further generalizations]
  In \Cref{T:etale-local-gms-connected} and \Cref{T:refinement}, 
  we provide a more refined description of the stack
  $\stk{W}$ as a quotient stack $[\Spec A/G]$ for a specific group scheme $G$
  in terms of properties of the gerbe $\stk{W}_0$.  For instance, if the stabilizer group of $\stk{W}_0$ is connected, 
  we can arrange that $G$ is split reductive, and if $\cW_0 = BG_0$ is neutral, we can arrange that $G$ is a 
  deformation of $G_0$ which is linearly reductive under mild 
  characteristic assumptions.

  In work with Halpern-Leistner \cite{non-local}, this theorem is generalized to allow the case where $\cW_0$ is an arbitrary linearly fundamental stack (rather than a gerbe over a field), where $x \in |\cX|$ is an arbitrary point (rather than the finiteness of $\kappa(x)/\kappa(s)$), and where $\cW_0 \to \cG_x$ is syntomic (rather than smooth or \'etale). 
\end{remark}

The proof of \Cref{T:base} is given in \Cref{S:local-structure} and
follows the same general strategy as the proof of \cite[Thm.~1.1]{luna-field}:
\begin{enumerate}
\item We begin by constructing smooth infinitesimal deformations $h_n\colon
  \stk{W}_n\to \cX_n$ where $\cX_n$ is the $n$th infinitesimal
  neighborhood of $\cG_x$ in $\cX$. 
\item We show that the system $\stk{W}_n$ effectivizes to a \emph{coherently
  complete} stack $\widehat{\stk{W}}$. This is \Cref{T:effectivity}.
\item Tannaka duality~\cite{hr-tannaka} (see also \S \ref{SS:tannaka})
 then gives us a unique formally smooth
  morphism $\widehat{h}\colon \widehat{\stk{W}}\to \cX$.
\item Finally we apply equivariant Artin
  algebraization~\cite[App.~A]{luna-field} to approximate $\widehat{h}$ with a
  smooth morphism $h\colon \stk{W}\to \cX$.
\end{enumerate}
Step (1) follows by standard infinitesimal deformation theory.
Step (2) is the main technical result of this paper. \Cref{T:effectivity} is far more general
than the related results in \cite{luna-field}---even over an algebraically closed
field.  Step (3) was handled in~\cite{hr-tannaka}. 
Steps (1)--(3) are summarized in \Cref{T:microlocalization}.
The equivariant Artin algebraization results established in
\cite[App.~A]{luna-field} are sufficient for step (4).

\begin{remark}[Existence theorem]
\Cref{T:base} and its refinements
are fundamental ingredients in the recent article of the first author 
with Halpern-Leistner and Heinloth on
establishing necessary and sufficient conditions for an algebraic stack to admit a good 
moduli space \cite{ahlh}.
\end{remark}

\subsection{Coherent completeness} The notion of an algebraic stack that is coherently complete with respect to a closed substack plays an essential role in this paper.  It is not only used in Step (2) above but it is featured in many of our other results and techniques.  The definition first appeared in \cite[Defn.~2.1]{luna-field}.

\begin{definition} \label{D:coherently-complete}
Let $\cZ \subseteq \cX$ be a closed immersion of noetherian algebraic stacks. We say that the pair $(\cX,\cZ)$ is \emph{coherently complete} (or $\cX$ is {\it coherently complete along $\cZ$}) if the natural functor
\begin{equation*}
  \COH(\cX) \to \varprojlim_n \COH(\thck{\cX}{\cZ}{n}),
\end{equation*}
from the abelian category of coherent sheaves on $\cX$ to the category of
projective systems of coherent sheaves on the $n$th nilpotent
thickenings $\cX^{[n]}_{\cZ}$ of $\cZ \subset \cX$, is an equivalence
of categories. 
\end{definition}

A noetherian affine scheme $\Spec A$ is coherently complete along $\Spec(A/I)$ if and only if $A$ is $I$-adically complete
(\Cref{EX:adic-ring}).
The following statement was an essential ingredient in all of the main results of \cite{luna-field}: if $A$ is a noetherian
$k$-algebra, where $k$ is a field, and $G$ is a
linearly reductive affine group scheme over $k$ acting on $\spec A$ such that there is a $k$-point fixed by $G$ and
the ring of invariants $A^G$ is a complete local ring, then the
quotient stack $[\spec A /G]$ is coherently complete along the residual gerbe of its unique closed point \cite[Thm.~1.3]{luna-field}.

In this article, coherent completeness also features prominently. In fact, we
establish the following generalization of \cite[Thm.~1.3]{luna-field}, where we
do not assume a
priori that $\cX$ has the resolution property, only that the closed substack
$\cZ$ has it.

\begin{theorem}[Coherent completeness]\label{T:complete}
  Let $\cX$ be a noetherian algebraic stack with affine diagonal
  and good moduli space $\pi \colon \cX \to X =\Spec A$. Let
  $\cZ \subseteq \cX$ be a closed substack defined by a
  coherent ideal $\cI$.  Let $I=\Gamma(\cX,\cI)$.  If $\cZ$ has the
  resolution property, then $\cX$ is coherently complete along
  $\cZ$ if and only if $A$ is $I$-adically complete. If this is the case,
  then $\cX$ has the resolution property.
\end{theorem}

This theorem is proved in \Cref{S:completeness-and-effectivity}.
The difference between the statement above and formal GAGA for good moduli space morphisms  is that the statement above asserts that $\cX$ is coherently complete along $\cZ$ and not
merely along $\pi^{-1}(\pi(\cZ))$.  Indeed, as a consequence of this theorem, we can easily deduce the following version of formal GAGA (\Cref{C:formal-gaga}), which had been established in  \cite[Thm.~1.1]{geraschenko-zb_fGAGA} with the additional hypotheses that $\cX$ has the resolution property and $I$ is maximal, and in \cite[Cor.~4.14]{luna-field} in the case that $\cX$ is defined over a field and $I$ is maximal.  

\begin{corollary}[Formal GAGA]\label{C:formal-gaga}
Let $\cX$ be a noetherian algebraic stack with affine diagonal. Suppose there exists a good moduli space $\pi \colon \cX \to \Spec A$, where $A$ is noetherian and $I$-adically complete.  Suppose that either
\begin{enumerate}
	\item $I \subset A$ is a maximal ideal; or
	\item $\cX \times_{\Spec A} \Spec(A/I)$ has the resolution property.
\end{enumerate}
Then $\cX$ has the resolution property and is coherently complete along $\cX \times_{\Spec A} \Spec(A/I)$.
\end{corollary}

\subsection{Effectivity}\label{S:def-of-adic-sequence}
The key method to prove many of the results in this paper
is an effectivity result
for algebraic stacks. This is similar in spirit to Grothendieck's
result on algebraization of formal schemes \cite[III.5.4.5]{EGA}. 

\begin{definition} \label{D:adic}
An \emph{adic sequence} is a sequence of closed immersions
\[
\cX_0\hookrightarrow\cX_1\hookrightarrow\dots
\]
of noetherian algebraic stacks such that if $\cI_{(j)}$ denotes the coherent
sheaf of ideals defining $u_{0j}\colon \cX_0 \hookrightarrow \cX_j$, then
$\cI_{(j)}^{i+1}$ defines $u_{ij} \colon \cX_i \hookrightarrow \cX_j$ for
every $i\leq j$.
\end{definition}

The sequence of infinitesimal thickenings of a closed substack of a noetherian algebraic stack is adic. 

\begin{definition} \label{D:coherent-completion}
Let $\{\cX_n\}_{n\geq 0}$ be an adic sequence of algebraic stacks. An algebraic stack $\hat{\cX}$ is a \emph{completion} of $\{\cX_n\}$ if
\begin{enumerate}
\item there are compatible closed immersions $\cX_n\hookrightarrow
  \hat{\cX}$ for all $n$ such that $\cX_n$ is the $n$th infinitesimal neighborhood of $\cX_0$ in $\hat{\cX}$;
\item $\hat{\cX}$ is noetherian with affine diagonal; and
\item $\hat{\cX}$ is coherently complete along $\cX_0$.
\end{enumerate}
\end{definition}

By Tannaka duality (see \S \ref{SS:tannaka}), the completion is unique if it exists. Moreover,
Tannaka duality  implies that
if the completion exists, then it is the colimit of $\{\cX_n\}_{n\geq 0}$ in
the category of noetherian stacks with quasi-affine diagonal (and in the
category of algebraic stacks with affine stabilizers if $\cX_0$ is
excellent). 

The following result %
is our main effectivity theorem.  An algebraic stack $\cX$ is {\it linearly fundamental} if $\cX$ is cohomologically affine and isomorphic to $[U/\GL_{n,\ZZ}]$ for an affine scheme $U$.  An important example of a linearly fundamental stack is the quotient stack of an affine scheme by a linearly reductive group scheme.  

\begin{theorem}[Effectivity]\label{T:effectivity}
  Let $\{\cX_n\}_{n\geq 0}$ be an adic sequence of noetherian algebraic
  stacks. If $\cX_0$ is linearly fundamental, then the 
  completion $\hat{\cX}$ exists and is 
  linearly fundamental.
\end{theorem}
We prove \Cref{T:effectivity} in three stages of increasing
generality in \Cref{S:completeness-and-effectivity}. The case of characteristic
zero (\Cref{T:effectivity:char-0}) is reasonably
straightforward whereas the case
of positive and mixed characteristic requires a short detour through group
schemes (\Cref{P:nice-deformations}).
In \Cref{S:local-structure}, we then use \Cref{T:effectivity}
to establish the existence of formally smooth neighborhoods and completions, such as the following results.

\begin{theorem}[Formal neighborhoods]\label{T:microlocalization}
  Let $\cX$ be noetherian algebraic stack with affine stabilizers
  and $\cX_0 \subset \cX$ be a locally closed substack. Let
  $h_0 \colon \cW_0 \to \cX_0$ be a syntomic (e.g., smooth) morphism. Assume that $\cW_0$ is linearly fundamental and that its good moduli space is quasi-excellent.  Then there is a cartesian diagram
  \[
    \xymatrix{
      \cW_0 \ar[d]^{h_0} \ar@{-->}[r]			& \hat{\cW} \ar@{-->}[d]^h \\
      \cX_0 \ar[r] & \cX, }
\]
where $h \colon \hat{\cW} \to \cX$ is flat and
  $\hat{\cW}$ is noetherian, linearly fundamental and coherently complete along
  $\cW_0$. 
\end{theorem}

Applying the theorem above to the case when $\cW_0 = \cX_0$ is the residual gerbe of a point with linearly reductive stabilizer gives the existence of completions.

\begin{corollary}[Existence of completions]\label{C:existence-completions}
Let $\cX$ be a noetherian algebraic stack with affine stabilizers. For any point $x \in |\cX|$ with linearly reductive stabilizer, the completion of $\cX$ at $x$ exists and is linearly fundamental.
\end{corollary}

In fact, the last two results are proven more generally for pro-immersions (\Cref{T:microlocalization:general} and  \Cref{C:existence-completions:general}). 

\subsection{Further results and applications}
In the course of establishing the results above, we prove several foundational results of independent interest.  

\subsubsection{Local structure of good moduli spaces}
	We provide the following refinement of \Cref{T:base}: if $\cX$ admits a good moduli space $X$, then  \'etale-locally on $X$, $\cX$ is of the form $[\Spec A / \GL_n]$  (\Cref{T:etale-local-gms}). Also see \Cref{T:etale-local-gms-connected}.
  We also prove that a good moduli space $\cX \to X$ necessarily has affine diagonal as long as $\cX$ has separated diagonal and affine stabilizers (\Cref{T:etale-local-gms}); this result is new even over an algebraically closed field.

\subsubsection{Structure of linearly reductive affine group schemes}
We prove that every linearly reductive group scheme $G\to S$ is \'etale-locally
embeddable (\Cref{C:lin-red-groups}) and canonically an extension
of a finite flat tame group scheme by a smooth linearly reductive group scheme
with connected fibers $G^0_\sm$ (\Cref{T:char-of-lin-reductive-groups}). If $S$
is of equal characteristic, then $G$ is canonically an extension of a finite
\'etale tame group scheme by a linearly reductive group scheme $G^0$ with
connected fibers. In equal positive characteristic, $G^0$ is of multiplicative
type and we say that $G$ is nice.
We also prove that if $(S,s)$ is henselian and $G_s\to \Spec \kappa(s)$ is
linearly reductive, then there exists an embeddable linearly reductive group
scheme $G\to S$ extending $G_s$ (\Cref{P:deformation-linearly-reductive}).

\subsubsection{Representability of local quotient presentations}
We have resolved the issue (see \cite[Question 1.10]{luna-field}) of representability of the local quotient presentation in the presence of a separated diagonal (\Cref{P:refinement}\itemref{P:refinement:separated_diag}).

\subsubsection{Results on adequate moduli spaces}
We prove that adequate moduli spaces are universal for maps to algebraic spaces (\Cref{T:universal}) and establish Luna's fundamental lemma for adequate moduli spaces (\Cref{L:fundamental-lemma}).
We also prove that an adequate moduli space $\cX \to X$, where the closed points of $\cX$ have linearly reductive stabilizers, is necessarily a good moduli space (\Cref{T:adequate+lin-red=>good} and \Cref{C:adequate+lin-red=>good2}).  These foundational results are even new over an algebraically closed field.

\subsubsection{Compact generation and algebraicity}  
We prove compact generation of the derived category of an algebraic stack admitting a good moduli space (\Cref{P:compact-generation}). 
We also prove algebraicity results for stacks parameterizing coherent sheaves (\Cref{T:coh}), Quot schemes (\Cref{C:quot}), and Hom stacks (\Cref{T:hom}).

\subsubsection{Deforming objects}
We prove that various objects and properties, defined on a closed substack $\cX_0$ of $\cX$, extend to an \'etale neighborhood of $\cX_0$ if there is an affine good moduli space $\cX \to X$ (see \Cref{P:extension-gerbes,P:extension-groups,P:extension-resprop,P:extension-morphisms,P:extension-nicefund,P:extension-linfund}). These are applications of deformation results for henselian pairs of algebraic stacks (see \Cref{P:deformation-resprop,P:deformation-sections,P:deformation-morphisms,P:deformation-linearly-fundamental,P:deformation-linearly-reductive}). These results are new even over an algebraically closed field.

\subsubsection{Equivariant geometry}
  We provide generalizations of Sumihiro's theorem on torus actions (\Cref{T:sumihiro} and \Cref{C:sumihiro-torus}) and establish a relative version of Luna's \'etale slice theorem (\Cref{T:luna}).

\subsubsection{Completions and henselizations}
  In addition to establishing the existence of completions (\Cref{C:existence-completions}), we prove the existence of henselizations of algebraic stacks at points with linearly reductive stabilizer (\Cref{T:henselizations-general}).
	We prove that two algebraic stacks are \'etale-locally isomorphic near points with linearly reductive stabilizers if and only if they have isomorphic henselizations or completions (\Cref{T:etale-local-equivalences}).

\subsection{Overview} %
\Cref{S:definitions,S:pairs} consist of the basic setup with some applications to adequate moduli spaces.  We provide definitions and properties of reductive group schemes, fundamental stacks and local, henselian and coherently complete pairs.  %

\Cref{S:completeness-and-effectivity,S:local-structure} contain the proofs of the central theorems of this paper:  formal functions for good moduli spaces (\Cref{C:formal-fns}), coherent completeness for good moduli spaces (\Cref{T:complete})%
, effectivity of adic sequences (\Cref{T:effectivity})%
, 
the existence of formal neighborhoods (\Cref{T:microlocalization}),
the existence of completions (\Cref{C:existence-completions}),
and the local structure theorem (\Cref{T:base}).

\Cref{S:applications} contains our applications to algebraic stacks with good moduli spaces. %
\Cref{S:approximation-deformation} contains technical results on approximation of linearly fundamental stacks%
, approximation of good moduli spaces%
, and deformation of objects over henselian pairs. \Cref{S:refinements} uses these results to give refinements of %
the local structure theorem.
\Cref{S:structure} contains our structure results for linearly reductive group schemes.  \Cref{S:further-applications} contains our applications to equivariant geometry, henselizations and algebraicity. %

\subsection{Conventions on good moduli spaces}
Throughout this paper, we use the concepts of cohomologically affine morphisms and adequately affine morphisms:
\begin{definition}
  A quasi-compact and quasi-separated morphism $f \co \cX \to \cY$ of algebraic stacks is {\it cohomologically affine} (resp.\ {\it adequately affine}) if
  \begin{enumerate}
  \item $f_*$ is exact on the category of quasi-coherent $\oh_{\cX}$-modules (resp.\ if for every surjection $\cA \to \cB$ of quasi-coherent $\oh_{\cX}$-algebras, then every section $s$ of $f_*(\cB)$ over a smooth morphism $\Spec A \to \cY$ has a positive power that lifts to a section of $f_*(\cA)$); and
  \item property (1) holds after arbitrary base change $\cY'\to \cY$.
  \end{enumerate}
\end{definition}
In the original definitions, \cite[Defn.\ 3.1]{alper-good}  and \cite[Defn.\ 4.1.1]{alper-adequate}, condition (2) was not required. If $\cY$ has quasi-affine diagonal (e.g., $\cY$ is a quasi-separated algebraic space), then (2) holds automatically (\cite[Prop.\ 3.10(vii)]{alper-good} and \cite[Prop.\ 4.2.1(6)]{alper-adequate}).

We also use throughout the concepts of good moduli spaces \cite[Defn.\ 4.1]{alper-good} and adequate moduli spaces \cite[Defn.\ 5.1.1]{alper-adequate}:
\begin{definition}
  A quasi-compact and quasi-separated morphism $\pi \co \cX \to X$ of algebraic stacks, where $X$ is an algebraic space, is a {\it good moduli space} (resp.\ an {\it adequate moduli space}) if $\pi$ is cohomologically affine (resp.\ adequately affine) and $\oh_X \to \pi_* \oh_{\cX}$ is an isomorphism.
\end{definition}

\begin{definition}\label{D:lin/geom-reductive}
  A group algebraic space $G\to S$ is \fndefn{linearly reductive} (resp.\ \fndefn{geometrically reductive}) if
  \begin{enumerate}
  \item $G\to S$ is flat and of finite presentation;
  \item $G\to S$ is affine; and
  \item $BG\to S$ is a good (resp.\ adequate) moduli space.
  \end{enumerate}
\end{definition}
In \cite[Defn.~12.1]{alper-good} and \cite[Defn.~9.1.1]{alper-adequate},
condition (2) was not required. This is, however, often automatic, see
\Cref{R:linearly-reductive-affine}.

\begin{definition}
  An algebraic stack $\cX$ is said to have the {\it resolution property} if every quasi-coherent $\oh_{\cX}$-module of finite type is a quotient of a locally free sheaf of
  finite rank.
\end{definition}
By the main theorems of \cite{MR2108211} and \cite{gross-resolution}, a quasi-compact and quasi-separated algebraic stack $\cX$ is isomorphic to $[U/\GL_N]$, where $U$ is a quasi-affine scheme and $N$ is a positive integer, if and only if the closed points of $\cX$ have affine stabilizers and $\cX$ has the resolution property. Note that when this is the case, $\cX$ has affine diagonal.

\subsection{Recollection of Tannaka duality}\label{SS:tannaka}  
  If $\cX$ is a locally noetherian algebraic stack and $\cZ \subseteq
  \cX$ is a closed substack, we denote by $\thck{\cX}{\cZ}{n}$ the $n$th order thickening of 
  $\cZ$ in $\cX$, that is, if $\cZ$ is defined by a sheaf of ideals $\shv{I}$, then $\thck{\cX}{\cZ}{n}$ is 
  defined by $\shv{I}^{n+1}$.
  If $i\colon \cZ \to \cX$ denotes the closed 
  immersion, then we write $\thckmap{i}{n}\colon \thck{\cX}{\cZ}{n} \to 
  \cX$ for the $n$th order thickening of~$i$.

We freely use the following form of Tannaka duality, which was established in \cite{hr-tannaka}.  Let $\cX$ be a noetherian algebraic stack
and let $\cZ \subseteq \cX$ be a closed substack such that $\cX$ is coherently complete along $\cZ$. 
Let $\cY$ be a noetherian algebraic stack with affine stabilizers.  Suppose that either 
\begin{enumerate}[label=(\alph*), ref=\alph*]
	\item \label{SS:tannaka:a} $\cX$ is locally the spectrum of a G-ring (e.g., quasi-excellent); or 
	\item  \label{SS:tannaka:b}  $\cY$ has quasi-affine diagonal.  
\end{enumerate}
Then the natural functor
\[
  \Hom(\cX, \cY) \to \ilim_n \Hom\bigl(\cX^{[n]}_{\cZ}, \cY\bigr)
\]
is an equivalence of categories.  This statement follows directly from \cite[Thms.~1.1 and~8.4]{hr-tannaka}; cf.\ the proof of  \cite[Cor.~2.8]{luna-field}.

\subsection{Acknowledgements}
We thank the referee for their careful reading and numerous suggestions that
have improved the text.  We
would also like to thank Johan de Jong, Dan Edidin, Elden Elmanto, Daniel
Halpern-Leistner, Jochen Heinloth, Marc Hoyois, Adeel Khan,
Amalendu Krishna, Svetlana Makarova,
Maria Yakerson and many others for useful conversations.

During the preparation of this paper, the first author was partially supported by the Australian Research Council (DE140101519), the National Science Foundation (DMS-1801976 and DMS-2100088), a Humboldt Fellowship and a Simons Fellowship. The second author was partially supported by the Australian Research Council (DE150101799, DP210103397 and FT210100405). The third author was partially supported by the Swedish Research Council (2011-5599 and 2015-05554) and the G\"oran Gustafsson Foundation for Research in Natural Sciences and Medicine. This collaboration was also supported by the G\"oran Gustafsson Foundation.

\section{Reductive group schemes and fundamental stacks}   \label{S:definitions}

In this section, we recall various notions of reductivity for group schemes and introduce certain classes of algebraic stacks that we will refer to as fundamental, linearly fundamental, and nicely fundamental. We also recall various relations between these notions.  Besides some approximation results at the end, this section is largely expository. 

\subsection{Reductive group schemes}  \label{S:reductive-definition}
Recall that an affine, flat, and finitely presented group algebraic space $G\to S$ is linearly (resp.\ geometrically) reductive if
$BG\to S$ is a good (resp.\ adequate) moduli space
(\Cref{D:lin/geom-reductive}). We also introduce the following notions.
\begin{definition}  \label{D:reductive/nice}
Let $G$ be a group algebraic space which is affine, flat, and of finite presentation over an algebraic space $S$.  We say that $G \to S$ is 
  \begin{enumerate}
  \item \fndefn{embeddable} if $G$ is a closed subgroup of $\GL(\cE)$ for a vector bundle $\cE$ on $S$;
  \item \fndefn{reductive} if $G \to S$ is smooth with reductive and connected geometric fibers \cite[Exp. XIX, Defn.~2.7]{MR0274460}; and  
  \item \fndefn{nice} if there is
    an open and closed normal subgroup $G^0 \subseteq G$ that is of multiplicative type over $S$ such that the \'etale group scheme $H=G/G^0$
    is finite over $S$ and $|H|$ is invertible on $S$. 
  \end{enumerate}
\end{definition}

Linearly reductive group schemes are the focus of this paper as their
representation theory is semi-simple. Geometrically reductive group schemes
appear since in positive characteristic $\GL_n$ is not linearly reductive but
merely geometrically reductive. They also appear as global deformations of
linearly reductive group schemes in mixed characteristic (\Cref{R:char0}).
Nice group schemes is a special class of linearly reductive group schemes
that deform well also in mixed characteristic.

\begin{remark}[Relations between the notions]
For group schemes, we have the implications:
\[
  \mbox{nice} \implies \mbox{linearly reductive} \implies \mbox{geometrically reductive} \Longleftarrow \mbox{reductive}.
\]
The first implication follows since a nice group algebraic space $G$ is an extension of the linearly reductive groups $G^0$ and $H$, and is thus linearly reductive \cite[Prop.~2.17]{alper-good}.
The second implication is immediate from the definitions, and is reversible in characteristic $0$ \cite[Rem.~9.1.3]{alper-adequate}.  The third implication is Seshadri's generalization \cite{seshadri} of Haboush's theorem, and is reversible if $G \to S$ is smooth with geometrically connected fibers \cite[Thm.~9.7.5]{alper-adequate}.  If $k$ is a field of characteristic $p$,  then $\GL_n$ is reductive over $k$ but not linearly reductive, and a finite non-reduced group scheme (e.g., $\alpha_p$) is geometrically reductive but not reductive.
\end{remark}

\begin{remark}[Positive characteristic]\label{R:nice_grp_pos_char} The notion of niceness is particularly useful in positive characteristic and was introduced in 
\cite[Defn.~1.1]{hallj_dary_alg_groups_classifying} for affine group schemes over a field $k$.   If $k$ is a field of characteristic $p$, an affine group scheme $G$ of finite type over $k$ is nice if and only if the connected component of the identity $G^0$ is of multiplicative 
type and $p$ does not divide the number of geometric components of $G$.  In this case, 
by Nagata's theorem \cite{MR0142667} and its generalization to the non-smooth case (cf.\ \cite[Thm.~1.2]{hallj_dary_alg_groups_classifying}), $G$ is nice if and only if it is linearly reductive; moreover, this is also true  over a base of equal characteristic $p$ (\Cref{T:char-of-lin-reductive-groups}). 
\end{remark}

\begin{remark}[Mixed characteristic] \label{R:char0}
 Consider a scheme $S$, a point $s \in S$ and a linearly reductive group scheme
 $G_0$ over $\kappa(s)$. If $G_0$ is nice (e.g., if $s$ has positive
 characteristic), then it deforms to a nice group scheme $G' \to
 S'$ over an \'etale neighborhood $S' \to S$ of $s$
 (\Cref{P:nice-deformations}).

  Deformations of linearly reductive group schemes are more subtle.  
  It is possible for a schemes $S$ to have a closed point $s \in S$ of characteristic 0 which has no open neighborhood of characteristic $0$.  For instance, let $R$ be the localization $\Sigma^{-1} \ZZ[x]$ where $\Sigma$ is
 the multiplicative submonoid generated by the elements $p+x$ as $p$ ranges over
 all primes. Then $S = \Spec R$ is a noetherian and excellent integral
 scheme, and $s = (x) \in S$ is a closed point with residue field $\QQ$ which
 has no characteristic $0$ neighborhood. Also see
 \Cref{SS:noetherian-counterexample}.  In such examples, a linearly reductive group scheme $G_0$ need not deform to a linearly reductive group scheme $G \to S'$
 over an \'etale neighborhood $S' \to S$ of $s$. For example, take
 $G_0=\GL_{2,\kappa(s)}$. However, $G_0$ does deform to a geometrically reductive embeddable group scheme over an \'etale neighborhood of $s$ (\Cref{P:extension-groups}).
\end{remark}

\begin{remark}[Embeddability and geometric reductivity]\label{R:resolution}
Any affine group scheme of finite type over a field is embeddable.  It is not known to which extent general affine group schemes are embeddable---even over the dual numbers \cite{mathoverflow_groups-over-dual-numbers}.   Thomason proved that certain reductive group schemes are embeddable \cite[Cor.~3.2]{thomason}; in particular, if $S$ is a normal, quasi-projective scheme, then every reductive group scheme $G \to S$ is embeddable. 
There is an example \cite[Exp.\ X,\S1.6]{MR0274459} of a 2-dimensional torus over the nodal cubic curve that is not locally isotrivial and hence not Zariski-locally embeddable.   We will eventually show that every linearly
reductive group scheme $G\to S$ is embeddable if $S$ is a normal quasi-projective scheme (\Cref{C:lin-red-embeddable-over-normal}) and always \'etale-locally embeddable (\Cref{C:lin-red-groups}\itemref{C:lin-red:etale-loc-emb}).

If $G$ is a closed subgroup of $\GL(\cE)$ for a vector bundle $\cE$ on an algebraic space $S$, then a generalization of Matsushima's theorem asserts that $G \to S$ is geometrically reductive if and only if the quotient $\GL(\cE)/G$ is affine \cite[Thm.~9.4.1]{alper-adequate}. 

If $S$ is affine and $G \to S$ is embeddable and geometrically reductive, then any quotient stack 
$\cX = [\spec A  / G]$ has the resolution property.  
Indeed, if $G$ is a closed subgroup of  $\GL(\cE)$ for some 
vector bundle $\cE$ of rank $n$ on $S$, then the $(\GL(\cE), \GL_{n,S})$-bitorsor $\Isom_{\oh_S}(\cE, \oh_S^n)$ induces an isomorphism $B_SGL(\cE) \cong B_S\GL_{n}$, and the composition $\cX=[\spec A/G] \to B_SG \to B_S \GL(\cE) \cong B_S\GL_{n}$ is affine, that is $\cX \cong [\spec B/\GL_{n,S}]$. By \cite[Thm.~1.1]{gross-resolution}, $\cX$ has the resolution property.
\end{remark}

\begin{remark}[Affineness]\label{R:linearly-reductive-affine}
In contrast to~\cite{alper-good}, we have only defined linear reductivity for
affine group schemes $G\to S$. We will however prove that
if $G \to S$ is a separated, flat group scheme of finite presentation with
affine fibers such that $BG\to S$ is cohomologically affine,
then $G\to S$ is affine (\Cref{C:lin-red-groups}\itemref{C:lin-red:sep+af=>affine}).
\end{remark}

\subsection{Fundamental stacks} \label{S:fundamental}
In \cite{luna-field}, we dealt with stacks of the form $[\Spec A/G]$ where
$G$ is a linearly reductive group scheme over a field $k$.  In this paper,  we are working over an arbitrary base and
it will be convenient to 
introduce the following classes of quotient stacks.

\begin{definition} \label{D:fundamental}
  Let $\cX$ be an algebraic stack.
  We say that $\cX$ is:
  \begin{enumerate}
  \item \fndefn{fundamental} if $\cX \cong [U/\GL_{n,\ZZ}]$ for an affine scheme $U$ and some integer $n$, i.e., $\cX$ admits an affine morphism to $B\GL_{n,\ZZ}$;
  \item \fndefn{linearly fundamental} if $\cX$ is fundamental and cohomologically affine; and
  \item \fndefn{nicely fundamental} if $\cX \cong [U/Q]$ for an affine scheme $U$ and a
    nice and embeddable group scheme $Q$, both over some common affine scheme $S$, i.e., $\cX$ admits an affine morphism to $B_SQ$.
 \end{enumerate}
\end{definition}

\begin{remark}[Relations between the notions]
For algebraic stacks, we have the obvious implications:
\[
\mbox{nicely fundamental} \implies \mbox{linearly fundamental} \implies \mbox{fundamental}
\]
If $\cX$ is fundamental (resp.\ linearly fundamental), then
$\cX$ admits an adequate (resp.\ good) moduli space: $\spec \Gamma(\cX,\Orb_{\cX})$.

In characteristic $0$, an algebraic stack is linearly fundamental if and only if it is
fundamental. We will show that in positive equicharacteristic, a linearly fundamental stack is nicely fundamental \'etale-locally over its good moduli space (\Cref{P:nice-neighborhood}).

The additional condition of a fundamental stack to be linearly fundamental is that $\cX \cong [\Spec B/\GL_N]$ is 
cohomologically affine, which means that the adequate
moduli space $\cX \to \Spec B^{\GL_n}$ is a good moduli space. We will show
that this happens
precisely when the stabilizer of every closed point is linearly reductive 
(\Cref{C:adequate+lin-red=>good:fundamental}).
\end{remark}

\begin{remark}[Equivalences I] \label{R:BG-fundamental}
If $G$ is a group scheme which is affine, flat and of finite presentation over an affine scheme $S$, then:
\begin{align*}
\mbox{$BG$ fundamental} &\iff \mbox{$G$ geometrically reductive and embeddable}\\
\mbox{$BG$ linearly fundamental} &\iff \mbox{$G$ linearly reductive and embeddable}\\
\mbox{$BG$ nicely fundamental} &\iff \mbox{$G$ nice and embeddable.}
\end{align*}
This follows from \Cref{R:resolution} and the definitions of geometrically reductive and linearly reductive, using that flat closed subgroups of diagonalizable groups are diagonalizable for the nice case.
If $\stX$ is an algebraic stack, then it also follows from the definitions that:
\begin{align*}
\mbox{$\stX$ fundamental} &\iff \mbox{$\stX=[\Spec A/G]$ for $G$ geom.\ red.\ and emb.}\\
\mbox{$\stX$ lin.\ fundamental} &\mathrlap{{}\impliedby}\phantom{{}\iff{}} \mbox{$\stX=[\Spec A/G]$ for $G$ lin.\ red.\ and emb.}\\
\mbox{$\stX$ nicely fundamental} &\iff \mbox{$\stX=[\Spec A/G]$ for $G$ nice and emb.}
\end{align*}
The second implication is not an equivalence, see \Cref{SS:noetherian-counterexample}. We will, however, show that under mild mixed characteristic hypotheses it is an equivalence \emph{\'etale-locally} over the good moduli space (\Cref{C:linearly-fundamental:lin-red-quot}).
\end{remark}

\begin{remark}[Equivalences II] An algebraic stack $\cX$ is a {\it global quotient stack} if $\cX \cong [U/\GL_n]$, where $U$ is an algebraic space.
Since adequately affine and representable morphisms are necessarily affine (\cite[Thm.~4.3.1]{alper-adequate}), we have the following equivalences for a quasi-compact and quasi-separated algebraic stack $\cX$:
\begin{align*}
\mbox{fundamental}& \iff \mbox{adequately affine and a global quotient}\\
\mbox{linearly fundamental}&\iff \mbox{cohomologically affine and a global quotient.}
\end{align*}
\end{remark}

\begin{proposition}\label{P:lr-gerbe-field}
  Let $\cG$ be a gerbe over a field $k$.
  \begin{enumerate}
  \item \label{PI:lr-gerbe-field:nice} If the stabilizer group of the unique
    point of $\cG$ is nice, then $\cG$ is nicely fundamental.
  \item \label{PI:lr-gerbe-field:char-p} If the characteristic of $k$
    is $p>0$ and $\cG$ has linearly reductive stabilizer group, then
    $\cG$ is nicely fundamental.
  \end{enumerate}
\end{proposition}
\begin{proof}
  Claim \eqref{PI:lr-gerbe-field:char-p} follows from \eqref{PI:lr-gerbe-field:nice} and \Cref{R:nice_grp_pos_char}. For claim \eqref{PI:lr-gerbe-field:nice}: since $\cG \to \spec k$ is smooth, there
  is a finite separable extension $k \subseteq k'$ that neutralizes
  the gerbe. Hence, $\cG_{k'} \cong BQ'$, for some nice
  group scheme $Q'$ over $k'$. Let $Q$ be the Weil restriction of $Q'$
  along $\spec k' \to \spec k$; then $Q$ is nice and there is an
  induced affine morphism $\cG \to BQ$.
\end{proof}

\subsection{Approximation of fundamental and nicely fundamental stacks}\label{SS:approximation-fund}
Here we establish that standard limit arguments, allowing to reduce arguments
to schemes of finite type, admit variants for fundamental and nicely fundamental stacks.
These results will be used to reduce from the situation of a complete local ring to an
excellent henselian local ring (via Artin approximation), from a henselian
local ring to an \'etale neighborhood, and from (non-)noetherian rings to
excellent rings.  The linearly fundamental case is more subtle---see
\Cref{A:mixed-char-counterexamples} for some counterexamples---but will eventually be
established in equal characteristic and in certain mixed characteristics
(\Cref{T:approximation-of-lin-fund:relative}).

We begin with the following standard limit result, in the style of \cite[IV.8]{EGA}, for the
property of an embeddable group scheme being geometrically reductive or nice. 

\begin{lemma}\label{L:approximation-reductivity}
  Let $\{S_\lambda\}_{\lambda \in \Lambda}$ be an inverse system of
  quasi-compact and quasi-separated algebraic spaces with affine
  transition maps and limit $S$. Let $\alpha \in \Lambda$ and let
  $G_{\alpha} \to S_{\alpha}$ be a flat group algebraic space of
  finite presentation. For
  $\lambda \geq \alpha$, let $G_{\lambda}$ be the pullback of
  $G_{\alpha}$ along $S_\lambda \to S_{\alpha}$ and let $G$ be
  the pullback of $G_{\alpha}$ along $S \to S_{\alpha}$.  If $G$
  is geometrically reductive (resp.\ nice) and embeddable over $S$, then $G_\lambda$ is
  geometrically reductive (resp.\ nice) and embeddable over $S_\lambda$ for all $\lambda
  \gg \alpha$.
\end{lemma}
\begin{proof} 
  Let $\sE$ be a vector bundle on $S$ and let $G\inj \GL(\sE)$ be a closed
  embedding. By standard limit methods, there exists a vector bundle
  $\sE_\lambda$ on $S_\lambda$ and a closed embedding $G_\lambda\inj
  \GL(\sE_\lambda)$ for all sufficiently large $\lambda$. If $G$ is
  geometrically reductive, then $\GL(\sE)/G\to S$ is affine and so is
  $\GL(\sE_\lambda)/G_\lambda\to S_\lambda$ for all sufficiently large $\lambda$; hence
  $G_\lambda$ is geometrically reductive by Matsushima's theorem (\Cref{R:resolution}).

  If $G^0\subseteq G$ is an open and closed normal subgroup as in the
  definition of a nice group scheme, then by standard limit methods, we can find an open and
  closed normal subgroup $G_\lambda^0\subseteq G_\lambda$ for all sufficiently
  large $\lambda$ satisfying the conditions in the definition of nice group
  schemes.
\end{proof}

\begin{lemma}\label{L:nicely-fund-fp-Q}
  An algebraic stack $\cX$ is nicely fundamental if and only if there exists an
  affine scheme $S$ of finite presentation over $\Spec \ZZ$, a nice and
  embeddable group scheme $Q\to S$ and an affine morphism $\cX\to B_SQ$.
\end{lemma}
\begin{proof}
The condition is sufficient by definition.  For necessity, by definition, we have an affine map $\cX \to B_S Q$, where $Q$ is a nice and embeddable group scheme over an affine scheme $S$.   
We can write $S = \varprojlim_\lambda S_{\lambda}$ as a limit of affine schemes of finite type over $\ZZ$,
and find a flat group scheme $Q_{\alpha} \to S_{\alpha}$ of finite type such that
$Q=Q_{\alpha}\times_{S_{\alpha}} S$. Let $Q_{\lambda}=Q_{\alpha}\times_{S_{\alpha}} S_\lambda$
so that $\cX\to B_S Q\to B_{S_\lambda} Q_\lambda$ is affine for all $\lambda\geq \alpha$.
Then \Cref{L:approximation-reductivity} implies that $Q_\lambda$ is nice for $\lambda \gg \alpha$.
\end{proof}

\begin{lemma}\label{L:excellent-approx-fundamental}
  Let $\cX$ be a fundamental (resp.\ a nicely fundamental) stack.
  Then there exists an inverse system of fundamental (resp.\ nicely fundamental)
  stacks $\cX_\lambda$ of finite type over $\Spec \ZZ$ with affine transition
  maps such that $\cX=\varprojlim_\lambda \cX_\lambda$.
\end{lemma}
\begin{proof}
  If $\cX$ is fundamental, then we have an affine morphism $\cX\to
  B\GL_{n,\ZZ}$. Since every quasi-coherent sheaf on the noetherian stack $B\GL_{n,\ZZ}$ is a union of its
  finitely generated subsheaves~\cite[Prop.~15.4]{lmb}, we can write $\cX=\varprojlim_\lambda \cX_\lambda$, where the
  $\cX_\lambda\to B\GL_{n,\ZZ}$ are affine and of finite type.  If $\cX$ is nicely
  fundamental, we argue analogously with $B_SQ$ of \Cref{L:nicely-fund-fp-Q}
  instead of $B\GL_{n,\ZZ}$.
\end{proof}
The following proposition shows that the property of being (nicely) fundamental
descends under limits and also describes how adequate moduli spaces commute
with inverse limits.
\begin{proposition}\label{P:approximation-fundamental}
  Let $\cX=\varprojlim_\lambda \cX_\lambda$ be an inverse limit of quasi-compact
  and quasi-separated algebraic stacks with affine transition maps.
  \begin{enumerate}
    \item\label{PI:approx:fundamental}
      If $\cX$ is fundamental (resp.\ nicely fundamental), then so is
      $\cX_\lambda$ for all sufficiently large $\lambda$.
    \item\label{PI:approx:nice-gerbes}
      Let $x\in |\cX|$ be a point with image $x_\lambda\in |\cX_\lambda|$.
      If $\cG_x$ (resp.\ $\overline{\{x\}}$) is nicely fundamental,
      then so is $\cG_{x_\lambda}$ (resp.\ $\overline{\{x_\lambda\}}$)
      for all sufficiently large $\lambda$.
    \item\label{PI:approx:adequate}
      If $\cX\to X$ and $\cX_\lambda\to X_\lambda$ are
      adequate moduli spaces, then $X=\varprojlim_\lambda X_\lambda$.
  \end{enumerate}
\end{proposition}
\begin{proof}
  For the first statement, let $\cY=B\GL_{n,\ZZ}$ (resp.\ $\cY=B_S Q$ for $Q$ as in
  \Cref{L:nicely-fund-fp-Q}). Then there is an affine morphism $\cX\to \cY$
  and hence an affine morphism $\cX_\lambda\to \cY$ for all sufficiently large
  $\lambda$~\cite[Prop.~B.1, Thm.~C]{rydh-2009}.
  The second statement follows from the first by noting that
  $\cG_x = \varprojlim_\lambda \cG_{x_\lambda}$ and
  $\overline{\cG_x} = \varprojlim_\lambda \overline{\cG_{x_\lambda}}$.

  The third statement follows directly from the following two facts (a)
  push-forward of quasi-coherent sheaves along $\pi_\lambda\colon
  \cX_\lambda\to X_\lambda$ preserves filtered colimits and (b) if $\cA$ is a
  quasi-coherent sheaf of algebras, then the adequate moduli space of
  $\Spec_{\cX_\lambda} \cA$ is $\Spec_{X_\lambda} (\pi_\lambda)_*\cA$.
\end{proof}

\begin{remark} \label{R:not-limit-preserving}
The analogous statements of \Cref{L:approximation-reductivity}
(resp.\ \Cref{P:approximation-fundamental}\itemref{PI:approx:fundamental}) for linearly reductive and
embeddable group schemes (resp.\ linearly fundamental stacks) are false in mixed
characteristic. Indeed, $\GL_{2,\QQ} = \varprojlim_m \GL_{2,\ZZ[\frac{1}{m}]}$
and $\GL_{2,\QQ}$ is linearly reductive but $\GL_{2,\ZZ[\frac{1}{m}]}$ is never
linearly reductive.  Likewise, $B\GL_{2,\QQ}$ is linearly fundamental but $B
\GL_{2, \ZZ[\frac{1}{m}]}$ is never linearly fundamental.

It is not true in general that the induced maps $\cX \to \cX_{\lambda} \times_{X_{\lambda}} X$ are isomorphisms for sufficiently large $\lambda$. Take, for example,
$\cX_n = [\Spec A_n / \Gm]$ where $A_n = k[x_1, \ldots, x_n]$ with the standard scaling action.
\end{remark}

\section{Luna's fundamental lemma and universality of adequate moduli spaces} \label{S:pairs}

We begin by defining local, henselian, and coherently complete
pairs, and stating a general version of Artin approximation
(\Cref{T:artin-approximation}).  We then prove that a pair of stacks is henselian if and only if their adequate moduli spaces are henselian (\Cref{T:henselian-pair-gms}), which provides a henselian analogue of our main result on coherent completeness (\Cref{T:complete}).  We apply this theorem to give quick proofs of the universality of
adequate moduli spaces (\Cref{T:universal}) and Luna's fundamental
lemma (\Cref{L:fundamental-lemma}).

\subsection{Henselian pairs}
Recall that we have defined the notion of coherently complete pairs in
\Cref{D:coherently-complete}. We now introduce the following weaker notions:
\begin{definition} \label{D:pairs} Fix a closed immersion of algebraic stacks $\cZ \subseteq \cX$. The pair $(\cX, \cZ)$ is said to be
\begin{enumerate}
	\item {\it local} if every non-empty closed subset of $|\cX|$ intersects $|\cZ|$ non-trivially; and
	\item {\it henselian} if for every finite morphism $\cX'\to \cX$, the restriction map
\begin{equation} \label{E:henselian-pair}
\ClOpen(\cX')\to \ClOpen(\cZ\times_{\cX} \cX'),
\end{equation}
is bijective, where $\ClOpen(\cX)$ denotes the set of closed and open substacks of $\cX$ \cite[IV.18.5.5]{EGA}.
\end{enumerate}
In addition, we call a pair $(\cX, \cZ)$ {\it affine}  if $\cX$ is affine and an affine pair $(\cX, \cZ)$ {\it (quasi-)excellent} if $\cX$ is (quasi-)excellent. The completion of an affine (quasi-)excellent pair is (quasi-)excellent~\cite{kurano-shimomoto}. Occasionally, we will also say $\cX$ is local, henselian, or coherently complete along $\cZ$ if the pair $(\cX,\cZ)$ has the corresponding property. 
\end{definition}

We list some examples of henselian and coherently complete pairs. 
\begin{example}[Adic rings]\label{EX:adic-ring}
  Let $A$ be a noetherian ring and let $I \subseteq A$ be an
  ideal. Then $(\spec A,\spec A/I)$ is a coherently complete pair if
  and only if $A$ is $I$-adically complete. The sufficiency is
  trivial. For the necessity, we note that $\varprojlim_n
  \COH(A/I^{n+1}) \simeq \COH(\hat{A})$, where $\hat{A}$ denotes the
  completion of $A$ with respect to the $I$-adic topology. Hence,
  the natural functor $\COH(A) \to \COH(\hat{A})$ is an
  equivalence of abelian tensor categories. It follows from 
  Tannaka duality (see \S \ref{SS:tannaka}) that the
  natural map $A \to \hat{A}$ is an isomorphism.
\end{example}

\begin{example}[Proper maps]
  Let $A$ be a ring and let $I \subseteq A$ be an ideal.  
  Let $f\colon \cX \to \spec A$
  be a proper morphism of algebraic stacks. Let $\cZ = f^{-1}(\spec A/I)$. 
  \begin{enumerate}
  \item If $A$ is $I$-adically complete, then $(\cX,\cZ)$ is coherently
    complete. This is just the usual
    Grothendieck existence theorem, see \cite[III.5.1.4]{EGA} for the
    case of schemes and \cite[Thm.~1.4]{MR2183251} for algebraic
    stacks.
  \item If $A$ is henselian along $I$, then $(\cX,\cZ)$ is
    henselian. This is part of the proper base change theorem in
    \'etale cohomology; the case where $I$ is maximal is well-known,
    see \cite[Rem.~B.6]{MR3148551} for further discussion.
\end{enumerate}
\end{example}

Let $A$ be a noetherian ring and let $I \subseteq J \subseteq A$ be
ideals. Assume that $A$ is $J$-adically complete. Recall that $A/I$ is
then $J$-adically complete and $A$ is also $I$-adically complete. This
is analogous to parts \itemref{L:coho-complete-stable:finite} and
\itemref{L:coho-complete-stable:bigger}, respectively, of the
following result.
We omit the proof.
\begin{lemma}\label{L:coho-complete-stable}
  Let $\cZ \subseteq \cX$ be a closed immersion of algebraic stacks.
  Assume that the pair $(\cX,\cZ)$ is henselian or coherently complete. 
  \begin{enumerate}
  \item \label{L:coho-complete-stable:finite} Let
    $f\colon \cX' \to \cX$ be a finite morphism and let
    $\cZ' \subseteq \cX'$ be the pullback of $\cZ$. Then
    $(\cX',\cZ')$ is henselian or coherently complete, respectively.
  \item \label{L:coho-complete-stable:bigger} Let
    $\stk{W} \subseteq \cX$ be a closed substack. If
    $|\cZ| \subseteq |\stk{W}|$, then $(\cX,\cW)$ is
    henselian or coherently complete, respectively.
  \end{enumerate}
\end{lemma}

\begin{remark}\label{R:pair-implications}
  For a pair $(\cX,\cZ)$, we have the following sequence of implications:
  \[
    \mbox{coherently complete} \implies \mbox{henselian} \implies \mbox{local}.
  \]
  The second implication is trivial: if $\cW\subseteq \cX$ is a closed
  substack, then $\ClOpen(\cW) \to \ClOpen(\cZ \cap \cW)$ is
  bijective. For the first implication, note that we have bijections:
  \[
    \ClOpen(\cX) \simeq \varprojlim_n \ClOpen(\thck{\cX}{\cZ}{n}) \simeq \ClOpen(\cZ)
  \]
  whenever $(\cX,\cZ)$ is coherently complete. The implication now follows from the elementary \Cref{L:coho-complete-stable}\itemref{L:coho-complete-stable:finite}.
\end{remark}

\begin{remark}
  It follows from the main result of \cite{rydh-2014} that 
  if $\cX$ is quasi-compact and
  quasi-separated, then $(\cX, \cZ)$ is a henselian pair if and only if \eqref{E:henselian-pair} is bijective for every
  integral morphism $\cX'\to \cX$.
\end{remark}

\begin{remark}[Nakayama's lemma for stacks] \label{R:nakayama}
As seen by descent from a smooth presentation, the following variants of Nakayama's lemma hold for local pairs $(\cX, \cZ)$:  (1) if $\cF$ is a quasi-coherent $\oh_{\cX}$-module of finite type and $\cF|_{\cZ} = 0$, then $\cF = 0$; and (2) if $\varphi \co \cF \to \cG$ is a morphism of quasi-coherent $\oh_{\cX}$-modules with $\cG$ of finite type and $\varphi|_{\cZ}$ is surjective, then $\varphi$ is surjective. 
\end{remark}

We will frequently use Artin approximation over henselian pairs to pass from
completions to henselizations, especially in
\Cref{S:approximation-deformation}. This version of Artin approximation is due
to Popescu~\cite[Thm.~1.3]{MR868439} and follows from his desingularization
theorem as we will also explain below. Artin's original approximation
theorem~\cite[Cor.~2.2]{MR0268188} is valid for henselian pairs $(S,s)$ where
$S$ is the spectrum of the henselization of a local ring essentially of finite
type over either a field or an excellent Dedekind domain.

\begin{theorem}[Artin approximation over henselian pairs]\label{T:artin-approximation}
Let $A$ be a G-ring, e.g., quasi-excellent, let
$(S,S_0)=(\Spec A,\Spec A/I)$ be an affine henselian pair and let
$\widehat{S}=\Spec \hat{A}$ be its $I$-adic completion. Let $F\colon
(\SCH{S})^{\opp} \to \SETS$ be a limit preserving functor.
Given an element
$\overline{\xi}\in F(\hat{S})$ and an integer $n\geq 0$, there exists an
element $\xi\in F(S)$ such that $\xi$ and $\overline{\xi}$ have equal images
in $F(S_n)$ where $S_n=\Spec A/I^{n+1}$.
\end{theorem}
\begin{proof}
The completion map $\widehat{S}\to S$ is regular. Hence, by N\'eron--Popescu
desingularization \cite[Thm.~1.8]{MR868439}, there exists a smooth morphism $S'\to S$ and a section
$\xi'\in F(S')$ such that $\xi'|_{\widehat{S}}=\overline{\xi}$. By
Elkik~\cite[Thm.,
  p.~568]{MR0345966}, there is an element $\xi\in F(S)$ as requested.
\end{proof}

\subsection{Adequate moduli spaces and henselian/complete pairs}
The following proposition gives a generalization of one direction for
\Cref{T:complete}: coherent completeness passes to adequate moduli spaces. The other direction, which we defer until later, is
much more involved.
\begin{proposition}\label{P:gms-cohocomp-necc}
  Let $\cZ \subseteq \cX$ be a closed immersion of noetherian
  algebraic stacks. Let $\pi \colon \cX \to X$ be an adequate moduli
  space, where $X$ is affine and noetherian.  If the pair $(\cX,\cZ)$
  is coherently complete, then the pair $(X,\pi(\cZ))$ is coherently
  complete.
\end{proposition}
\begin{proof}
  Let $I\subseteq A$ be the ideal defining $Z=\pi(\cZ)$, let $A\to \widehat{A}$
  be the $I$-adic completion and let $\widehat{X}=\Spec
  \widehat{A}$. The composition $\cX^{[n]}_{\cZ}\to \cX\to X$ factors
  through $X^{[n]}_Z$, hence lifts uniquely to $\widehat{X}$. By
  Tannaka duality, we obtain a unique lift $\cX\to \widehat{X}$. But,
  by definition of an adequate moduli space, $\Gamma(\cX,\Orb_\cX)=A$, so we obtain a retraction $\widehat{A} \to A$.  It follows that $A$ is $I$-adically complete.
\end{proof}

As we show below, for henselian pairs, the analog of \Cref{T:complete} (coherent completeness)
and \Cref{P:gms-cohocomp-necc} is straightforward. 

\begin{theorem}\label{T:henselian-pair-gms}
Let $\cX$ be a quasi-compact and quasi-separated algebraic stack with
adequate moduli space $\pi\colon \cX\to X$. Let $\cZ \subseteq \cX$
be a closed substack with $Z=\pi(\cZ)$. The pair $(\cX,\cZ)$ is 
henselian if and only if the pair $(X,Z)$ is henselian.
\end{theorem}

\begin{proof} The induced morphism $\cZ \to Z$ factors as the composition of an adequate moduli space
$\cZ \to \tilde{Z}$ and an adequate homeomorphism $\tilde{Z} \to Z$ \cite[Lem.\ 5.2.11]{alper-adequate}. 
If $\cX'\to \cX$ is integral, then $\cX'$ admits an adequate moduli space $X'$ and
$X' \to X$ is integral. Conversely, if $X' \to X$ is
integral, then $\cX\times_{X} X'\to X'$ factors as the composition of an adequate moduli space $\cX\times_{X} X'\to \tilde{X}'$ and an adequate homeomorphism $\tilde{X}' \to X'$ \cite[Prop.\ 5.2.9(3)]{alper-adequate}. 
It is thus enough to show that
\[
\ClOpen(\cX)\to \ClOpen(\stZ)
\]
is bijective if and only if
\[
\ClOpen(X)\to \ClOpen(Z)
\]
is bijective. But $\cX\to X$ and $\stZ\to Z$ are surjective
and closed with connected fibers \cite[Thm.\ 5.3.1]{alper-adequate}. Thus we have identifications
$\ClOpen(\cX)=\ClOpen(X)$ and $\ClOpen(\stZ)=\ClOpen(Z)$ that
are compatible with the restriction maps. The result follows.
\end{proof}

\subsection{Characterization of henselian pairs}

A quasi-compact and quasi-separated pair of schemes $(X,X_0)$ is henselian if and
only if for every \'etale morphism $g\colon X'\to X$, every section of
$g_0\colon X'\times_X X_0\to X_0$ extends to a section of $g$ (for
$g$ separated see \cite[IV.18.5.4]{EGA} and in general see \cite[Exp.~XII,Prop.~6.5]{MR0354654}). This is also true for stacks:

\begin{proposition}\label{P:henselian-sections-of-etale-repr}
Let $(\cX,\cX_0)$ be a pair of quasi-compact and quasi-separated algebraic
stacks. Then the following are equivalent
\begin{enumerate}
\item $(\cX,\cX_0)$ is henselian.
\item For every \emph{representable} \'etale morphism $g\colon \cX'\to \cX$,
  the induced map
  \[
  \Gamma(\cX'/\cX)\to \Gamma(\cX'\times_{\cX}\cX_0/\cX_0)
  \]
  is bijective.
\end{enumerate}
\end{proposition}
\begin{proof}
This is the equivalence between (1) and (3) of
\cite[Prop.\ 5.4]{mayer-vietoris}.
\end{proof}

We will later prove that (2) holds for
non-representable \'etale morphisms when $\cX$ is a stack with a good moduli
space and affine diagonal (\Cref{P:deformation-sections}).  A henselian pair does not always satisfy (2) for general
non-representable morphisms though, see \Cref{E:non-sep-counter-example}.

\subsection{Application: Universality of adequate moduli spaces}

For noetherian algebraic stacks, good moduli spaces were shown in \cite[Thm.~6.6]{alper-good} to be universal for maps to quasi-separated algebraic spaces and adequate moduli spaces were shown in \cite[Thm.~7.2.1]{alper-adequate}  to be universal for maps to algebraic spaces which are either locally separated or Zariski-locally have affine diagonal. We now establish this result unconditionally for adequate (and hence good) moduli spaces---see \Cref{T:universal-gms-variant} for a generalization to good moduli space morphisms.

\begin{theorem} \label{T:universal}
Let $\cX$ be an algebraic stack.  An adequate moduli space $\pi \co \cX \to X$ is universal for maps to algebraic spaces.
\end{theorem}

\begin{proof} We need to show that if $Y$ is an algebraic space, then the natural map
\begin{equation} \label{E:universal}
	\Map(X,Y) \to \Map(\cX, Y)
\end{equation}
is bijective.  To see the injectivity of \eqref{E:universal}, suppose that $h_1, h_2 \co X \to Y$ are maps such that $h_1 \circ \pi = h_2 \circ \pi$.  Let $E \to X$ be the equalizer of $h_1$ and $h_2$, that is, the pullback of the diagonal $Y \to Y \times Y$ along $(h_1, h_2) \co X \to Y \times Y$. The equalizer is a monomorphism and locally of finite type. By assumption $\pi\colon \cX \to X$ factors through $E$, and since $\cX \to X$ is universally closed, so is $E \to X$.  It follows that $E \to X$ is a closed immersion~\cite[Tag~\spref{04XV}]{stacks-project}.  Since $\cX \to X$ is schematically dominant, so is $E \to X$, hence $E = X$.
  
The surjectivity of \eqref{E:universal} is an \'etale-local property on $X$ since the injectivity of \eqref{E:universal} implies the gluing condition in \'etale descent. Thus, we may assume that $X$ is affine.  In particular, $\cX$ is quasi-compact and since any map $\cX \to Y$ factors through a quasi-compact open of $Y$, we may assume that $Y$ is also quasi-compact. 
 
Let $g \co \cX \to Y$ be a morphism and let $p \co Y' \to Y$ be an \'etale presentation where $Y'$ is an affine scheme.  To show that $g$ factors through $\pi \co \cX \to X$, we claim that after replacing $X$ with an \'etale cover, the base change $f \co \cX' \to \cX$ of $p$ along $g$ admits a section $s \co \cX \to \cX'$.  If this claim is established, then the map $g \co \cX \to Y$ factors as $\cX \xrightarrow{s} \cX' \xrightarrow{g'} Y' \xrightarrow{p} Y$.  Since $X$ and $Y'$ are affine, the equality $\Gamma(X, \oh_{X}) = \Gamma(\cX, \oh_{\cX})$ 
implies that the map $\cX \xrightarrow{s} \cX' \xrightarrow{g'} Y'$ factors through $\pi \co \cX \to X$.

To show the claim, observe that $f \co \cX' \to \cX$ is representable, \'etale, surjective and induces an isomorphism of stabilizer group schemes at all points.  
Let $x \in |X|$ be a point, $q \in |\cX|$ be the unique closed point over $x$ and $q' \in |\cX'|$ any point over $q$. Note that $\kappa(q)/\kappa(x)$ is a purely inseparable extension. After 
replacing $X$ with an \'etale neighborhood of $x$ (with a residue field extension), we may thus assume that $\kappa(q') = \kappa(q)$.  Since $f$ induces an isomorphism of stabilizer groups, the induced map $\cG_{q'} \to \cG_{q}$ on residual gerbes is an isomorphism.   \Cref{T:henselian-pair-gms} implies that $(\cX \times_X \Spec \oh_{X,x}^h, \cG_q)$ is a henselian pair so
\Cref{P:henselian-sections-of-etale-repr} gives a section of $\cX'\times_X \Spec \oh_{X,x}^h\to \cX\times_X \Spec \oh_{X,x}^h$.
Since $f$ is locally of finite presentation, we obtain a section $s \co \cX \to \cX'$ of $f \co \cX' \to \cX$ after replacing $X$ with an \'etale neighborhood of $x$.
\end{proof}
\subsection{Application: Luna's fundamental lemma}
\begin{definition} \label{D:strongly-etale}
If $\cX$ and $\cY$ are algebraic stacks admitting adequate moduli spaces $\cX \to X$ and $\cY \to Y$, we say that a morphism $f \co \cX \to \cY$ is {\it strongly \'etale} if the induced morphism $X \to Y$ is \'etale and $\cX \cong X \times_Y \cY$.
\end{definition}

The following result generalizes \cite[Thm.~6.10]{alper-quotient} from good
moduli spaces to adequate moduli spaces and also removes noetherian and
separatedness assumptions. 

\begin{theorem}[Luna's fundamental lemma] \label{L:fundamental-lemma}
Let $f \co \cX \to \cY$ be a morphism of algebraic stacks with adequate moduli
spaces $\pi_{\cX} \co \cX \to X$ and $\pi_{\cY} \co \cY \to Y$. Let $x \in
|\cX|$ be a point, closed in its fiber $\pi_{\cX}^{-1}(\pi_{\cX}(x))$, such that
\begin{enumerate}
\item\label{LI:fundamental-lemma:et+rep}
  $f$ is \'etale and representable in a neighborhood of $x$;
\item $y:=f(x) \in |\cY|$ is closed in its fiber $\pi_{\cY}^{-1}(\pi_{\cY}(y))$; and
\item $f$ induces an isomorphism of stabilizer groups at $x$.
\end{enumerate}
Then there exists an open neighborhood $\cU \subseteq \cX$ of $x$ such that
$\pi_{\cX}^{-1} (\pi_{\cX}(\cU)) = \cU$ and $f|_{\cU} \co \cU \to \cY$ is
strongly \'etale.  In particular, $X \to Y$ is \'etale at $\pi_{\cX}(x)$.
\end{theorem}

\begin{remark}
If $G$ a smooth algebraic group over an algebraically closed field $k$ such that $G^0$ is reductive and $\varphi \co U \to V$ is a $G$-equivariant morphism of irreducible normal affine varieties over $k$, then \cite[Thm.~4.1]{br-luna} (see also \cite[pg.~198]{git3}) established the result above for $f \co [U/G] \to [V/G]$.  
\end{remark}

\begin{proof}[Proof of~\Cref{L:fundamental-lemma}]
An open subset $\cU \subseteq \cX$ such that $\pi_{\cX}^{-1} (\pi_{\cX}(\cU)) =
\cU$ is called \emph{saturated} and it has adequate moduli space
$\pi_{\cX}(\cU)$.  Given any open neighborhood $\cU\subseteq \cX$ of $x$, the
smaller open neighborhood $|\cX|\smallsetminus \pi_{\cX}^{-1}(
\pi_{\cX}(|\cX|\smallsetminus \cU))$ is saturated.

By \itemref{LI:fundamental-lemma:et+rep},
we may replace $\cX$ with a saturated open neighborhood of $x$ such that $f$
becomes \'etale and representable. As adequate moduli spaces commute with flat base change, the question is \'etale-local
on $Y$.  We may therefore assume that $Y$ is affine in which case $\cY$ is quasi-compact and
quasi-separated.

If $Y$ is strictly henselian with closed point $\pi_{\cY}(y)$, then $(\cY,y)$ is a
henselian pair (\Cref{T:henselian-pair-gms}) and $\cG_x\to
\cG_y$ is an isomorphism.  Since $f$ is representable, we may apply \Cref{P:henselian-sections-of-etale-repr} to construct a section $s$ of $f$ such that $s(y)=x$. For general $Y$, since $f$ is
locally of finite presentation, we obtain a section $s$ of $f$ such that
$s(y)=x$ after replacing $Y$ with an \'etale neighborhood $(Y',y')\to (Y,\pi_{\cY}(y))$.
The image of $s$ is an open substack $\cU\subseteq \cX$ and $f|_\cU$ is an
isomorphism. After replacing $\cX$ with a saturated open neighborhood of $x$
contained in $\cU$, we can thus assume that $f$ is an open immersion.
After repeating the argument we obtain a section $s$ which is open and closed.
Then $\cU\subseteq \cX$ is automatically saturated and we are done.
\end{proof}

The result is not true in general if $f$ is not representable in a neighborhood of $x$ as the following example shows. However, if $\cY$ has separated diagonal, then $f$ is always representable in a neighborhood of $x$; see \Cref{P:refinement}\eqref{P:refinement:separated_diag}.

\begin{example}\label{E:non-sep-counter-example}
Let $S=\Spec k\llbracket t\rrbracket$ where $k$ is an algebraically closed
field. Let $G=(\ZZ/2\ZZ)_S$ and let $G'=G/H$ where $H\subset G$ is the
open subgroup that is the complement of the non-trivial element over the
origin. Let $\cX=BG$ and $\cY=BG'$ which both have good moduli space $S$
(adequate if $\kar k=2$) but $\cY$ does not have separated diagonal.
The induced morphism $f\colon \cX\to \cY$ is \'etale, but not representable,
and induces an isomorphism of the residual gerbes $B\ZZ/2\ZZ$ of the unique
closed points. But $f$ is not strongly \'etale and does not admit a section.
\end{example}

\section{Coherent completeness and effectivity}\label{S:completeness-and-effectivity}
In this section we prove \Cref{T:complete} (coherent completeness)  and \Cref{T:effectivity} (effectivity), which are both essential
ingredients in our proof of \Cref{T:base} (local structure).

The proofs of coherent completeness and effectivity are intertwined. First, we prove
an important special case of coherent completeness, namely when the resolution
property holds for $\cX$ (\Cref{T:complete:res-prop}). Then we use this special case to
prove effectivity in characteristic zero (\Cref{T:effectivity:char-0})
and in local characteristic $p$ (\Cref{C:effectivity-local}).  Then we prove
that adequate moduli spaces with linearly reductive stabilizers are good moduli
spaces (\Cref{T:adequate+lin-red=>good}). Then we can prove
effectivity in general and finally use it to prove coherent completeness
in general.

\subsection{Theorem on formal functions} \label{SS:formal-functions}
The following theorem on formal functions for good moduli spaces is a key component in the proof of the coherent completeness result (\Cref{T:complete}). This theorem is close in spirit to
\cite[III.4.1.5]{EGA} and is a generalization of \cite[Thm.~1.1]{alper-local}.  Surprisingly, we obtain, and will use, a version that also holds for adequate moduli spaces.

\begin{theorem}[Formal functions, adequate version]\label{T:almost-formal-fns:adequate}
  Let $\cX$ be an algebraic stack admitting an adequate moduli space $\pi \co \cX \to \Spec A$.  
  Let $\cZ \subseteq \cX$ be a closed substack defined by a  
  sheaf of ideals $\shv{I}$. Let
  $I=\Gamma(\cX,\shv{I})$ be the corresponding ideal of $A=\Gamma(\cX,\Orb_{\cX})$. If $A$ is noetherian and $I$-adically complete, and
  $\pi$ is of finite type, then for every
  $\shv{F} \in \COH(\cX)$ the natural map
  \begin{equation} \label{E:formal-functions:adequate}
  \Gamma(\cX,\shv{F}) \to \operatorname{\smash{\varprojlim_{n}}}
  \Gamma(\cX,\shv{F})/\Gamma(\cX,\shv{I}^{n}\shv{F})
 \end{equation}
 is an isomorphism.
\end{theorem}
\begin{proof}
  The $A$-module $F:=\Gamma(\cX,\cF)$ is finitely generated
  since $A$ is noetherian and $\pi$ is of finite type~\cite[Thm.~6.3.3]{alper-adequate}. Let
  $I_n=\Gamma(\cX,\shv{I}^n)$ and $F_n=\Gamma(\cX,\shv{I}^n\shv{F})$.
  Note that $\shv{I}^*:=\bigoplus \shv{I}^n$ is a finitely generated
  $\Orb_{\cX}$-algebra and $\shv{I}^*\shv{F}:=\bigoplus \shv{I}^n\shv{F}$
  is a finitely generated $\shv{I}^*$-module. %
  If we let $I_*=\bigoplus I_n=\Gamma(\cX,\shv{I}^*)$, then $\Spec_{\cX}
  \shv{\shv{I}^*}\to \Spec I_*$ is an adequate moduli
  space~\cite[Lem.~5.2.11]{alper-adequate}. Since $A$ is noetherian and
  $\Spec_{\cX}  \shv{\shv{I}^*}\to \Spec A$ is of finite type,
  it follows that $I_*$ is a finitely
  generated $A$-algebra and that $F_*:=\bigoplus F_n=\Gamma(\cX,\shv{I}^*\shv{F})$
  is a finitely generated $I_*$-module~\cite[Thm.~6.3.3]{alper-adequate}.

  By \cite[II.2.1.6(v)]{EGA}, there is an %
  integer $N\geq 1$ 
  such that
  $I_{kN}=(I_N)^k$ for all $k\geq 1$. That is, the topology induced by the
  non-adic system $I_n$ is equivalent to the $I_N$-adic topology.
  Without loss of generality, we can replace $\shv{I}$ with $\shv{I}^N$
  so that $I_*=I^*=\bigoplus_{k\geq 0} I^k$.

  Similarly, for
  sufficiently large $n$ (e.g., larger than all degrees of a set of homogeneous
  generators of $F_*$), $F_{n+1}=IF_n$ \cite[Lem.~10.8]{MR0242802}; 
  that is, $(F_n)$ is an $I$-stable filtration on $F$. It follows that $(F_n)$
  induces the same topology on $F$ as $(I^nF)$ \cite[Lem.~10.6]{MR0242802}. But
  $F$ is a finite $A$-module, hence $I$-adically complete, hence complete
  with respect to $(F_n)$.
\end{proof}

\begin{corollary}[Formal functions, good version]\label{C:formal-fns}
  Let $\cX$ be a noetherian algebraic stack admitting a good moduli space $\pi \co \cX \to \Spec A$.
  Let $\cZ \subseteq \cX$ be a closed substack defined by a  
  sheaf of ideals $\shv{I}$. Let
  $I=\Gamma(\cX,\shv{I})$ be the corresponding ideal of $A=\Gamma(\cX,\Orb_{\cX})$. If $A$ is $I$-adically complete, then for every
  $\shv{F} \in \COH(\cX)$ the natural map
  \begin{equation} \label{E:formal-functions}
  \Gamma(\cX,\shv{F}) \to \operatorname{\smash{\varprojlim_n}}
  \Gamma(\cX,\shv{F}/\shv{I}^{n}\shv{F})
 \end{equation}
 is an isomorphism.
\end{corollary}
\begin{proof}
  By \cite[Thm.~4.16(x)]{alper-good}, the ring $A$ is noetherian and
  by~\cite[Thm.~A.1]{luna-field}, $\cX\to\Spec A$ is of finite type so
  \Cref{T:almost-formal-fns:adequate} applies. For good moduli spaces,
  the natural map
  $\Gamma(\cX,\shv{F})/\Gamma(\cX,\shv{I}^{n}\shv{F})\to
  \Gamma(\cX,\shv{F}/\shv{I}^{n}\shv{F})$ is an isomorphism by definition.
\end{proof}

\begin{remark}
The formal functions theorem generalizes the isomorphism of
\cite[Eqn.\ (2.1)]{luna-field} from the case of $\cX=[\Spec B/G]$ for $G$
linearly reductive and $A=B^G$ complete local, all defined over a field $k$, to
$\cX=[\Spec B/\GL_n]$ and $A=B^{\GL_n}$ complete but not
necessarily local. This also includes $[\Spec B/G]$ for $G$ geometrically
reductive and embeddable since $\GL_n/G$ is affine. %
\end{remark}

\begin{remark}
In the setting of the adequate version, \Cref{T:almost-formal-fns:adequate},
suppose that $\cX = [\Spec(B)/G]$ where $G\to S$ is a reductive group scheme.
Then van der Kallen has shown that $H^i(\cX,-)$ preserves coherence for all $i$
\cite[Cor.~1.2]{MR3525842}. Using an argument similar to \cite[III.4.1.5]{EGA}
one can then show that the map \eqref{E:formal-functions} is an isomorphism,
see~\cite[Cor.~4.9]{alper-hall-lim_arXiv}. We will not use this result.
\end{remark}

\subsection{Coherent completeness I: with resolution property}\label{SS:coherently-complete:res-prop}
We can now prove the first version of \Cref{T:complete}.
\begin{theorem}[Coherent completeness assuming resolution property]\label{T:complete:res-prop}
  Let $\cX$ be a noetherian algebraic stack with affine diagonal
  and good moduli space $\pi \colon \cX \to X =\Spec A$. Let
  $\cZ \subseteq \cX$ be a closed substack defined by a
  coherent sheaf of ideals $\cI \subseteq \Orb_{\cX}$ and let
  $I=\Gamma(\cX,\cI)$.  Assume that
  $\cX$ has the resolution property. If $A$ is $I$-adically
  complete, then $\cX$ is coherently complete along $\cZ$.
\end{theorem}
Note that in this theorem, $\cX$ is assumed to have the resolution property, whereas in \Cref{T:complete} it is only assumed that $\cZ$ has the resolution property.
The following full faithfulness result does not require any resolution property hypothesis and follows from arguments similar to those of \cite[III.5.1.3]{EGA} and 
  \cite[Thm.~1.1(i)]{geraschenko-zb_fGAGA}.
\begin{lemma}\label{L:f_faithful}
 Let $\cX$ be a noetherian algebraic stack that is cohomologically affine.   
 Let $\cZ \subseteq \cX$ be a closed substack defined by a  
  sheaf of ideals $\shv{I}$. Let
  $I=\Gamma(\cX,\shv{I})$ be the corresponding ideal of $A=\Gamma(\cX,\Orb_{\cX})$. If  $A$ is $I$-adically complete, then the functor
  \[
      \COH(\cX) \to \operatorname{\smash{\varprojlim_n}} \COH(\thck{\cX}{\cZ}{n}).
    \]
    is fully faithful.
\end{lemma}
\begin{proof}
 Following \cite[\S1]{conradfmlgaga}, let 
  $\Orb_{\stk{\widehat{X}}}$ denote the sheaf of rings on the lisse-\'etale site of $\cX$ 
  that assigns to each smooth morphism $p\colon \spec B \to \cX$ the ring 
  $\varprojlim_n B/\shv{I}^nB$. The sheaf of rings 
  $\Orb_{\stk{\widehat{X}}}$ is coherent and the natural functor
  \[
  \COH(\stk{\widehat{X}}) \to \varprojlim_n \COH(\thck{\cX}{\cZ}{n})
  \]
  is an equivalence of categories \cite[Thm.~2.3]{conradfmlgaga}. Let $c\colon 
  \stk{\widehat{X}} \to \cX$ 
  denote the induced morphism of ringed topoi and let $\shv{F}, \shv{G} \in 
  \COH(\cX)$; then it remains to prove that the map
  \[
    \Hom_{\Orb_{\cX}}(\shv{F}, \shv{G}) \to    
    \Hom_{\Orb_{\stk{\widehat{X}}}}(c^*\shv{F}, c^*\shv{G})
  \]
  is bijective. Now we have the following commutative square, whose vertical arrows are 
  isomorphisms:
  \[
  \xymatrix{\Hom_{\Orb_{\cX}}(\shv{F}, \shv{G}) \ar[r] \ar[d] & \ar[d]
  \Hom_{\Orb_{\stk{\widehat{X}}}}(c^*\shv{F}, c^*\shv{G})\\
  \Gamma(\cX,\SHom_{\Orb_{\cX}}(\shv{F}, \shv{G})) \ar[r] & 
  \Gamma(\stk{\widehat{X}}, 
  \SHom_{\Orb_{\stk{\widehat{X}}}}(c^*\shv{F}, c^*\shv{G})).}
  \]
  Since $c$ is flat and $\shv{F}$ is coherent the natural morphism 
  \[
  c^*\SHom_{\Orb_{\cX}}(\shv{F},\shv{G}) \to 
  \SHom_{\Orb_{\stk{\widehat{X}}}}(c^*\shv{F},c^*\shv{G})
  \]
  is an isomorphism \cite[Lem.~3.2]{geraschenko-zb_fGAGA}. Thus, it 
  remains to prove that the map
  \[
  \Gamma(\cX,\shv{Q}) \to \Gamma(\stk{\widehat{X}},c^*\shv{Q})
  \]
  is an isomorphism whenever $\shv{Q} \in \COH(\cX)$. But there are natural 
  isomorphisms:
  \[
  \Gamma(\stk{\widehat{X}},c^*\shv{Q}) \cong \varprojlim_n 
  \Gamma(\stk{\widehat{X}},\shv{Q}/\shv{I}^{n+1}\shv{Q}) \cong \varprojlim_n 
  \Gamma(\thck{\cX}{\cZ}{n},\shv{Q}/\shv{I}^{n+1}\shv{Q}) \cong \varprojlim_n 
  \Gamma(\cX,\shv{Q}/\shv{I}^{n+1}\shv{Q}).
  \]
  The result now follows from  \Cref{C:formal-fns}. 
  \end{proof}
  \begin{proof}[Proof of  \Cref{T:complete:res-prop}]
    By  \Cref{L:f_faithful} it remains to show that if $\{\shv{F}_n\} \in  \varprojlim_n \COH(\thck{\cX}{\cZ}{n})$, then there exists a coherent sheaf $\shv{F}$ on $\cX$ with $(i^{[n]})^*\shv{F} \simeq \shv{F}_n$ for all $n$. 
  The following argument is similar to the proof of essential surjectivity of \cite[Thm.~1.3]{luna-field}.
  Since $\cX$ has the
  resolution property, there is a vector bundle $\shv{E}$ on $\cX$
  together with a surjection $\phi_0\colon \shv{E} \to \shv{F}_0$. We
  claim that $\phi_0$ lifts to a compatible system of morphisms
  $\phi_n \colon \shv{E} \to \shv{F}_n$ for every $n>0$.  Indeed, since
  $\shv{E}^{\vee}\otimes\shv{F}_{n+1}\to \shv{E}^{\vee}\otimes\shv{F}_n$
  is surjective and $\Gamma(\cX,-)$ is exact, it follows that
  the natural map
  $\Hom_{\Orb_{\cX}}(\shv{E}, \shv{F}_{n+1}) \to 
  \Hom_{\Orb_{\cX}}(\shv{E}, \shv{F}_n)$ is surjective.
  By Nakayama's
  lemma (see \Cref{R:nakayama}), 
  each $\phi_n$ is surjective.
  
  It follows that we obtain an induced morphism of systems
  $\{\phi_n\} \colon \{\shv{E}_n\} \to
  \{\shv{F}_n\}$, which is surjective. Applying this
  procedure to the kernel of $\{\phi_n\}$, there is another
  vector bundle $\shv{H}$ and a morphism of systems 
  $\{\psi_n\} \colon \{\shv{H}_n\} \to \{\shv{E}_n\}$
  such that $\coker \{\psi_n\} \cong \{\shv{F}_n\}$. By the full 
  faithfulness (\Cref{L:f_faithful}), the morphism $\{\psi_n\}$
  arises from a unique morphism $\psi \colon \shv{H} \to \shv{E}$.
  Let $\widetilde{\shv{F}} = \coker \psi$; then the universal
  property of cokernels proves that there is an isomorphism of systems
  $\{\widetilde{\shv{F}}_n\} \cong \{\shv{F}_n\}$ and
  the result follows.
\end{proof}
\Cref{T:complete:res-prop} in particular applies to complete affine quotient stacks.

\begin{example}
Let $S= \Spec B$ where $B$ is a noetherian ring. Let $G\subseteq \GL_{n,S}$ be a linearly reductive closed subgroup scheme acting on a noetherian affine scheme $X=\spec A$.   Then $[\Spec A/G]$ satisfies the resolution property; see \Cref{R:resolution}.
If $(A^G, \fm)$ is an $\fm$-adically complete local ring, then it follows from \Cref{T:complete:res-prop} that $[\spec A / G]$ is coherently complete along the unique closed point. When $S$ is the spectrum of a field and the unique closed $G$-orbit is a fixed point, this is \cite[Thm.~1.3]{luna-field}.
\end{example}

\subsection{Effectivity: general setup}\label{SS:effectivity-0}
We will now introduce a general setup for the proof of the effectivity theorem.
\begin{setup} \label{S:effectivityI-setup}
 Let $\{\cX_n\}_{n \ge 0}$ be an adic sequence of noetherian algebraic stacks, i.e., $\cX_0 \hookrightarrow \cX_1 \hookrightarrow \cdots$ is a sequence of closed immersions such that if $\cI_{(j)}$ denotes the coherent sheaf of ideals
 defining the closed immersion $u_{0j} \co \cX_0 \hookrightarrow \cX_j$,   then
 $(\cI_{(j)})^{i+1}$ defines $u_{ij} \colon \cX_i \hookrightarrow \cX_j$ for
 every $i\leq j$.
 Let $A_n = \Gamma(\cX_n,\Orb_{\cX_n})$, $X_n = \spec A_n$, $A=\varprojlim_n A_n$,
 $I_n=\ker(A \to A_{n-1})$, and $X=\spec A$.
\end{setup}

A classical result states that if each $\cX_i$ is affine, then $A=\varprojlim_n
\Gamma(\cX_n,\Orb_{\cX_n})$ is a noetherian ring and $\cX_i$ is the $i$th
infinitesimal neighborhood of $\cX_0$ in $\spec A$
\cite[0${}_{\mathrm{I}}$.7.2.8]{EGA}. The main effectivity theorem
(\Cref{T:effectivity}) is an analogous result when $\cX_0$ is linearly
fundamental.

A key observation here is that even if $\cX_0$ is linearly fundamental,
the sequence of closed immersions of affine schemes:
\begin{equation} \label{E:not-adic}
  X_0 \hookrightarrow X_1 \hookrightarrow \cdots .
\end{equation}
is not adic (this is just as in the proof of \Cref{T:almost-formal-fns:adequate}). 
The following lemma shows that the sequence \eqref{E:not-adic} is equivalent to an adic one, however. 
\begin{lemma}\label{L:adic_noetherian_gms}
Let $\{\cX_n\}_{n \ge 0}$ be an adic sequence as in \Cref{S:effectivityI-setup}.
If $\cX_0$ is cohomologically affine, then $A$ is a noetherian
  $I_1$-adically complete ring.
\end{lemma}
\begin{proof}
  For $n\geq 0$, let $\cF_n = (\cI_{(n)})^{n}$; then $\cF_n$ defines
  the closed immersion $\cX_{n-1} \hookrightarrow \cX_n$ and may be
  regarded as a coherent $\oh_{\cX_0}$-module.
  Let $\mathcal{A} = \bigoplus_{n=0}^\infty \cF_n$ as an
  $\oh_{\cX_0}$-module. The truncation $\mathcal{A}_{\leq N} = \bigoplus_{n=0}^N \cF_n$
  coincides with the graded $\Orb_{\cX_0}$-algebra $\Gr_{\cI_{(N)}} \oh_{\cX_N}
  := \bigoplus_{n=0}^N (\cI_{(N)})^n/(\cI_{(N)})^{n+1}$ which
  is finitely generated by elements of degree $1$ since $\cX_N$ is noetherian.
  The truncation map $\mathcal{A}_{\leq N}\to \mathcal{A}_{\leq M}$ is an
  algebra homomorphism for every $N\geq M$.  We can thus endow $\mathcal{A}$
  with the structure of a graded algebra, finitely generated by elements in
  degree $1$.
  
  By cohomological affineness, we have that
  $I_n/I_{n+1}=\ker(A_n\to A_{n-1})=\Gamma(\cX_0,\cF_n)$ and we obtain
  an isomorphism of graded algebras
  $\Gamma(\cX_0,\mathcal{A})=\Gr_{I_*} A:=\bigoplus I_n/I_{n+1}$.
  Now by \cite[Lem.~A.2]{luna-field}, the graded $A_0$-algebra $\Gr_{I_*} A$ is finitely
  generated (but not necessarily by elements of degree $1$). That
  is, for the filtration $\{I_n\}_{n\geq 0}$ on the ring $A$, the
  associated graded ring is a noetherian
  $A_0$-algebra. It follows from \cite[Thm.~4]{godement_top-m-adiques} that $A$
  is noetherian.

  Since $A$ is noetherian and complete with respect to the topology
  defined by $\{I_n\}_{n\geq 0}$ and $I_1^n \subseteq I_n$ for all
  $n$, it is also complete with respect to the $I_1$-adic
  topology. Indeed, the ideals $I_1^n$ are finitely generated, hence closed in
  the $(I_n)$-topology and we can conclude by~\cite[0\textsubscript{I}.7.2.4,
    Err\textsubscript{III}, 3]{EGA}. Alternatively, we can give a direct argument
  as follows. If $\hat{A}$
  denotes the $I_1$-adic completion of $A$, then there is a natural
  factorization
  \[
  A\to \hat{A}=\varprojlim_n A/(I_1)^n \to A=\varprojlim_n A/I_n
  \]
  of the
  identity. Since $\hat{A}\to A$ is surjective and $\hat{A}$ is noetherian and
  complete with respect to the $I_1$-adic topology, so is $A$.
\end{proof}

An easy case of effectivity is when the adic sequence $\{\cX_n\}_{n \ge 0}$ 
embeds into a linearly fundamental stack. 

\begin{lemma} \label{L:effectivity-if-embedded-into-fundamental}
  Let $\cH$ be a noetherian linearly fundamental algebraic stack.  
  If $\{\cX_n\}_{n \ge 0}$ is an adic sequence of noetherian algebraic
  stacks with compatible closed immersions $\cX_n \hookrightarrow \cH$,
  then the completion of $\{\cX_n\}_{n \ge 0}$ exists and is affine
  over $\cH$.
\end{lemma}

\begin{proof}
  Let $\pi \co \cH \to H$ be its good moduli space, $\hat{H}$ be the 
  completion of $H$ along $\pi(\cX_0)$, and $\hat{\cH}$ be the base change 
  $\cH \times_H \hat{H}$.  Let $\cH_0 = \cX_0$ be the induced closed substack 
  of $\hat{\cH}$ and $\cH_n$ be its $n$th infinitesimal
  neighborhood. We have a compatible sequence of closed immersions 
  $\cX_n \hookrightarrow \cH_n$.  If we let $\cK$ be the sheaf of ideals
  defining $\cH_0\to \hat{\cH}$, then
    $\cK\Orb_{\cX_n}=\cI_{(n)}$ and hence 
    $\cK^n\Orb_{\cX_n}=\cI^n_{(n)}$.  This means that the sequence 
    $\{\oh_{\cX_n} \}_{n \ge 0}$ of coherent $\oh_{\cH_n}$-modules is 
    adic, i.e., is an element of  $\varprojlim_n \COH(\cH_n)$.
    Since $\hat{\cH}\to \cH$ is affine, $\hat{\cH}$ is also linearly fundamental
    and hence has the resolution property.
    Thus, \Cref{T:complete:res-prop} 
    implies that $\hat{\cH}$ is coherently complete along $\cH_0$.  Hence, 
    there exists a closed
    immersion $\hat{\cX} \hookrightarrow \hat{\cH}$ that induces the $\cX_n$.
\end{proof}

We now prove the effectivity of an adic sequence $\{\cX_n\}$ under the 
hypothesis that $\cX_0$ is cohomologically affine and admits a representable 
map to a linearly fundamental stack $\cY$ that is smooth over $X$.  We will apply 
this result in the case that $\cY = B \GL_{n,X}$ or $\cY = B Q$ for a nice and
embeddable group scheme $Q \to X$.

\begin{proposition} \label{P:effectivity-over-smooth}
  Let $\{\cX_n\}_{n \ge 0}$ be an adic sequence as in \Cref{S:effectivityI-setup}.
  Let $\cY$ be a noetherian linearly 
  fundamental algebraic stack that is smooth over $X$.  If $\cX_0$ is
  cohomologically affine and there is a representable morphism
  $\cX_0\to \cY$, then the completion of $\{\cX_n\}_{n \ge 0}$ 
  exists and is affine over $\cY$.
\end{proposition}

\begin{proof}
  By \Cref{L:effectivity-if-embedded-into-fundamental}, it suffices to show 
  that there are compatible embeddings $\cX_n \hookrightarrow \cH$ into an 
  algebraic stack $\cH$ affine and of finite type over $\cY$.  This is a 
  consequence of the following lemma (which holds more generally if 
  $\cY$ is merely fundamental).
\end{proof}

\begin{lemma}[Embedding]\label{L:adic_embedding}
  Let $\{\cX_n\}_{n \ge 0}$ be an adic sequence as in \Cref{S:effectivityI-setup}.
  Let $\cY$ be a noetherian 
  fundamental algebraic stack that is smooth over $X$. If $\cX_0$ is
  cohomologically affine and there is a representable morphism
  $\cX_0\to \cY$ over $X$, then there exist
  \begin{enumerate}
  \item a smooth affine morphism $\cH \to \cY$ of finite type; and
  \item compatible closed immersions $\cX_n \hookrightarrow \cH$ over $X$.
  \end{enumerate}
\end{lemma}

\begin{remark}
In the proof of \cite[Thm.~1]{luna-field}, effectivity was established 
by embedding the residual gerbe $\cX_0$ and its thickenings into the 
normal space of $\cX_0$ in $\cX$.  The technique here is similar, but 
we instead first choose a deformation $\cX_1 \to \cY$, and then embed 
$\cX_1$ and its thickenings into the affine space $\cH$ over $\cY$ defined 
by a vector bundle resolution of the push-forward of $\oh_{\cX_1}$.
\end{remark}

\begin{proof}[Proof of \Cref{L:adic_embedding}]
  By \cite[Thm.~A.1]{luna-field}, $\cX_0 \to X_0$ is of finite
  type. Hence, $\cX_0 \to X$ is of finite type and cohomologically
  affine. But the diagonal of $\cY \to X$ is affine and of finite
  type, so $\phi_0 \colon \cX_0 \to \cY$ is cohomologically affine and
  of finite type. By assumption, it is representable, so Serre's
  theorem (e.g., \cite[Prop.~3.3]{alper-good}) tells us that
  $\phi_0 \colon \cX_0 \to \cY$ is also affine. 
  
  We claim that there is an $X$-morphism $\phi_1 \colon \cX_1 \to \cY$ that lifts $\phi_0$.
  The obstruction to such a lift belongs to the group
  $\Ext^1_{\Orb_{\cX_0}}(\LDERF \phi_0^*L_{\cY/X},\cI_{(1)})$
  \cite[Thm.~1.5]{olsson-deformation}. Since $\cY \to X$ is smooth,
  the cotangent complex $L_{\cY/X}$ is perfect of amplitude
  $[0,1]$. The assumption that $\cX_0$ is cohomologically affine now
  proves that this obstruction group vanishes. 
  As before, it follows that $\phi_1$ is affine and of finite type.

  Since $\cY$ is fundamental it has the resolution property, so
  there exists a vector
  bundle of finite rank $\cE$ on $\cY$ and a surjection of
  quasi-coherent $\Orb_{\cY}$-algebras
  $\Sym_{\Orb_{\cY}}(\mathcal{E}) \to
  (\phi_1)_*\Orb_{\cX_1}$. Defining
  $$\cH := \Spec_{\cY} \Sym_{\Orb_{\cY}}(\mathcal{E}),$$ 
  there is an induced closed immersion
  $i_1 \colon \cX_1 \hookrightarrow \cH$ and
  $\cH \to X$ is smooth. By using the same deformation theory 
  argument as above, we can produce compatible
  $X$-morphisms $i_n \colon \cX_n \to \cH$ lifting $i_1$. 
  The following Nakayama-like lemma implies that each $i_n$ 
  is a closed immersion.  
\end{proof}

The following lemma is a generalization of \cite[Prop.~A.8(1)]{luna-field} 
to the case where $|\cX_1|$ is not necessarily a single point.
\begin{lemma}\label{L:adic_thickening_embedding}
  Let $f \colon \cX \to \cY$ be a morphism of algebraic stacks. Let
  $\cI$ be a nilpotent quasi-coherent sheaf of ideals of
  $\Orb_{\cX}$. Let $\cX_1 \subseteq \cX$ be the closed immersion
  defined by $\cI^2$. If the composition
  $\cX_1 \to \cX \xrightarrow{f} \cY$ is a closed immersion, then $f$
  is a closed immersion.
\end{lemma}
\begin{proof}
  The statement is local on $\cY$ for the smooth topology, so we may
  assume that $\cY=\spec A$. Then $\cX_1$ is affine, and since $\cX$ is an
  infinitesimal thickening of $\cX_1$, it follows that $\cX$ is also
  affine \cite[Cor.~8.2]{rydh-2009}. Hence, we may assume that
  $\cX = \spec B$ and $\cI = \tilde{I}$ for some nilpotent ideal $I$
  of $B$. Let $\phi\colon A\to B$ be the induced morphism.
  The assumptions are that the composition $A \to B\to B/I^2$ is
  surjective and that $I^{n+1}=0$ for some $n\geq 0$. Let
  $K=\ker(A \to B/I)$.
  Since $KB\to I\to I/I^2$ is surjective and $I^{n+1}=0$,
  it follows that %
  $I=KB+I^2=KB+I^4=\dots=KB$. That is,
  $KB=I$.
  Further, since $A \to B\to B/KB=B/I$ is surjective and $K^{n+1}B=I^{n+1}=0$, it
  follows that %
  $B=\im \phi + KB = \im \phi + K^2B = \dots = \im \phi$. That is, $\phi$ is surjective.
\end{proof}

\subsection{Effectivity I: characteristic zero}\label{SS:effectivity-1}

\begin{theorem}[Effectivity in characteristic zero]\label{T:effectivity:char-0}
  Let $\{\cX_n\}_{n\geq 0}$ be an adic sequence of noetherian
  algebraic $\QQ$-stacks. If $\cX_0$ is linearly fundamental, then the
  completion of the sequence exists and is linearly
  fundamental.
\end{theorem}
\begin{proof}
  Since $\cX_0$ is linearly fundamental, it admits an affine morphism
  to $B\GL_{N,\QQ}$ for some $N>0$. This gives an affine morphism
  $\cX_0\to \cY:=B\GL_{N,X}$.  Since
  $X = \Spec \big(\varprojlim_n \Gamma(\cX_n, \oh_{\cX_n}) \big)$ is
  a $\QQ$-scheme, 
  $\cY$ is linearly fundamental.  The conclusion now follows from \Cref{P:effectivity-over-smooth}. 
\end{proof}
To prove effectivity in positive and mixed characteristic
(\Cref{T:effectivity}), we will need to make a better choice of group than
$\GL_{N,\QQ}$. To do this, we will next study the deformations of nice group
schemes.

\subsection{Deformation of nice group schemes}\label{SS:nice}
We will now prove that a nice and embeddable group scheme (see \Cref{D:reductive/nice}) can be deformed along an affine henselian pair (\Cref{D:pairs}).
After we have established the general effectivity result, we will prove
the corresponding result for linearly reductive group schemes (\Cref{P:deformation-linearly-reductive}).

\begin{proposition}[Deformation of nice group schemes]\label{P:nice-deformations} 
  Let $(S,S_0)$ be an affine henselian pair. If $G_0\to S_0$ is a nice and embeddable
  group scheme, then there exists a nice and embeddable group scheme $G\to S$ whose restriction
  to $S_0$ is
  isomorphic to $G_0$.
\end{proposition}
\begin{proof}
  Let $(S,S_0)=(\Spec A,\Spec A/I)$.
  By limit methods (\Cref{L:approximation-reductivity}), we may
  assume that $S$ is the henselization
  of an affine scheme of finite type over $\spec \ZZ$. Let 
  $S_n=\spec A/I^{n+1}$. Also, let $R$ be the $I$-adic 
  completion of $A$ and let $\hat{S} = \spec R$.

  Let $F \colon
  (\SCH{S})^{\opp} \to \SETS$ be the functor that assigns to each
  $S$-scheme $T$ the set of isomorphism classes of nice and embeddable
  group schemes over $T$. By 
  \Cref{L:approximation-reductivity}, $F$ is limit preserving.
  Suppose that we have a nice embeddable group scheme
  $G_{\hat{S}}\in F(\hat{S})$ restricting to $G_0$. By Artin approximation
  (\Cref{T:artin-approximation}), there exists $G_S
  \in F(S)$ that restricts to $G_0$.
  We can thus replace $S$ by $\hat{S}$ and assume that $A$ is complete.

  Fix a closed 
  immersion of $S_0$-group schemes
  $i\colon G_0 \to \GL_{n,S_0}$. By definition, there is an open and closed
  subgroup $(G_0)^0\subset G_0$ of multiplicative type. By \cite[Exp.~XI,
  Thm.~5.8]{MR0274459}, there is a lift of $i$ to a closed
  immersion of group schemes $i_S \colon G^0_S \to \GL_{n,S}$, where
  $G^0_S$ is of multiplicative type. Let
  $N=\mathrm{Norm}_{\GL_{n,S}}(G^0_S)$ be the normalizer, which is a
  smooth $S$-group scheme and closed $S$-subgroup scheme of
  $\GL_{n,S}$ \cite[Exp.~XI, 5.3~bis]{MR0274459}.

  Since $(G_0)^0$ is a
  normal $S_0$-subgroup scheme of $G_0$, it follows that $G_0$ is a closed
  $S_0$-subgroup scheme of $N\times_S S_0$. In particular, there is an
  induced closed immersion $q_{S_0}\colon (G_0)/(G_0)^0 \to (N/G^0_S)\times_S 
  S_0$ of group schemes over $S_0$. Since $G_0$ is nice, the locally constant group scheme
  $(G_0)/(G_0)^0$ has order prime to $p$. Since $R$ is complete, there is
  a unique locally constant group scheme $H$ over $S$ such that
  $H\times_S S_0 = (G_0)/(G_0)^0$. Note that $H$ is finite and linearly reductive 
  over $S$. 
  
  Since $N/G^0_S$ is a smooth and affine group scheme over $S$, there are 
  compatible closed immersions of $S_n$-group
  schemes $q_{S_n} \colon H\times_S S_n \to (N/G^0_S) \times_S S_n$ lifting $q_{S_0}$, 
  which are 
  unique up to conjugation \cite[Exp.~III,  Cor.~2.8]{sga3i}. Since $H$ is finite, 
  these morphisms effectivize to a morphism of group schemes $q_S \colon 
  H \to N/G^0_S$.
  We now
  define $G_S$ to be the preimage of $H$ under the quotient map
  $N\to N/G^0_S$. Then $G_S$ is nice and embeddable, and $G_S \times_S S_0 \cong G_0$.
\end{proof}

\subsection{Effectivity II: local case in positive characteristic}\label{SS:effectivity-2}
We can now establish the effectivity theorem for nicely fundamental stacks (\Cref{D:fundamental}).

\begin{theorem}[Effectivity for nice stacks]\label{T:effectivity-nice}
  Let $\{\cX_n\}_{n\geq 0}$ be an adic sequence of noetherian algebraic
  stacks. If $\cX_0$ is nicely fundamental, then the completion
  of the sequence exists and is nicely fundamental. 
\end{theorem}
\begin{proof} Let $X_0$ be the good moduli space of $\cX_0$.
  Since $\cX_0$ is nicely fundamental, it admits an affine morphism to
  $B_{X_0}Q_0$, for some nice and embeddable group scheme $Q_0 \to X_0$.
  By \Cref{L:adic_noetherian_gms}, the affine scheme
  $X = \Spec \big(\varprojlim_n \Gamma(\cX_n, \oh_{\cX_n}) \big)$
  is noetherian and complete along
  $X_0$. It follows from \Cref{P:nice-deformations} that there
  is a nice and embeddable group scheme $Q \to X$ lifting
  $Q_0 \to X_0$. Let $\cY = B_XQ$; then $\cY$ is linearly fundamental
  and smooth over $X$. The result now follows from
  \Cref{P:effectivity-over-smooth}.
\end{proof}
The following corollary will shortly be subsumed by 
\Cref{T:effectivity}, but is sufficient for many applications,
e.g., it is sufficient for \Cref{T:base}.
\begin{corollary}[Effectivity for local stacks]\label{C:effectivity-local}
  Let $\{\cX_n\}_{n\geq 0}$ be an adic sequence of noetherian
  algebraic stacks. Assume that $\cX_0$ is a gerbe over a field $k$.
  If $\cX_0$ is linearly fundamental (i.e., has linearly reductive stabilizer),
  then the completion of the sequence
  exists and is linearly fundamental.
\end{corollary}
\begin{proof}
  If $\cX$ is a $\QQ$-stack, then we
  are already done by \Cref{T:effectivity:char-0}. If not,
  then $k$ has characteristic
  $p>0$ and $\cX_0$ is nicely fundamental by
  \Cref{P:lr-gerbe-field}\eqref{PI:lr-gerbe-field:char-p}. \Cref{T:effectivity-nice}
  completes the proof.
\end{proof}

\subsection{Adequate moduli spaces with linearly reductive stabilizers are good} \label{SS:adequate}
We prove that adequate moduli spaces of stacks with linearly
reductive stabilizers at closed points are good
(\Cref{T:adequate+lin-red=>good}) by using the adequate version of the
formal function theorem (\Cref{T:almost-formal-fns:adequate}) and the
effectivity theorem in the form of \Cref{C:effectivity-local}.

\begin{lemma}\label{L:adequate-good-finite}
Let $\cX$ be an algebraic stack and let $\cZ\inj \cX$ be a closed
substack defined by the sheaf of ideals $\shv{I}$. Assume that $\cX$ has an adequate
moduli space $\pi\colon \cX\to \Spec A$ of finite type, where $A$ is noetherian and $I$-adically complete along
$I=\Gamma(\cX, \shv{I})$. Let $B_n=\Gamma(\cX, \Orb_{\cX}/\shv{I}^{n+1})$ for
$n \ge 0$ and $B=\varprojlim_n B_n$. If $\cZ$ is cohomologically
affine with affine diagonal, then the induced homomorphism $A\to B$ is finite.
\end{lemma}
\begin{proof}
Since $\cZ=\thck{\cX}{}{0}$ is cohomologically affine with affine diagonal,
so are its infinitesimal neighborhoods $\thck{\cX}{}{n}$. The surjection 
$\oh_{\cX}/\cI^{n+1} \to \oh_{\cX}/\cI^{n}$ therefore induces a surjection 
$B_n\to B_{n-1}$ of rings with kernel $\Gamma(\cX,\shv{I}^n/\shv{I}^{n+1})$
for all $n$. Thus, if we let $J_{n+1} = \ker(B\to B_n)$; then $J_n/J_{n+1}=
\Gamma(\cX, \shv{I}^n/\shv{I}^{n+1})$ and $B$ is complete with respect to the topology given by 
the filtration $(J_n)$.

Let $I_n=\Gamma(\cX,\shv{I}^n)$; then \Cref{T:almost-formal-fns:adequate} 
implies that $A$ is also complete with respect to the filtration given by $(I_n)$,
that is, $A=\varprojlim_n A/I_n$.  Taking global sections of the 
short exact sequence $0 \to \cI^{n+1} \to \oh_{\cX} \to \oh_{\cX}/\cI^{n+1} \to 0$ 
induces an injection $A/I_{n+1}\to B_n$, which is an adequate ring homomorphism, 
i.e., every element of $B_n$ has a positive power contained in the image. 
Passing to inverse limits, we see that the homomorphism $A\to B$ is an
injective continuous map between complete topological rings.

Taking global sections of the exact sequence $0 \to \shv{I}^{n+1} \to \shv{I}^n \to \shv{I}^n/\shv{I}^{n+1} \to 0$ induces an injective map
$I_n/I_{n+1} \to J_n/J_{n+1}$. Taking direct sums gives a surjection of
algebras $\bigoplus \shv{I}^n \to \Gr_{\shv{I}}(\Orb_{\cX})$ with kernel $\bigoplus \shv{I}^{n+1}$. Since $\pi$ is an adequate moduli space, taking global sections provides an
injective adequate map $\Gr_{I_*} A=\bigoplus I_n/I_{n+1}\to \Gr_{J_*} B=\bigoplus J_n/J_{n+1}$.

Also, $\Gr_{\shv{I}}(\Orb_{\cX})$ is a finitely generated
algebra and since $\Spec\bigl(\Gr_{J_*} B\bigr)$ is the adequate moduli space of
$\Spec_{\cX}\bigl(\Gr_{\shv{I}}(\Orb_{\cX})\bigr)$, it follows that
$\Gr_{J_*} B$ is a finitely generated
$A$-algebra~\cite[Thm.~6.3.3]{alper-adequate}. Thus $\Gr_{I_*} A\to \Gr_{J_*}
B$ is an injective adequate map of finite type, hence finite. It follows that
$A\to B$ is finite~\cite[Lem.\ on p.~6]{godement_top-m-adiques}.
\end{proof}

\begin{remark}
It is, a priori, not clear that $A\to B$ is adequate. Consider the following
example: $A=\FF_2\llbracket x \rrbracket$, $B=A[y]/(y^2-x^2y-x)$. Then $\Spec B\to \Spec A$ is a
ramified, generically \'etale, finite flat cover of degree $2$, so not
adequate. But the induced map on graded rings $\FF_2[x]\to \FF_2[x,y]/(y^2-x)$
is adequate. Nevertheless, it follows from
\Cref{T:adequate+lin-red=>good}, proven below, that $A=B$ in \Cref{L:adequate-good-finite}. If the formal
functions theorem (\Cref{C:formal-fns}) holds for stacks with adequate
moduli spaces, then $A=B$ without assuming that $\cZ$ is cohomologically
affine.
\end{remark}

\begin{theorem}\label{T:adequate+lin-red=>good}
Let $S$ be a noetherian algebraic space. Let $\cX$ be an algebraic stack
of finite type over $S$ with an adequate moduli space $\pi\colon \cX\to
X$. Assume that $\pi$ has affine diagonal. Then $\pi$ is a good moduli space if
and only if every closed point of $\cX$ has linearly reductive stabilizer.
\end{theorem}

\begin{remark}
See \Cref{C:adequate+lin-red=>good:fundamental,C:adequate+lin-red=>good2}
for non-noetherian versions.
\end{remark}

\begin{proof}
By \cite[Thm.~6.3.3]{alper-adequate}, 
$X$ is of finite type over $S$.
We can thus replace $S$ with $X$. If $\pi$
is a good moduli space, then every closed point has linearly reductive
stabilizer~\cite[Prop.~12.14]{alper-good}. For the converse, we need to prove that
$\pi_*$ is exact. This can be verified after replacing $S$ with the completion
at every closed point (since adequate moduli spaces commute with flat base change).
We may thus assume that $X=S$ is a complete local scheme. Since $\pi$ is an adequate moduli space, $\cX$ has a unique closed point and we let $\cZ \hookrightarrow \cX$ be the corresponding closed immersion.

By \Cref{C:effectivity-local}, the adic sequence
$\thck{\cX}{\cZ}{0}\hookrightarrow\thck{\cX}{\cZ}{1}\hookrightarrow\dots$ has completion
$\hat{\cX}$ that has a good moduli space $X'$. By Tannaka
duality (see \S \ref{SS:tannaka}), there is a natural map $f\colon
\hat{\cX}\to \cX$ and it suffices to prove it is an isomorphism. Now $f$ induces a map $g\colon X'\to X$ of adequate
moduli spaces. In the notation of \Cref{L:adequate-good-finite}, $X'=\Spec
B$ and $X=\Spec A$, and we conclude that $X'\to X$ is finite.
In particular, $f\colon \hat{\cX}\to \cX$ is also
of finite type since the good moduli map $\hat{\cX}\to X'$ is of finite
type~\cite[Thm.~A.1]{luna-field}. The morphism $f\colon \hat{\cX}\to \cX$ is 
formally \'etale, hence \'etale, and also 
affine~\cite[Prop.~3.2]{luna-field}, hence representable.  
Moreover, $f\colon \hat{\cX}\to \cX$
induces an isomorphism of stabilizer groups at the unique closed 
points so we may apply
Luna's fundamental lemma (\Cref{L:fundamental-lemma}) to conclude 
that $X' \times_{X} \cX = \hat{\cX}$ and thus $f \co \hat{\cX} \to \cX$ is 
finite. But $f$ is an isomorphism over the
unique closed point of $\cX$, hence $f$ is a closed immersion. But $f$
is also \'etale, hence a closed and open immersion, hence an isomorphism.
\end{proof}

\begin{corollary} \label{C:geom-red=>lin-red}
Let $S$ be a noetherian algebraic space and let $G\to S$ be an 
affine flat group scheme of finite presentation. Then $G\to S$ is linearly reductive if and
only if $G\to S$ is geometrically reductive and every closed fiber is linearly
reductive.  \epf
\end{corollary}
\begin{proof}
  Apply \Cref{T:adequate+lin-red=>good} to $BG \to S$, which is an
  adequate (resp.~good) moduli space if and only if $G \to S$ is geometrically (resp.~linearly)
  reductive.
\end{proof}
The corollary also holds in the non-noetherian case by
\Cref{C:adequate+lin-red=>good2}.

\subsection{Effectivity III: the general case}\label{SS:effectivity-3}
We now finally come to the proof of the general effectivity result for
adic systems of algebraic stacks. Recall that this says that the completion
$\hat{\cX}$ of an adic sequence $\{\cX_n\}_{n\geq 0}$ of noetherian algebraic stacks,
such that $\cX_0$ is linearly fundamental, exists and is linearly fundamental.

\begin{proof}[Proof of \Cref{T:effectivity}]
Let $X$ be as in \Cref{S:effectivityI-setup}.
  Since $\cX_0$ is (linearly) fundamental, it admits an affine morphism
  to $\cY = B\GL_{N,X}$. By \Cref{L:adic_embedding}, there
  is an affine morphism $\cH \to \cY$ of finite type and 
  compatible closed immersions
  $\cX_n \hookrightarrow \cH$.
  Let $H=\Spec \Gamma(\cH,\Orb_{\cH})$ be the
  adequate moduli space of $\cH$. Since the composition $\cH \to \cY \to X$ is
  of finite type and $X$ is noetherian (\Cref{L:adic_noetherian_gms}), 
  $H \to X$ is of finite type
  \cite[Thm.~6.3.3]{alper-adequate} and so $\cH \to H$
  is of finite type and $H$ is noetherian. 
  
  Since $\cX_n \to X_n$ is a good moduli space, there are uniquely induced
  morphisms $X_n \to H$. Passing to limits, we produce a
  unique morphism $X \to H$, which is a closed immersion as the composition
  $X \to H \to X$ is the identity. Take $\cH'$ to be the base
  change of $\cH \to H$ along $X \to H$. Let
   $H'=\spec \Gamma(\cH',\Orb_{\cH'})$; then the induced morphism $H' \to X$ is a 
   finite type adequate universal homeomorphism, hence finite. Since $\cH' \to X$ 
   is universally closed, the closed points of $\cH'$ are identified with the 
   closed points of $\cX_0$ and thus have linearly reductive stabilizer.  By
   \Cref{T:adequate+lin-red=>good}, $\cH'$ is cohomologically
   affine.  We may now apply \Cref{L:effectivity-if-embedded-into-fundamental} to 
   the induced closed immersions $\cX_n \hookrightarrow \cH'$ 
   to conclude that the completion of $\{\cX_n\}_{n \ge 0}$ exists.
\end{proof}
\begin{remark}[Quasi-excellence]\label{R:qe-effective}
  In \Cref{T:effectivity}, if $\Gamma(\cX_0,\oh_{\cX_0})$ is
  quasi-excellent, then $\hat{\cX}$ is locally
  quasi-excellent. Indeed, using the notation of
  \Cref{S:effectivityI-setup}, we know that
  $A=\Gamma(\hat{\cX},\oh_{\hat{\cX}})$ is an
  $I_1$-adically complete noetherian ring
  (\Cref{L:adic_noetherian_gms}). Since $\cX_0$ is cohomologically
  affine, $A/I_1 = \Gamma(\cX_0,\oh_{\cX_0})$, which is
  quasi-excellent by assumption. Hence, $A$ is quasi-excellent by the
  Gabber--Kurano--Shimomoto theorem \cite[Main Thm.\
  1]{kurano-shimomoto}. But $\hat{\cX} \to \spec A$ is of finite
  type, so $\hat{\cX}$ is locally quasi-excellent.
\end{remark}
\subsection{Coherent completeness II: general case}\label{SS:coherently-complete:general}
We can now finish the general coherent completeness theorem. Recall that this says that
a noetherian algebraic stack $\cX$ with affine good moduli space $X$ is coherently
complete along $\cZ$ if and only if $X$ is coherently complete along the image of
$\cZ$. This is under the assumption that $\cZ$ has the resolution property and it also
follows that $\cX$ has the resolution property.

\begin{proof}[Proof of \Cref{T:complete}]
  The necessity of the condition follows from
  \Cref{P:gms-cohocomp-necc}. For the sufficiency:  
  the completion $\hat{\cX}$ of
  $\{\cX_{\cZ}^{[n]}\}$ exists and is linearly fundamental (\Cref{T:effectivity}).
  By formal functions (\Cref{C:formal-fns}), 
  \[
    A=\Gamma(\cX,\oh_{\cX}) \simeq \varprojlim_n \Gamma(\cX,\oh_{\thck{\cX}{\cZ}{n}}) = \varprojlim_n \Gamma(\hat{\cX},\oh_{\thck{\cX}{\cZ}{n}}) \simeq \Gamma(\hat{\cX},\oh_{\hat{\cX}}).
  \]
  Hence, the good moduli space of $\hat{\cX}$ is $X$. By Tannaka
  duality, there is an induced morphism
  $f\colon \hat{\cX}\to \cX$ and it is
  affine~\cite[Prop.~3.2]{luna-field}.  The composition
  $\hat{\cX}\to \cX\to X$ is a good moduli space and hence of finite
  type~\cite[Thm.~A.1]{luna-field}. It follows that $f$ is of finite type. 
  Since $f$ is formally
  \'etale, it is thus \'etale. Luna's fundamental lemma (\Cref{L:fundamental-lemma}) 
  now implies that $f \co \hat{\cX} \to \cX$ is an isomorphism. 
  In particular, $\cX$ is linearly fundamental, i.e., has the
  resolution property.
\end{proof}

We are now in position to prove Formal GAGA (\Cref{C:formal-gaga}).

\begin{proof}[Proof of \Cref{C:formal-gaga}]  The first case follows from the second since if $I \subset A$ is a maximal ideal, 
$\cX \times_{\Spec A} \Spec(A/I)$ necessarily has the resolution property \cite[Cor.~4.14]{luna-field}.   
The corollary then follows from applying \Cref{T:complete} with $\cZ = \cX \times_{\Spec A} \Spec(A/I)$.
\end{proof}

\section{The local structure of algebraic stacks}\label{S:local-structure}
In this section, we first prove \Cref{T:microlocalization}, which
establishes the existence of formally syntomic neighborhoods of locally
closed substacks.  We then use this theorem
to prove \Cref{C:existence-completions} establishing the existence of completions at points with linearly reductive stabilizers. 
Finally, we prove the local structure of algebraic stacks (\Cref{T:base}) in a
slightly more general form, see \Cref{T:base-general}.

The results of this section establish the local structure of algebraic stacks near not necessarily closed points or immersions.  It turns out to be convenient to work in the more general setting  of \emph{pro-unramified} morphisms%
.
Recall that if $\cX$ is a noetherian algebraic stack, then a morphism
$\cV \to \cX$ is \emph{pro-unramified} (resp.\ a \emph{pro-immersion}) if it can be
written as a composition
$\cV \hookrightarrow \cV' \to \cX$, where $\cV \hookrightarrow \cV'$
is a flat quasi-compact monomorphism and $\cV' \hookrightarrow \cX$ is unramified and of finite type (resp.\ a closed immersion). Clearly, pro-immersions are pro-unramified. Note that residual gerbes on quasi-separated algebraic
stacks are pro-immersions \cite[Thm.~B.2]{MR2774654}. Moreover, every monomorphism of finite type is pro-unramified.

\subsection{Existence of formally syntomic neighborhoods}\label{SS:formal-neighborhoods}
Recall that a morphism is \emph{syntomic} if it is flat and locally of
finite presentation, with fibers that are local complete intersections
(e.g., smooth). In particular, if a morphism is syntomic and
representable, then its cotangent complex is perfect of tor-amplitude
$[-1,0]$ \cite[Tag \spref{0FK3}]{stacks-project}, which is the only
property of syntomic morphisms that we will use. As promised, we now
establish the following generalization of \Cref{T:microlocalization}.

\begin{theorem}[Formal neighborhoods]\label{T:microlocalization:general}
  Let $\cX$ be a noetherian algebraic stack. Let
  $\cX_0 \to \cX$ be pro-unramified. Let
  $h_0 \colon \cW_0 \to \cX_0$ be a syntomic (e.g., smooth) morphism. Assume that $\cW_0$ is linearly fundamental. If either
  \begin{enumerate}
  \item \label{TI:microlocalization:quaff} $\cX$ has quasi-affine diagonal; or
  \item \label{TI:microlocalization:exc} $\cX$ has affine stabilizers
    and $\Gamma(\cW_0,\Orb_{\cW_0})$ is quasi-excellent;
  \end{enumerate}
  then there is a flat morphism $h \colon \hat{\cW} \to \cX$, where
  $\hat{\cW}$ is noetherian, linearly fundamental,
  $h|_{\cX_0} \simeq h_0$, and $\hat{\cW}$ is coherently complete
  along $\cW_0=h^{-1}(\cX_0)$. Moreover if $h_0$ is smooth (resp.\ \'etale),
  then $h$ is unique up to non-unique $1$-isomorphism (resp.\ unique up to
  unique $2$-isomorphism).
\end{theorem}
\begin{proof}
  We first reduce to the situation where $\cX_0 \to \cX$ is a pro-immersion.
  Since $\cX_0 \to \cX$ is pro-unramified, it factors as
  $\cX_0 \xrightarrow{j} \cV_0 \xrightarrow{u} \cX$, where $j$ is a
  flat quasi-compact monomorphism and $u$ is unramified and
  of finite type. By
  \cite[Thm.~1.2]{MR2818725}, there is a further factorization
  $\cV_0 \xrightarrow{i} \cX' \xrightarrow{p} \cX$ of $u$, where $i$ is a
  closed immersion and $p$ is \'etale, representable and finitely presented.
  In particular, $\cX_0 \xrightarrow{j} \cV_0 \xrightarrow{i} \cX'$ is a pro-immersion.
  Since $p$ has quasi-affine diagonal, $\cX'$ inherits the conditions
  \eqref{TI:microlocalization:quaff} or \eqref{TI:microlocalization:exc} from
  $\cX$; hence, we may replace $\cX$ by $\cX'$ and assume that $\cX_0 \to \cX$
  is a pro-immersion.

  We thus have a factorization $\cX_0 \xrightarrow{j} \cV_0 \xrightarrow{i} \cX$
  where $j$ is a flat quasi-compact monomorphism and $i$ is a closed immersion.
  Note that $j$ is schematic~\cite[Tag
  \spref{0B8A}]{stacks-project} and even
  quasi-affine~\cite[Prop.~1.5]{raynaud_sem-samuel} and that $\cX_0$
  is noetherian~\cite[Prop.~1.2]{raynaud_sem-samuel}. In particular, $\cW_0$ is
  also noetherian.
  
  Let
  $g_0 = j\circ h_0 \colon \cW_0 \to \cV_0=\thck{\cX}{\cV_0}{0}$ which is flat.
  We claim that it 
  suffices to prove, using induction on $n \ge 1$, that there are
  compatible cartesian diagrams:
  \[
    \xymatrix{\cW_{n-1} \ar[r] \ar[d]_{g_{n-1}} & \cW_n \ar[d]^{g_n}
      \\ \thck{\cX}{\cV_0}{n-1} \ar[r] & \thck{\cX}{\cV_0}{n},}
  \]
  where each $g_n$ is flat and the $\cW_n$ are noetherian. Indeed, the
  flatness of the $g_n$ implies that the resulting system
  $\{\cW_n\}_{n\geq 0}$ is adic. By \Cref{T:effectivity}, the
  completion $\hat{\cW}$ of the sequence
  $\{\cW_n\}_{n\geq 0}$ exists and is noetherian and linearly
  fundamental. If $\cX$ has quasi-affine diagonal, then the morphisms
  $\cW_n \to \cX$ induce a unique morphism $\hat{\cW}$ by Tannaka
  duality (case \itemref{SS:tannaka:b} of \S \ref{SS:tannaka}). If $\cX$ only has affine stabilizers, however,
  then Tannaka duality (case \itemref{SS:tannaka:a} of \S \ref{SS:tannaka}) has the additional hypothesis that
  $\hat{\cW}$ is locally the spectrum of a G-ring, e.g., locally quasi-excellent, which follows from the quasi-excellency of $\Gamma(\cW_0,\Orb_{\cW_0})$ (\Cref{R:qe-effective}). The flatness of $\hat{\cW} \to \cX$ is just the local
  criterion for flatness \cite[$0_{\mathrm{III}}$.10.2.1]{EGA}.

  We now get back to solving the lifting problem. If $g_0$ is not representable, choose an
  affine morphism $\cW_0 \to B \GL_N$ for some $N$.
  Since $B\GL_N$ has smooth diagonal, the induced
  \emph{representable} morphism $\cW_0 \to \cX_0 \times B\GL_N$ is syntomic. Hence, we
  may replace $\cX$ with $\cX \times B\GL_N$ and  assume that $g_0$ is representable. By
  \cite[Thm.~1.4]{olsson-deformation}, the obstruction to lifting
  $g_{n-1}$ to $g_n$ belongs to the group
  $\Ext^2_{\Orb_{\cW_{0}}}(L_{\cW_0/\cV_0},g_0^*(\cI^n/\cI^{n+1}))$,
  where $\cI$ is the coherent ideal sheaf defining the closed
  immersion $i \colon \cV_0 \hookrightarrow \cX$. 

  Now since
  $\cX_0 \to \cV_0$ is a flat monomorphism, $L_{\cX_0/\cV_0} \simeq 0$ \cite[Prop.~17.8]{lmb}. Hence,
  $L_{\cW_0/\cV_0} \simeq L_{\cW_0/\cX_0}$. But $\cW_0 \to \cX_0$ is
  syntomic, so $L_{\cW_0/\cX_0}$ is perfect of amplitude $[-1,0]$
  and $\cW_0$ is
  cohomologically affine. Thus, the Ext-group vanishes, and we have
  the required lift. That $\cW_n$ is noetherian is clear: it is a
  thickening of a noetherian stack by a coherent sheaf of ideals.

  For
  the uniqueness statement:
  Let $h\colon \hat{\cW}\to \cX$ and $h'\colon \hat{\cW'}\to \cX$ be two
  different morphisms as in the theorem. Let $g_n=j_n\circ h_n\colon \cW_n\to
  \thck{\cX}{\cV_0}{n}$ and $g'_n=j_n\circ h'_n\colon \cW'_n\to
  \thck{\cX}{\cV_0}{n}$ be the induced $n$th infinitesimal neighborhoods.
  By Tannaka duality, it is enough
  to show that an isomorphism $f_{n-1}\colon \cW_{n-1}\to \cW'_{n-1}$ lifts
  (resp.\ lifts up to a unique $2$-isomorphism) to an isomorphism
  $f_n\colon \cW_n\to \cW'_n$. The obstruction to a lift lies in
  $\Ext^1_{\Orb_{\cW_{0}}}(f_0^*L_{\cW'_0/\cV_0},g_0^*(\cI^n/\cI^{n+1}))$,
  which vanishes if $h_0=h'_0$ is smooth. The obstruction to the existence
  of a $2$-isomorphism between two lifts lies in
  $\Ext^0_{\Orb_{\cW_{0}}}(f_0^*L_{\cW'_0/\cV_0},g_0^*(\cI^n/\cI^{n+1}))$ and
  the $2$-automorphisms of a lift lies in
  $\Ext^{-1}_{\Orb_{\cW_{0}}}(f_0^*L_{\cW'_0/\cV_0},g_0^*(\cI^n/\cI^{n+1}))$.
  All three groups vanish if $h_0$ is \'etale.
\end{proof}
\subsection{Existence of completions}\label{SS:coh-completions-along-substacks}
If $\cX_0 \to \cX$ is a morphism of algebraic stacks, we
say that a morphism of pairs $(\cW, \cW_0) \to (\cX, \cX_0)$, that is, compatible maps $\cW \to \cX$ and $\cW_0 \to \cX_0$, is the {\it completion of $\cX$ along $\cX_0$}
if $(\cW, \cW_0)$ is a coherently complete pair (\Cref{D:coherently-complete}) and  $(\cW, \cW_0) \to (\cX, \cX_0)$ is final among morphisms from coherently complete pairs.
That is, if $ (\cZ,\cZ_0) \to (\cX,\cX_0)$ is any other morphism of pairs
from a coherently complete pair, there exists a morphism $(\cZ, \cZ_0) \to
(\cW,\cW_0)$ over $\cX$ unique up to unique $2$-isomorphism.  In
particular, the pair $(\cW,\cW_0)$ is unique up to unique
$2$-isomorphism.

We prove the following generalization of \Cref{C:existence-completions}.
\begin{corollary}[Existence of completions]\label{C:existence-completions:general}
Let $\cX$ be a noetherian algebraic stack. Let
$\cX_0\to \cX$ be a pro-immersion such that $\cX_0$ is linearly fundamental,
e.g., the residual gerbe at a point with linearly reductive stabilizer.
If either
\begin{enumerate}
\item $\cX$ has quasi-affine diagonal; or
\item $\cX$ has affine stabilizers
  and $\Gamma(\cX_0,\Orb_{\cX_0})$ is quasi-excellent;
\end{enumerate}
then
the completion of $\cX$ along $\cX_0$ exists and is linearly fundamental.
\end{corollary}
\begin{proof}
Applying \Cref{T:microlocalization:general} to the pro-immersion $\cX_0 \to \cX$ with $\cW_0 = \cX_0$, we obtain a flat morphism
$h\colon \widehat{\cX}\to \cX$ where $\widehat{\cX}$ is a linearly fundamental stack,
coherently complete along $h^{-1}(\cX_0) \cong \cX_0$.
Let $(\cZ,\cZ_0)$ be any
other coherently complete stack with a morphism $\varphi\colon \cZ\to \cX$ such
that $\varphi|_{\cZ_0}$ factors through $\cX_0$. Let $\cI\subset \Orb_{\cX}$ be
the sheaf of ideals defining the closure of $\cX_0$. Then $\cX_n=V(\cI^{n+1}\Orb_{\widehat{\cX}})$
and $\cZ_n\subseteq V(\cI^{n+1}\Orb_{\cZ})$. Since $\cX_n\to
V(\cI^{n+1})$ is a flat monomorphism, it follows that $\cZ_n\to
\cX$ factors uniquely through $\cX_n$. By coherent completeness of $\cZ$ and
Tannaka duality (using that $\widehat{\cX}$ has affine diagonal), there is a
unique morphism $\cZ\to \widehat{\cX}$.
\end{proof}

Let $\cX$ be a noetherian algebraic stack and let $x$ be a point of $\cX$ with linearly reductive stabilizer. Applying \Cref{C:existence-completions:general} to the pro-immersion corresponding to the residual gerbe $\cG_x \to \cX$ of $x$, we obtain a flat morphism $\widehat{\cX}_x \to \cX$, which we refer to as the \emph{completion} at the point $x$. Note that
when $\cG_x=V(\cI)$ is a closed point, then $\widehat{\cX}_x=\varinjlim_n
V(\cI^{n+1})$ in the category of noetherian algebraic stacks with affine stabilizers.

\subsection{Representability properties of presentations}
We will now give a representability criteria for morphisms from fundamental
stacks to algebraic stacks, generalizing \cite[Prop.~3.2 and Prop.~3.4]{luna-field}
and partially answering \cite[Question 1.10]{luna-field}.
This will then be used in the local structure theorem.

\begin{proposition} \label{P:refinement}
 Let $f \co \cW \to \cX$ be a morphism of algebraic stacks such that $\cW$ is adequately affine with affine diagonal (e.g., fundamental). Suppose $\cW_0 \subset \cW$ is a closed substack such that $f|_{\cW_0}$ is representable.
 \begin{enumerate}
 \item \label{P:refinement:affine_diag} If $\cX$ has affine diagonal, then there exists an adequately affine open neighborhood $\cU \subseteq \cW$ of $\cW_0$ such that $f|_{\cU}$ is affine.
 \item \label{P:refinement:separated_diag} If $\cX$ has separated diagonal and $\cW$ is fundamental, then there exists an adequately affine open neighborhood $\cU \subseteq \cW$ of $\cW_0$ such that $f|_{\cU}$ is representable.
 \end{enumerate} 
\end{proposition}
To prove \Cref{P:refinement}, we require the following Lemma.
\begin{lemma}\label{L:fundamental:quasi-finite-subgroup-of-inertia}
Let $\cW$ be a fundamental stack and let $G\inj I_\cW$ be a closed subgroup.
If $G\to \cW$ is quasi-finite, then $G\to \cW$ is finite.
\end{lemma}
\begin{proof}
Note that $I_\cW\to \cW$ is affine so $G\to \cW$ is also affine. If $h\in |G|$
is a point, then the order of $h$ is finite. It is thus enough to prove the
following: if $h\in |I_\cW|$ is a point of finite order such that
$\cZ:=\overline{\{h\}}\to \cW$ is quasi-finite, then $\cZ\to \cW$ is finite.
Using approximation of fundamental stacks
(\Cref{L:excellent-approx-fundamental}) we reduce this question to the case
where $\cW$ is of finite presentation over $\Spec \ZZ$.

By~\cite[Lem.~8.3.1]{alper-adequate}, it is enough to prove that $\cZ\to \cW$
takes closed points to closed points and that the morphism on their adequate
moduli spaces $Z\to W$ is universally closed. This can be checked using DVRs as
follows: for every DVR $R$ with fraction field $K$, every morphism $f\co \Spec
R\to W$ and every lift $h\co \Spec K\to \cZ$, there exists a lift $\tilde{h}\co
\Spec R\to \cZ$ such that the closed point $0\in \Spec R$ maps to a point in
$\cW$ that is closed in the fiber over $f(0)$.

Since $\cW\to W$ is universally closed, we can start
with a lift $\xi\co \Spec R\to \cW$, such that $\xi(0)$ is closed in the fiber
over $f(0)$. We can then identify $h$ with an
automorphism $h\in \Aut_\cW(\xi)(K)$ of finite order. Applying
\cite[Prop.~5.11 and Lem.~5.10]{ahlh} gives us an extension of DVRs $R\inj R'$
and a new lift $\xi'\co \Spec R'\to \cW$ such that $\xi'(0)=\xi(0)$ together
with an automorphism $\tilde{h}\in \Aut_\cW(\xi')(R')$. Since $\cZ$ is closed
in $I_\cW$, this is a morphism $\tilde{h}\co \Spec R\to \cZ$ as requested. Note that while the paper \cite{ahlh} cites this paper on several occasions, the proofs of \cite[Prop.~5.11 and Lem.~5.10 (when $\cX$ is fundamental)]{ahlh} do not rely on it.
\end{proof}

\begin{remark}
If $h\in |I_\cW|$ is any element of finite order, then every element of
$\cZ=\overline{\{h\}} \subseteq I_{\cW}$ is of finite order but $\cZ$ is not always quasi-finite.
For example, consider the action of $\Gm\rtimes \ZZ/2\ZZ$ on $\AA^2$ as in~\cite[Ex.~3.56]{ahlh}. Then the generic point has stabilizer group $\ZZ/2\ZZ$. If we let $h=\tau$ be the non-trivial element of the generic point, then $\overline{\{h\}}$ has fiber
$\{(x/y,\tau)\}$ outside $xy=0$, is empty along $xy=0$ outside $x=y=0$ and
is $\Gm\times \{\tau\}$ over $x=y=0$.

Note that subgroups of inertia stacks are automatically normal and the corresponding result for non-normal quasi-finite subgroup schemes of geometrically reductive group schemes is false. Indeed, take $H \subseteq \GL_{2,k[T]}$ to be the closed subgroup with only non-trivial element $
\begin{bmatrix}
  0 & T^{-1}\\ T & 0
\end{bmatrix}$, which has order $2$. Then $H$ is quasi-finite over $\Spec k[T]$, not finite, but is also not normal in $\GL_{2,k[T]}$.
\end{remark}

\begin{proof}[Proof of \Cref{P:refinement}]
Since $f|_{\cW_0}$ is representable, we can after replacing $\cW$ with an open,
adequately affine, neighborhood of $\cW_0$, assume that $f$ has
quasi-finite diagonal (or in fact, even unramified diagonal).
For 
\itemref{P:refinement:affine_diag} we argue exactly as in
\cite[Prop.~3.2]{luna-field} but replace \cite[Prop.~3.3]{alper-good}
with \cite[Cor.~4.3.2]{alper-adequate}.

For \itemref{P:refinement:separated_diag}, we note that the subgroup
$G:=I_{\cW/\cX}\inj I_\cW$ is closed because $\cX$ has separated diagonal
and is quasi-finite over $\cW$ because $f$ has quasi-finite diagonal.
We conclude by \Cref{L:fundamental:quasi-finite-subgroup-of-inertia}
and Nakayama's lemma.
\end{proof}

\subsection{The local structure theorem}\label{SS:proof-main-theorem}
We will now prove the following local structure theorem. Note that \Cref{T:base} is the case where $\cW_0$ is a gerbe over a field.

\begin{theorem}[Local structure]\label{T:base-general}
  Suppose that
  $S$ is a quasi-separated algebraic space;
 $\cX$ is an algebraic stack, locally of finite presentation and quasi-separated
  over $S$, with affine stabilizers;
 $x\in |\cX|$ is a point with
  residual gerbe $\cG_x$ and image $s\in |S|$ such that the residue field extension $\kappa(x)/\kappa(s)$ is finite; and
 $h_0 \colon \stk{W}_0 \to \cG_x$ is a smooth (resp.\ \'etale) morphism, where  $\stk{W}_0$ is linearly fundamental and $\Gamma(\cW_0,\Orb_{\cW_0})$ is a field.
      Then there exists a cartesian diagram of algebraic stacks
\[
\xymatrix{
\cW_0 \ar[r]^{\smash{h_0}} \ar[d]		& \cG_x \ar[d] \\
\mathllap{[\Spec A/\GL_n] = \;}\cW \ar[r]^h		& \cX
}
\]
where
  $h\colon (\stk{W},w) \to (\cX,x)$ is a 
smooth (resp.\ \'etale) pointed morphism and $w$ is closed in its fiber over $s$.
   Moreover, if $\cX$ has separated (resp.\ affine) diagonal and $h_0$ is representable, then $h$
  can be arranged to be representable (resp.\ affine). 
\end{theorem}
\begin{remark}
In \Cref{T:base,T:base-general}, the condition that $\kappa(x)/\kappa(s)$ is finite
is equivalent to the condition that the morphism $\cG_x\to \cX_s$ is of finite type.
In particular, it holds if $x$ is closed in its fiber $\cX_s = \cX \times_S \Spec \kappa(s)$.
\end{remark}
To prove \Cref{T:base-general}, we will need the following version of equivariant Artin algebraization (cf.~\cite[Thm.~A.18]{luna-field}).
\begin{theorem}[Equivariant Artin algebraization]\label{T:algebraization-fundamental}
Let $S$ be an excellent scheme.
Let $\cX$ be an algebraic stack, locally of finite presentation over $S$.
Let $\cZ$ be a noetherian  fundamental stack with
adequate moduli space map $\pi\colon \cZ\to Z$ of finite type (automatic
if $\cZ$ is linearly fundamental). Let
$z\in |\cZ|$ be a closed point such that $\cG_z\to S$ is of finite type.
Let $\eta\colon \cZ\to \cX$
be a morphism over $S$ that is formally versal at $z$.
Then there exist
\begin{enumerate}
\item \label{TI:algebraization-fundamental:ft} an algebraic stack $\stk{W}$ which is  fundamental and
  of finite type
  over $S$;
\item \label{TI:algebraization-fundamental:pt} a closed point $w\in |\stk{W}|$;
\item \label{TI:algebraization-fundamental:map} a morphism $\xi\colon \stk{W}\to \cX$ over $S$;
\item \label{TI:algebraization-fundamental:fv} isomorphisms $\varphi^{[n]}\colon \thck{\stk{W}}{}{n}\to
  \thck{\cZ}{}{n}$ over $\cX$ for every $n$; and
\item \label{TI:algebraization-fundamental:lr} if $\Stab(z)$ is linearly reductive, an isomorphism
  $\hat{\varphi}\colon \widehat{\stk{W}}\to \widehat{\cZ}$
  over $\cX$, where $\widehat{\stk{W}}$ and $\widehat{\cZ}$ denote the
  completions of $\stk{W}$ at $w$ and $\cZ$ at $z$ which exist by
  \Cref{C:existence-completions}.
\end{enumerate}
In particular, $\xi$ is
formally versal at $w$.
\end{theorem}
\begin{proof}
Since $\cZ$ is a fundamental stack, by definition there exists an affine morphism $\cZ \to B\GL_m$ for some $m>0$. We now apply~\cite[Thm.~A.18]{luna-field} with $T=Z$ and $\cX_1=\cX$ and
$\cX_2=B\GL_m$, which gives \eqref{TI:algebraization-fundamental:ft}--\eqref{TI:algebraization-fundamental:fv}. Claim \eqref{TI:algebraization-fundamental:lr} is an immediate consequence
of \eqref{TI:algebraization-fundamental:fv}.
\end{proof}
\begin{proof}[Proof of \Cref{T:base-general}]
\textbf{Step 1: Reduction to $S$ an excellent scheme.}
It is enough to find a solution $(\stk{W},w)\to (\cX,x)$ after replacing
$S$ with an \'etale neighborhood of $s$ so we can assume that $S$ is affine.
We can also
replace $\cX$ with a quasi-compact neighborhood of $x$ and assume that
$\cX$ is of finite presentation.

Write $S$ as a limit of
affine schemes $S_\lambda$ of finite type over $\Spec \ZZ$. For
sufficiently large $\lambda$, we can
find $\cX_\lambda\to S_\lambda$ of finite presentation such that
$\cX=\cX_\lambda\times_{S_\lambda} S$. Let $w_0 \in |\cW_0|$ be the unique
closed point and let $x_\lambda\in
|\cX_\lambda|$ be the image of $x$. Since $\cG_x$ is the limit of the
$\cG_{x_\lambda}$, we can, for sufficiently large $\lambda$,
also find a smooth (or \'etale if $h_0$ is \'etale) morphism
$h_{0,\lambda}\colon (\stk{W}_{0,\lambda}, w_{0, \lambda})\to (\cG_{x_\lambda}, x_\lambda)$ with pull-back
$h_0$. For sufficiently large $\lambda$:
\begin{enumerate}
\item $\cX_\lambda$ has affine
stabilizers~\cite[Thm.~2.8]{hallj_dary_alg_groups_classifying};
\item if $\cX$ has separated (resp.\ affine) diagonal, then so has $\cX_\lambda$;
\item $\Stab(x_\lambda)=\Stab(x)$ (because
  $\Stab(x_\mu)\to \Stab(x_\lambda)$ is a closed immersion for every
  $\mu>\lambda$); and
\item $\stk{W}_{0,\lambda}$ is fundamental (\Cref{P:approximation-fundamental}).
\end{enumerate}
That $\cG_x\to \cG_{x_\lambda}$ is stabilizer-preserving implies that
$\cG_x=\cG_{x_\lambda}\times_{\Spec \kappa(x_\lambda)} \Spec \kappa(x)$ and, in
particular, $\stk{W}_0=\stk{W}_{0,\lambda}\times_{\Spec \kappa(x_\lambda)} \Spec
\kappa(x)$. It follows, by flat descent, that $\stk{W}_{0,\lambda}$ is cohomologically affine and that $\Gamma(\stk{W}_{0,\lambda},\Orb_{\stk{W}_{0,\lambda}})$ is the spectrum
of a field.  We can thus replace $S$, $\cX$, $\stk{W}_0$ with
$S_\lambda$, $\cX_\lambda$, $\stk{W}_{0,\lambda}$ and assume that $S$ is
an excellent scheme. By standard limit arguments, it is also enough to find
a solution after replacing $S$ with $\Spec \Orb_{S,s}$. We can thus assume
that $s$ is closed.

\textbf{Step 2: An effective formally smooth solution.}
Since $\cW_0$ is linearly fundamental and $\cX$ has affine stabilizers,
we can find a formal neighborhood of $\stk{W}_0\to \cX_0:=\cG_x \hookrightarrow \cX$, that is, deform the smooth morphism  $\stk{W}_0\to \cX_0$ to a flat morphism 
$\hat{\stk{W}}\to \cX$ where $\hat{\stk{W}}$ is a linearly fundamental stack which is
coherently complete along $\stk{W}_0$ (\Cref{T:microlocalization}).
Since $\stk{W}_n\to \cX_n$ is smooth, $\hat{\stk{W}}\to
\cX$ is formally smooth at $\stk{W}_0$~\cite[Prop.~A.14]{luna-field}.
Since the good moduli space of $\stk{W}_0$ is a field, hence equals the good moduli
space of the residual gerbe $\cG_{w_0}$, it follows that
$\stk{W}_0$ and $\hat{\stk{W}}$ are coherently complete along $w_0$ (\Cref{T:complete}).

\textbf{Step 3: Algebraization.}  We now apply equivariant Artin algebraization 
(\Cref{T:algebraization-fundamental}), with $\stk{Z}=\widehat{\stk{W}}$, to obtain
a fundamental
stack $\stk{W}$, a closed point $w\in |\stk{W}|$, a morphism
$h\colon (\stk{W},w)\to (\cX,x)$ smooth
at $w$, and an isomorphism $\widehat{\stk{W}}_w\cong \widehat{\stk{W}}$
over $\cX$. Let $\widetilde{\stk{W}}_0=h^{-1}(\overline{\stk{X}}_0)$. Then we have
the cartesian diagrams
\[
\vcenter{\xymatrix{%
  \stk{W}_0\ar[r]^{c_0}\ar[d] & \widetilde{\stk{W}}_0\ar[r]\ar[d] & \overline{\stk{X}}_0\ar[d] \\
  \widehat{\stk{W}}\ar[r]^{c} & \stk{W}\ar[r]^h & \stk{X}
}}\qquad\text{and}\quad%
\vcenter{\xymatrix{%
  \stk{W}_0\ar[r]^{c_0}\ar[d] & \widetilde{\stk{W}}_0\ar[d]^{\pi_0} \\
  \Spec \widehat{\Orb}_{\widetilde{W}_0,\pi_0(w)}\ar[r] & \widetilde{W}_0
}}%
\]
where $c$ and $c_0$ are completions at $w$ and $\pi_0$ is an adequate moduli spaces.
But $\stk{W}_0$ has good moduli space $\Spec k$ so $\widetilde{W}_0$ is the disjoint
union of $\{\pi_0(w)\}=\Spec k$ and its complement $Q$. If $\pi\colon \stk{W}\to W$ is
the adequate moduli space, then $\widetilde{W}_0\to W$ is closed and injective so
after replacing $W$ with an open neighborhood $V$ of $\pi(w)$ and replacing
$\stk{W}$, $\widetilde{W}_0$, $\widetilde{W}_0$ with the inverse images of $V$, we
can assume that $\widetilde{\stk{W}}_0=\stk{W}_0$.

Note that if $U\subset \stk{W}$ is an open neighborhood of $w$, we can shrink to the
smaller open
neighborhood $\pi^{-1}(V)$, where $V$ is an open affine neighborhood of $\pi(w)$
contained in $W\smallsetminus \pi(\stk{W}\smallsetminus U)\bigr)$;
then $\pi^{-1}(V) \to V$ remains adequately affine.

If $h_0\colon \stk{W}_0\to \stk{X}_0$ is \'etale, then $h$ is
\'etale at $w$. After shrinking $\stk{W}$ as above, we can assume that $h$ is smooth
(resp.\ \'etale). If $\cX$ has separated (resp.\ affine) diagonal, then we can shrink $\stk{W}$ as above 
so that $h$ becomes representable (resp.\ affine), see \Cref{P:refinement}.
\end{proof}

\section{Applications to stacks with good moduli spaces}\label{S:applications}
In this section, we prove that if $\pi\colon \stX\to X$ is a good moduli space,
with affine stabilizers and separated diagonal, then $\stX$ has the resolution
property \'etale-locally on $X$ (\Cref{T:etale-local-gms}). This generalizes
\cite[Thm.~4.12]{luna-field} to the relative case. As a consequence,
linearly reductive groups are Nisnevich-locally embeddable.
We also give a version for adequate moduli spaces (\Cref{T:etale-local-ams}).
It follows that the derived category of a stack with a good moduli space
is compactly generated (\Cref{P:compact-generation}).

\subsection{Good moduli spaces and linearly reductive groups}\label{SS:applications:gms}

\begin{theorem} \label{T:etale-local-gms}
Let $\cX$ be an algebraic stack with good moduli space $\pi \co \cX \to X$.
Assume that $\cX$ has affine stabilizers, separated diagonal and
is of finite presentation over a quasi-compact and quasi-separated algebraic
space $S$.
\begin{enumerate}
\item\label{TI:etale-local-gms:lin-fund}
  There is a Nisnevich covering $X'\to X$ such that the pull-back
  $\cX'=\cX\times_X X'$ is linearly fundamental.
\item\label{TI:etale-local-gms:affine-diag}
  $\pi\co \cX \to X$ has affine diagonal.
\item\label{TI:etale-local-gms:fp1}
  $\cX \to X$ and $X\to S$ are of finite presentation.
\item\label{TI:etale-local-gms:fp2}
  $\pi_*\sF$ is finitely presented if $\sF$ is a finitely presented
  $\Orb_\cX$-module.
\end{enumerate}
Moreover, if every closed point $x\in |X|$ either has $\kar \kappa(x)>0$ or has
an open neighborhood of characteristic zero, then we can arrange that $\cX'
\cong [\Spec A/G]$ where $G \to X'$ is linearly reductive and embeddable.
\end{theorem}

We will prove \Cref{T:etale-local-gms} at the end of \S \ref{SS:nice-nbhds} after establishing the adequate
version and some auxiliary results on \'etale neighborhoods of points with
nice stabilizers.

If there are closed points of characteristic zero without characteristic
zero neighborhoods, then it is sometimes impossible to find a linearly
reductive $G$; see \Cref{A:mixed-char-counterexamples}. See \Cref{T:etale-local-gms-connected} for some variants in mixed characteristic, however.

Applying \Cref{T:etale-local-gms} to the classifying stack $BG$ of a linearly
reductive group scheme we obtain:

\begin{corollary}\label{C:lin-red-groups} %
Let $S$ be a quasi-separated algebraic space.
\begin{enumerate}
\item\label{C:lin-red:sep+af=>affine}
  If $G\to S$ is a separated group algebraic space, flat and of finite
  presentation, with affine fibers such that $BG\to S$ is a good moduli space,
  then $G\to S$ is affine, that is, $G$ is linearly reductive.
\item\label{C:lin-red:etale-loc-emb}
  If $G\to S$ is linearly reductive, then there exists a Nisnevich covering $S'\to S$ such that $G'=G\times_S S'$ is embeddable.\epf
\end{enumerate}
\end{corollary}

\begin{remark}
A consequence of \Cref{C:lin-red-groups} is that in the definition of 
a tame group scheme given in
\cite[Defn.\ 2.26]{hoyois_six-operations}, if we assume that $G\to B$ is
separated with affine fibers, then the condition on having the $G$-resolution
property Nisnevich-locally is automatic.
\end{remark}

Over a field, normal subgroups, quotients and extensions of reductive groups
are reductive. The analogous statement holds for linearly reductive groups
over a base.

\begin{corollary}
Let $S$ be an algebraic space. Let $1\to G'\to G\to G''\to 1$ be an exact
sequence of flat group algebraic spaces of finite presentation over $S$. Then
the following are equivalent:
\begin{enumerate}
\item\label{C:lin-red-ses:middle}
  $G\to S$ is linearly reductive and $G'\to G$ is a closed immersion.
\item\label{C:lin-red-ses:outer}
  $G'\to S$ and $G''\to S$ are linearly reductive.
\end{enumerate}
\end{corollary}
\begin{proof}
\itemref{C:lin-red-ses:middle}$\implies$\itemref{C:lin-red-ses:outer}:
  Since $G'\to G$ is closed and $G\to S$ is affine, the quotient $G''\to S$
  is separated with affine fibers. Since $BG\to S$ is cohomologically affine,
  so is $BG''\to S$ \cite[Prop.~12.17(i)]{alper-adequate}. By
  \Cref{C:lin-red-groups}\itemref{C:lin-red:sep+af=>affine}, $G''\to S$
  is linearly reductive and in particular affine. Since $G''$ is affine,
  the $G''$-torsor $BG'\to BG$ is affine so
  $BG'\to S$ is also cohomologically affine, hence linearly reductive.

\itemref{C:lin-red-ses:outer}$\implies$\itemref{C:lin-red-ses:middle}:
  Since $G'\to S$ is affine, so is the $G'$-torsor $G\to G''$. Since $G''\to S$
  is affine, so is $G\to S$. The result then follows by
  \cite[Prop.~12.17(ii)]{alper-adequate}.
\end{proof}

\subsection{Adequate moduli spaces and geometrically reductive groups}
\begin{theorem} \label{T:etale-local-ams}
Let $\cX$ be an algebraic stack with adequate moduli space $\pi \co \cX \to X$ and let
$x\in X$ be a point. Assume that
\begin{enumerate}
\item\label{TI:gms:aff}
  $\cX$ has affine stabilizers and separated diagonal;
\item\label{TI:gms:fp}
  $\cX$ is of finite presentation over a quasi-separated algebraic space; and
\item\label{TI:gms:linred}
  the unique closed point $y$ in $\pi^{-1}(x)$ has linearly reductive stabilizer.
\end{enumerate}
Then there exists an \'etale neighborhood $(X', x') \to (X,x)$ with $\kappa(x') = \kappa(x)$ such that the pull-back $\cX'$ of $\cX$ is fundamental. That is, there is
a cartesian diagram
\[
\xymatrix{
\mathllap{[\Spec A / \GL_n]=\;}\cX'\ar[r]^{f} \ar[d]_{\smash{\pi'}}		& \cX \ar[d]^{\smash{\pi}} \\
\mathllap{\Spec B=\;}X' \ar[r]							& X. \ar@{}[ul]|\square
}
\]
where $\pi'$ is an adequate moduli space (i.e., $B = A^{\GL_n}$). 
In particular, $\pi$ has affine diagonal in an open neighborhood of $x$.
\end{theorem}
\begin{proof}
Applying \Cref{T:base} with $h_0 \co \cW_0 \to \cG_y$ an isomorphism yields an \'etale representable morphism 
$f \co ([\Spec A / \GL_n], w) \to (\cX, y)$ inducing an isomorphism $\cG_w \to \cG_y$. The result follows from Luna's fundamental lemma (\Cref{L:fundamental-lemma}).
\end{proof}

\begin{corollary}\label{C:geomred-linred-etale-loc-embeddable}
Let $S$ be a quasi-separated algebraic space. Let $G\to S$ be a
flat and separated group algebraic space of finite presentation with affine fibers 
such that $BG\to S$ is adequately affine (e.g., $G\to S$ is geometrically reductive).
If $s\in S$ is a point such that
$G_s$ is linearly reductive, then there exists an \'etale neighborhood
$(S',s')\to (S,s)$, with trivial residue field extension, such that $G'=G\times_S
S'$ is embeddable.
\end{corollary}
\begin{proof}
From \Cref{T:etale-local-ams} we obtain an \'etale neighborhood $S'\to S$ such
that $BG'$ is fundamental. Then $G'$ is affine and
embeddable (\Cref{R:BG-fundamental}).
\end{proof}

\begin{remark}
If $G\to S$ is a reductive group scheme (i.e., geometrically reductive, smooth,
and with connected fibers) then $G\to S$ is \'etale-locally split reductive.  A
split reductive group is a pull-back from $\Spec \ZZ$~\cite[Exp.~XXV, Thm.~1.1,
  Cor.~1.2]{MR0274459}, hence embeddable.
\end{remark}

\subsection{Nice neighborhoods}\label{SS:nice-nbhds}
Parts~\itemref{TI:etale-local-gms:lin-fund} and
\itemref{TI:etale-local-gms:affine-diag} of \Cref{T:etale-local-gms} follow
directly from \Cref{T:etale-local-ams}.
To deduce the rest of the theorem, we need to study the structure around
points of positive characteristic. The following proposition shows that
fundamental stacks (resp.\ geometrically reductive
and embeddable group schemes) are nicely fundamental
(resp.\ nice) in an \'etale neighborhood of a nice point.

\begin{proposition}[Niceness is \'etale-local]\label{P:nice-neighborhood}
  Let $\cX$ be a fundamental algebraic stack with adequate moduli space
  $\cX\to X$. Let $x\in |X|$ be a point and let $y\in |\cX|$ be the unique closed
  point in the fiber of $x$.
  If the stabilizer of $y$ is nice, then there exists an \'etale neighborhood
  $(X',x')\to (X,x)$, with $\kappa(x')=\kappa(x)$, such that $\cX\times_X X'$
  is nicely fundamental.
\end{proposition}
\begin{proof}
  Since nicely fundamental stacks can be approximated
  (\Cref{P:approximation-fundamental}\itemref{PI:approx:fundamental}), we may assume that $X$ is henselian
  with closed point $x$. Then $y$ is the unique closed point of $|\cX|$.
  By \Cref{P:lr-gerbe-field}\eqref{PI:lr-gerbe-field:nice}, the residual gerbe $\cG_y=\overline{\{y\}}$ is nicely fundamental.

  Now \Cref{L:excellent-approx-fundamental} says that we can write $\cX=\varprojlim_\lambda \cX_\lambda$, where the $\cX_\lambda$
  are fundamental and of finite type over $\Spec \ZZ$ with adequate moduli
  space $X_\lambda$ of finite type over $\Spec \ZZ$. Let $x_\lambda\in |X_\lambda|$ be
  the image of $x$ and let $y_\lambda\in |\cX_\lambda|$ be the unique closed
  point above $x_\lambda$. Then $y_\lambda$ is contained in the closure of
  the image of $y$. Thus, for all sufficiently large $\lambda$,
  the point $y_\lambda$
  has nice stabilizer (\Cref{P:approximation-fundamental}\itemref{PI:approx:nice-gerbes}).

  Let $X_\lambda^h$ denote the henselization of $X_\lambda$ at $x_\lambda$
  and $\cX_\lambda^h=\cX_\lambda\times_{X_\lambda} X_\lambda^h$.
  Then the canonical map $X\to X_\lambda$ factors uniquely through $X_\lambda^h$
  and the induced map $\cX\to \cX_\lambda^h$ is affine. It is thus enough
  to prove that $\cX_\lambda^h$ is nicely fundamental.
  By \Cref{T:adequate+lin-red=>good}, the adequate moduli space
  $\cX_\lambda^h\to X_\lambda^h$ is good, that is, $\cX_\lambda^h$ is
  linearly fundamental.

  We can thus assume that $X$ is excellent and henselian and that $\cX$ is
  linearly fundamental.
  Let $\cX_n$ be the $n$th infinitesimal
  neighborhood of $x$. Let $Q_0 \to \Spec \kappa(x)$ be a
  nice group scheme such that there exists an affine morphism $f_0 \co \cX_0
  \to B_{\kappa(x)} Q_0$.
  By the existence of deformations of nice group
  schemes (\Cref{P:nice-deformations}), there exists a nice and embeddable
  group scheme $Q \to X$. Let $\cI \subset \cX$ denote the sheaf of ideals
  defining $\cX_0$.  By \cite[Thm.~1.5]{olsson-deformation}, the obstruction to
  lifting a morphism $\cX_{n-1} \to B_X Q$ to $\cX_{n} \to
  B_X Q$ is an element of $\Ext^1_{\oh_{\cX_0}}(Lf_0^* L_{B_XQ/X},
  \cI^{n}/\cI^{n+1})$. The obstruction vanishes because the cotangent complex
  $L_{B_XQ/X}$ is perfect of amplitude $[0,1]$, since $B_XQ \to X$ is smooth,
  and $\cX_0$ is cohomologically affine.

  Let $\hat{X} = \Spec \hat{\oh}_{X,x}$ and $\hat{\cX} = \cX \times_X \hat{X}$.
  Since $\hat{\cX}$ is linearly fundamental, it is coherently complete along
  $\cX_0$ (\Cref{T:complete:res-prop}). By
  Tannaka duality (see \S\ref{SS:tannaka}), we may thus extend $\cX_0 \to
  B_{X_0}Q_0$ to a morphism $\hat{\cX} \to B_X Q$. Applying Artin
  approximation (\Cref{T:artin-approximation}) to the functor $\Hom_X(\cX \times_X -, B_X Q) \co (\SCH{X})^{\opp} \to
  \SETS$ yields a morphism $\cX \to B_X Q$, which is affine by
  \Cref{P:refinement}\itemref{P:refinement:affine_diag}.
\end{proof}

Note that if $\cX$ is linearly fundamental and $\kar \kappa(x)>0$, then $y$
has nice stabilizer. We thus have the following corollaries:

\begin{corollary} \label{C:linearly-fundamental:lin-red-quot}
Let $\cX$ be a linearly fundamental algebraic stack with good moduli space
$\cX\to X$ and let $x\in |X|$ be a point.
If either $\kar \kappa(x)>0$ or $x$
has an open neighborhood of characteristic zero, then there exists an \'etale
neighborhood $(X',x')\to (X,x)$, with $\kappa(x')=\kappa(x)$, such that
$\cX\times_X X' = [\Spec A / G]$ where $G\to X'$ is a linearly reductive
embeddable group scheme.\epf
\end{corollary}

\begin{corollary} \label{C:nice-charp}
Let $(S,s)$ be a Henselian local scheme such that $\kar \kappa(s) > 0$. 
\begin{enumerate}
	\item If $\cX$ is a linearly fundamental algebraic stack with good moduli
      space $\cX \to S$, then $\cX$ is nicely fundamental.
	\item  If $G \to S$ is a linearly reductive and embeddable group scheme, then $G \to S$ is nice.\epf
\end{enumerate}
\end{corollary}

We also obtain the following non-noetherian variants of
\Cref{T:adequate+lin-red=>good} at the expense of assuming that
either $\cX$ has the resolution property or is of finite presentation over
some base.
\begin{corollary}\label{C:adequate+lin-red=>good:fundamental}
Let $\cX$ be a fundamental algebraic stack. Then the following are equivalent.
\begin{enumerate}
\item\label{CI:adequate+lin-red=>good:fundamental:lf} $\cX$ is linearly fundamental.
\item\label{CI:adequate+lin-red=>good:fundamental:lr} Every closed point of $\cX$ has linearly reductive stabilizer.
\item\label{CI:adequate+lin-red=>good:fundamental:nice} Every closed point of $\cX$ with positive characteristic has nice stabilizer.
\end{enumerate}
\end{corollary}
\begin{proof}
The only non-trivial implication is \itemref{CI:adequate+lin-red=>good:fundamental:nice}$\implies$\itemref{CI:adequate+lin-red=>good:fundamental:lf}.
Let $\pi\colon \cX\to X$ be the adequate moduli space. It is enough to prove that $\pi$ is a good moduli space after base change to
the henselization at a closed point. We may thus assume that $X$ is the spectrum of a henselian local ring.
If $X$ is a $\QQ$-scheme, then the notions of adequate and good coincide.
If not, then the closed point of $X$ has positive characteristic, hence the
unique closed point of $\cX$ has nice stabilizer. We conclude that
$\cX$ is nicely fundamental by \Cref{P:nice-neighborhood}.
\end{proof}
\begin{corollary}\label{C:adequate+lin-red=>good2}
Let $\cX$ be an algebraic stack of finite presentation over a quasi-compact and quasi-separated
algebraic space $S$. Suppose that there exists an adequate moduli space $\pi \co \cX \to X$.
Then $\pi$ is a good moduli space with affine diagonal if and only if
\begin{enumerate}
\item $\cX$ has separated diagonal and affine stabilizers; and
\item every closed point of $\cX$ has linearly reductive stabilizer.
\end{enumerate}
\end{corollary}
\begin{proof}
The conditions are clearly necessary. If they are satisfied, then \Cref{T:etale-local-ams} implies that $\pi$ has affine diagonal. To verify that $\pi$ is
a good moduli space, we may replace $X$ with the henselization at a closed point.
Then $\cX$ is fundamental by \Cref{T:etale-local-ams} and
the result follows from \Cref{C:adequate+lin-red=>good:fundamental}.
\end{proof}

We will now finish the proof of \Cref{T:etale-local-gms}.

\begin{corollary}\label{C:base-change-of-ams/gms}
Let $\cX$ be a fundamental stack with adequate moduli space $\pi\colon \cX\to
X$. Let $g\colon X'\to X$ be a morphism of algebraic spaces such that
$\cX':=\cX\times_X X'$ has a good moduli space. Then $\pi'\colon \cX'\to X'$ is
its good moduli space and the natural transformation $g^*\pi_* \to \pi'_*g'^*$
is an isomorphism on all quasi-coherent $\Orb_\cX$-modules.
\end{corollary}
\begin{proof}
Both claims can be checked on stalks
so we may assume that
$X'=\Spec A'$ and $X=\Spec A$ are spectra of local rings and that the closed
point $x'\in X'$ maps to the closed point $x\in X$. Since $\cX'$ has a good
moduli space, it follows that the unique closed point of $\cX$ has linearly
reductive stabilizer. Hence $\cX$ is linearly fundamental
(\Cref{C:adequate+lin-red=>good:fundamental}) and the result follows from
\cite[Prop.~4.7]{alper-good}.
\end{proof}

\begin{corollary}\label{C:gms-is-of-fp}
Let $\cX$ be a linearly fundamental stack of finite presentation over a
quasi-separated algebraic space $S$ with good moduli space $\pi\colon \cX\to X$.
Then $X$ is of finite presentation over $S$ and $\pi_*$ takes finitely
presented $\Orb_\cX$-modules to finitely presented $\Orb_X$-modules.
\end{corollary}
\begin{proof}
We may assume that $S$ is quasi-compact and can thus write $S$ as an inverse
limit of algebraic spaces $S_\lambda$ of finite presentation over $\Spec \ZZ$
with affine transition maps \cite[Thm.~D]{rydh-2009}. For sufficiently large
$\lambda$, we can find $\cX_\lambda\to S_\lambda$ of finite presentation that
pulls back to $\cX\to S$. After increasing $\lambda$, we can assume that
$\cX_\lambda$ is fundamental by
\Cref{P:approximation-fundamental}\itemref{PI:approx:fundamental} and
that a given $\Orb_\cX$-module $\sF$ of finite presentation is the pull-back
of a coherent $\Orb_{\cX_\lambda}$-module $\sF_\lambda$. Then
$\cX_\lambda$ has an adequate moduli space $X_\lambda$ of finite presentation
over $S_\lambda$ and the push-forward of $\sF_\lambda$ is a coherent
$\Orb_{X_\lambda}$-module \cite[Thm.~6.3.3]{alper-adequate}. The result now
follows from \Cref{C:base-change-of-ams/gms}. In particular,
$X=X_\lambda\times_{S_\lambda} S$
is the good moduli space of $\cX$.
\end{proof}

\begin{proof}[Proof of \Cref{T:etale-local-gms}]
By \Cref{T:etale-local-ams}, we obtain a Nisnevich covering $X' \to X$ such that $\pi' \colon \cX' = \cX \times_X X'$ is fundamental. Since $\pi \colon \cX \to X$ is a good moduli space, it follows that $\pi'$ is also a good moduli space and so linearly fundamental. This proves \itemref{TI:etale-local-gms:lin-fund} and
\itemref{TI:etale-local-gms:affine-diag}. By \Cref{C:gms-is-of-fp}, we see that \itemref{TI:etale-local-gms:fp1} and \itemref{TI:etale-local-gms:fp2} hold for $\pi'$ and, by \'etale  descent, also for $\pi$. The final claim follows
from \Cref{C:linearly-fundamental:lin-red-quot}.
\end{proof}

\subsection{Compact generation of derived categories} \label{SS:compact-generation-results}
Here we prove a variant of \cite[Thm.~5.1]{luna-field} in the mixed characteristic situation. 
\begin{proposition}\label{P:compact-generation}
  Let $\cX$ be a quasi-compact algebraic stack with good moduli space
  $\pi \colon \cX \to X$. If $\cX$ has affine stabilizers, separated
  diagonal and is of finite presentation over a quasi-separated
  algebraic space $S$, then $\cX$ has the Thomason condition; that is,
  \begin{enumerate}
  \item $\DQCOH(\cX)$ is compactly generated by a countable set of perfect complexes; and 
  \item for every quasi-compact open immersion $\cU \subseteq \cX$,
    there exists a compact and perfect complex $P \in \DQCOH(\cX)$
    with support precisely $\cX\setminus \cU$.
  \end{enumerate}
\end{proposition}
\begin{proof}
  By \cite[Thm.~C]{perfect_complexes_stacks} and
  \cite[Prop.~8.4]{perfect_complexes_stacks}, it suffices to construct
  an \'etale, separated and representable covering
  $p\colon \cW \to \cX$ such that $\cW=[\Spec C/\GL_n]$ (note that
  $\cW$ is automatically concentrated because $\cX$ and $p$ are).  This
  follows from \Cref{T:etale-local-gms}.
\end{proof}

\section{Approximation and deformation of linearly fundamental stacks}\label{S:approximation-deformation}
In this section we use the local structure of good moduli stacks
(\Cref{T:etale-local-gms}) and nice points (\Cref{P:nice-neighborhood}) to extend
the approximation results for fundamental and nicely fundamental stacks
in \Cref{SS:approximation-fund} to linearly fundamental stacks (\Cref{T:approximation-of-lin-fund:relative}) and good moduli
spaces (\Cref{C:gms-approximation}).
We will then deform objects over henselian pairs (\Cref{D:pairs}), using
the approximation results to reduce from the henselian case to the
excellent henselian case.
These approximation and deformation results will be used 
prominently in \S \ref{S:refinements}--\ref{S:further-applications}, 
and in this section we give a first consequence by generalizing 
the universal property of good moduli spaces (\Cref{T:universal}) 
to good moduli space morphisms (\Cref{T:universal-gms-variant}). 

To this end, we introduce the following
mild mixed characteristic assumptions on an algebraic stack~$\cW$:
\begin{enumerate}
\myitem{FC}\label{Cond:FC} There is only a finite number of
  different characteristics in $\cW$.
\myitem{PC}\label{Cond:PC} Every closed point of $\cW$ has
  positive characteristic.
\myitem{N}\label{Cond:N} Every closed point of $\cW$ has nice
  stabilizer.
\end{enumerate}
\begin{remark}\label{R:fc-closed}
  Note that if $\eta \leadsto s$ is a specialization in $\cW$, then
  the characteristic of $\eta$ is $0$ or agrees with that of $s$. In
  particular, if $(\cW,\cW_0)$ is a local pair (\Cref{D:pairs}), so that
  every closed point of $\cW$ belongs to $\cW_0$, it follows that 
  $\cW$ satisfies
  \ref{Cond:FC}, \ref{Cond:PC} or \ref{Cond:N}, respectively, if and only if $\cW_0$ does so.
\end{remark}

\subsection{Approximation}

\newcommand{\nicelocus}{\mathrm{nice}}

  Let $\cX$ be a fundamental stack with adequate moduli space
  $\pi \co \cX\to X$. Let
  $X_\nicelocus\subseteq |X|$ be the locus of points $x\in |X|$ such that the
  unique closed point in the fiber $\pi^{-1}(x)$ has nice stabilizer.
  If $x\in X_\nicelocus$, then there exists an \'etale neighborhood
  $X'\to X$ of $x$ such that $\cX\times_X X'$ is nicely fundamental
  (\Cref{P:nice-neighborhood}). It follows that $X_\nicelocus$ is open
  and that $\cX\times_X X_\nicelocus\to X_\nicelocus$ is a good moduli space.

\begin{lemma}\label{L:abs-approximation:lin-fund-pos-char}
Let $\cX$ be a fundamental stack with adequate moduli space $X$.
Let $\cX=\varprojlim_\lambda \cX_\lambda$ be an inverse limit of fundamental
stacks with affine transition maps. Let $X_\lambda$ denote the adequate moduli space
of $\cX_\lambda$. Then
\begin{enumerate}
\item \label{LI:abs-approximation:lin-fund-pos-char:pullback} $X\times_{X_\lambda} (X_\lambda)_\nicelocus\subseteq X_\nicelocus$ for
every $\lambda$; and
\item \label{LI:abs-approximation:lin-fund-pos-char:qc} if $V\subseteq X_\nicelocus$ is a quasi-compact open subset, then
$V\subseteq X\times_{X_\lambda} (X_\lambda)_\nicelocus$ for
every sufficiently large $\lambda$.
\end{enumerate}
\end{lemma}
\begin{proof}
Note that the subtlety is that while the map $\cX\to \cX_\lambda\times_{X_\lambda} X$ is always affine, it is not necessarily an isomorphism.

For $x\in (X_\lambda)_\nicelocus$, let $U_\lambda\to X_\lambda$ be an \'etale
neighborhood of $x$ such that $\cX_\lambda\times_{X_\lambda} U_\lambda$ is nicely fundamental (\Cref{P:nice-neighborhood}). Then $\cX\times_{X_\lambda} U_\lambda$ is also
nicely fundamental as it is affine over the former. Thus
$X\times_{X_\lambda} (X_\lambda)_\nicelocus\subseteq X_\nicelocus$. This proves \itemref{LI:abs-approximation:lin-fund-pos-char:pullback}.

For \itemref{LI:abs-approximation:lin-fund-pos-char:qc}, let $U\to V$ be an \'etale surjective morphism such that $\cX\times_X
U$ is nicely fundamental (\Cref{P:nice-neighborhood}). Since $X=\varprojlim_\lambda
X_\lambda$ (\Cref{P:approximation-fundamental}\itemref{PI:approx:adequate}) and
$U\to X$ is affine, we can for all sufficiently large $\lambda$ find
$U_\lambda\to X_\lambda$ affine \'etale such that
$U=U_\lambda\times_{X_\lambda} X$.  Since $\cX\times_X U = \varprojlim_\lambda
\cX_\lambda\times_{X_\lambda} U_\lambda$ is nicely fundamental, so is
$\cX_\lambda\times_{X_\lambda} U_\lambda$ for all sufficiently large $\lambda$
(\Cref{P:approximation-fundamental}\itemref{PI:approx:fundamental}).  It
follows that $(X_\lambda)_\nicelocus$ contains the image of $U_\lambda$ so
$V \subseteq X\times_{X_\lambda} (X_\lambda)_\nicelocus$.
\end{proof}

The main theorem of this section is the following variant of
\Cref{P:approximation-fundamental}\itemref{PI:approx:fundamental} for
linearly fundamental stacks.

\begin{theorem}[Approximation of linearly fundamental]\label{T:approximation-of-lin-fund:relative}
Let $\cY$ be a quasi-compact and quasi-separated algebraic stack. Let
$\cX=\varprojlim_\lambda \cX_\lambda$ where $\cX_\lambda$ is an inverse system
of quasi-compact and quasi-separated algebraic stacks over $\cY$ with affine
transition maps. Assume that
\begin{enumerate*}
\item $\cY$ is \ref{Cond:FC}, or
\item $\cX$ is \ref{Cond:PC}, or 
\item $\cX$ is \ref{Cond:N}.
\end{enumerate*}
Then, if $\cX$ is linearly fundamental, so is $\cX_\lambda$ for all sufficiently
large $\lambda$.
\end{theorem}
\begin{proof}
By \Cref{P:approximation-fundamental}\itemref{PI:approx:fundamental} we can
assume that the $\cX_\lambda$ are fundamental.
Since $\cX$ is linearly fundamental, \ref{Cond:PC}$\implies$\ref{Cond:N}.
If $\cX$ satisfies \ref{Cond:N}, then $X_\nicelocus=X$ and it follows from \Cref{L:abs-approximation:lin-fund-pos-char} that $(X_\lambda)_\nicelocus=X_\lambda$ for all sufficiently large $\lambda$; hence that $\cX_\lambda$ is linearly fundamental.
Thus, it remains to prove the theorem when $\cY$ satisfies~\ref{Cond:FC}.
In this case, $\cY_\QQ:=\cY\times_{\Spec \ZZ} \Spec \QQ$ is open in 
$\cY$. Similarly for the other stacks. In particular, if $X$ denotes the good
moduli space of $\cX$, then $X$ is the union of the two open subschemes
$X_\nicelocus$ and $X_\QQ$. In addition, since $X\smallsetminus X_\QQ$ is
closed, hence quasi-compact, we may find a quasi-compact open subset
$V\subseteq X_\nicelocus$ such that $X=V\cup X_\QQ$. For sufficiently large
$\lambda$, we have
that $V\subseteq (X_\lambda)_\nicelocus\times_{X_\lambda} X$ (\Cref{L:abs-approximation:lin-fund-pos-char}\itemref{LI:abs-approximation:lin-fund-pos-char:qc}) and thus, after possibly increasing $\lambda$,
that $X_\lambda = (X_\lambda)_\nicelocus\cup (X_\lambda)_\QQ$.
It follows that $\cX_\lambda$ is linearly fundamental.
\end{proof}

\begin{corollary}\label{C:excellent-approx-lin-fund}
Let $\cX$ be a linearly fundamental stack. Assume that $\cX$ satisfies
\ref{Cond:FC}, \ref{Cond:PC} or \ref{Cond:N}.
Then we can write $\cX=\varprojlim_\lambda \cX_\lambda$ as an inverse limit of
linearly fundamental stacks, with affine transition maps, such that
each $\cX_\lambda$ is essentially of finite type over $\Spec \ZZ$.
\end{corollary}
\begin{proof}
If $\cX$ satisfies \ref{Cond:FC}, let $S$ be the semi-localization of $\Spec \ZZ$ in all
characteristics that appear in $\cX$. Then there is a canonical map $\cX\to
S$. If $\cX$ satisfies \ref{Cond:PC} or \ref{Cond:N}, let $S=\Spec \ZZ$.
Since $\cX$ is fundamental, we can write $\cX$ as an inverse limit of algebraic
stacks $\cX_\lambda$ that are fundamental and of finite presentation over $S$.
The result then follows from \Cref{T:approximation-of-lin-fund:relative}.
\end{proof}

\Cref{C:excellent-approx-lin-fund} is not true unconditionally, even if we
merely assume that the $\cX_\lambda$ are noetherian, see
\Cref{A:mixed-char-counterexamples}.

\begin{corollary}[Approximation of good moduli spaces]\label{C:gms-approximation}
Let $X=\varprojlim_\lambda X_\lambda$ be an inverse system of quasi-compact
algebraic spaces with affine transition maps. Let $\alpha$ be an index, let
$f_\alpha\co \cX_\alpha \to X_\alpha$ be a morphism of finite presentation and
let $f_\lambda\co \cX_\lambda\to X_\lambda$, for $\lambda\geq \alpha$, and $f\co
\cX\to X$ denote the base changes of $f_\alpha$. Assume that $X_\alpha$ satisfies
\ref{Cond:FC} or $\cX$ satisfies \ref{Cond:PC} or \ref{Cond:N}. Then if $\cX\to X$
is a good moduli space with affine diagonal,
so is $\cX_\lambda\to X_\lambda$ for all sufficiently
large~$\lambda$.
\end{corollary}
\begin{proof}
\Cref{T:etale-local-gms} gives an \'etale and surjective morphism
$X'\to X$ such that $\cX'=\cX\times_X X'$ is linearly fundamental. For
sufficiently large $\lambda$, we can find an \'etale surjective morphism
$X'_\lambda\to X_\lambda$ that pulls back to $X'\to X$. For sufficiently
large $\lambda$, we have that $\cX'_\lambda:=\cX_\lambda\times_{X_\lambda} X'_\lambda$ is linearly fundamental by \Cref{T:approximation-of-lin-fund:relative}.
Its good moduli space $\overline{X}'_\lambda$ is of finite presentation over
$X'_\lambda$ (\Cref{C:gms-is-of-fp}). It follows that $\overline{X}'_\lambda\to
X'_\lambda$ is an isomorphism for all sufficiently large $\lambda$. By descent,
it follows that $\cX_\lambda\to X_\lambda$ is a good moduli space for all
sufficiently large $\lambda$.
\end{proof}

\subsection{Deformation: setup}
For most of the remainder of this section, we will be in the following situation. 
\begin{setup}\label{Setup:deformations}
  Let $\cX$ be a quasi-compact algebraic stack with affine diagonal
  and affine good moduli space $X$. Let $\cX_0 \inj \cX$ be a closed
  substack with good moduli space $X_0$. Assume that $(X,X_0)$ is an
  affine henselian pair and one of the following conditions holds:
\begin{enumerate}[label=(\alph*), ref=\alph*]
\item\label{SetupI:deformations:complete} $\cX_0$ has the resolution property, $\cX$ is noetherian and $(X,X_0)$ is complete;
\item\label{SetupI:deformations:exc} $\cX_0$ has the resolution property, $\cX$ is noetherian and $(X,X_0)$ is quasi-excellent;
\item\label{SetupI:deformations:fcpcn} $\cX$ has the resolution property and $\cX_0$ satisfies \ref{Cond:FC}, \ref{Cond:PC}, or \ref{Cond:N}; or
\customitem{(c$'$)}{c$'$}\label{SetupI:deformations:fcpcn-2} $\cX_0$ has the resolution property, $\cX\to X$ is of finite presentation, and $\cX_0$ satisfies \ref{Cond:FC}, \ref{Cond:PC}, or \ref{Cond:N}.
\end{enumerate}
\end{setup}
\begin{remark}\label{R:deformations-setup}
  Note that \ref{Cond:FC} and \ref{Cond:PC} for $\cX_0$ are 
  equivalent to the corresponding properties for $X_0$. Since the pair
  $(X,X_0)$ is henselian and so local, it follows that these are
  equivalent to the corresponding properties for $X$ and $\cX$ (\Cref{R:fc-closed}).
\end{remark}

\subsection{Deformation of the resolution property}
The first result of this section is the following remarkable proposition. It is a simple consequence of some results proved several sections ago.
\begin{proposition}[Deformation of the resolution property]\label{P:deformation-resprop}
Let $(\cX,\cX_0)$ be as in
\Cref{Setup:deformations} \itemref{SetupI:deformations:complete},
\itemref{SetupI:deformations:exc} or \itemref{SetupI:deformations:fcpcn-2}.
Then $\cX$ has the resolution property; in particular, $\cX$ is linearly fundamental
and \itemref{SetupI:deformations:fcpcn-2} implies \itemref{SetupI:deformations:fcpcn}.
\end{proposition}
\begin{proof}
Case \itemref{SetupI:deformations:complete} is part of the coherent completeness result
(\Cref{T:complete}). For \itemref{SetupI:deformations:exc}, let $\widehat{X}$ denote the completion
of $X$ along $X_0$ and $\widehat{\cX}=\cX\times_X \widehat{X}$. By the complete case,
$\widehat{\cX}$ has the resolution property. Equivalently, there is a
quasi-affine morphism $\widehat{\cX}\to B\GL_n$ for some $n$. The functor
parametrizing quasi-affine morphisms to $B\GL_n$ is locally of finite
presentation~\cite[Thm.~C]{rydh-2009} so by Artin approximation (\Cref{T:artin-approximation}), there
exists a quasi-affine morphism $\cX\to B\GL_n$. In particular, $\cX$ has the resolution property.

For \itemref{SetupI:deformations:fcpcn-2}, let $S=\Spec \ZZ$ be
the semi-localization of $\Spec \ZZ$ in all characteristics that appear in
$X$. Since $(X,X_0)$ is an affine henselian pair over $S$, we may write it as
an inverse limit of affine excellent henselian pairs
$(X_\lambda,X_{\lambda,0})$ over $S$.
For a sufficiently large $\lambda$, we have a morphism
of finite presentation $\cX_\lambda\to X_\lambda$ such that the pull-back along
$X\to X_\lambda$ is $\cX\to X$. After increasing $\lambda$ we can assume that
$\cX_\lambda\to X_\lambda$ is a good moduli space
(\Cref{C:gms-approximation}) and that $\cX_\lambda\times_{X_\lambda}
X_{\lambda,0}$ has the resolution property
(\Cref{P:approximation-fundamental}\itemref{PI:approx:fundamental}).
Since $X_\lambda$ is excellent, $\cX_\lambda$ has the resolution property by case
\itemref{SetupI:deformations:exc} and we conclude
that $\cX$ has the resolution property since $\cX\to \cX_\lambda$ is affine.
\end{proof}

\subsection{Deformation of sections}
If $f \co \cX' \to \cX$ is a morphism of algebraic stacks, we will denote the groupoid of sections $s \co \cX \to \cX'$ of $f$ as $\Gamma(\cX'/\cX)$ \cite[I.2.5.5]{EGA} (the notation $\mathrm{Sec}(\cX'/\cX)$ is also used \cite{MR2239345}).

\begin{proposition}[Deformation of sections]\label{P:deformation-sections}
Let $(\cX,\cX_0)$ be as in \Cref{Setup:deformations}.
Let $f\colon \cX' \to \cX$ be a quasi-separated and smooth (resp.\ smooth gerbe, resp.\  
\'etale) morphism with affine stabilizers. In case
\Cref{Setup:deformations} \itemref{SetupI:deformations:complete}, also assume that
$f$ has quasi-affine diagonal. Then $\Gamma(\cX'/\cX)\to
\Gamma(\cX'\times_{\cX}\cX_0/\cX_0)$ is essentially surjective
(resp.\ essentially surjective and full, resp.\ an equivalence of groupoids). 
\end{proposition}

\begin{proof}
  Any section $s_0$
  of $\cX'\times_{\cX}\cX_0\to \cX_0$ has quasi-compact image. In
  particular, we may immediately reduce to the situation where $f$ is
  finitely presented.
  
  We first handle case \itemref{SetupI:deformations:complete}:
By \Cref{T:complete}, $\cX$ is coherently complete along $\cX_0$.  
Let $\cI$ be the ideal sheaf defining $\cX_0 \subset \cX$ and let $\cX_n := \cX^{[n]}_{\cX_0}$ be its nilpotent thickenings.  
Set $\cX'_n = \cX' \times_{\cX} \cX_n$.
Let $s_0 \co \cX_0 \to \cX'_0$ be a section of $\cX'_0 \to \cX_0$.  Given a section $s_{n-1}$ of $\cX'_{n-1} \to \cX_{n-1}$, lifting $s_0$, the obstruction to deforming $s_{n-1}$ to a section $s_n$ of $\cX'_n \to \cX_n$ is an element of $\Ext^1_{\oh_{\cX_0}}(Ls_0^* L_{\cX'/\cX}, \cI^{n}/\cI^{n+1})$ by \cite[Thm.~1.5]{olsson-deformation}.\footnote{Note that \cite[Thm.~1.5]{olsson-deformation} only treats
  the case of embedded deformations over a base \emph{scheme}. In the
  case of a relatively flat target morphism, however, this can be
  generalized to a base algebraic stack by deforming the graph and employing
  \cite[Thm.~1.1]{olsson-deformation}, together with the tor-independent base
  change properties properties of the cotangent complex. In the situation at
  hand we may also simply apply \cite[Thm.~1.1]{olsson-deformation}
  to $s_n\co \cX_n\to \cX'$ and $\cX_n\inj \cX$.}
Since $\cX' \to \cX$ is smooth (resp.\ a smooth gerbe, resp.\ \'etale), the cotangent complex $L_{\cX'/\cX}$  is perfect of amplitude
  $[0,1]$ (resp.\ perfect of amplitude $1$, resp.\ zero).  Further $\cX_0$ is cohomologically affine, so 
  there exists a lift (resp.\ a unique lift up to non-unique 2-isomorphism, resp.\ a unique lift up to unique 2-isomorphism).  
By Tannaka duality (see \S \ref{SS:tannaka}), these sections lift to a unique section $s \co \cX \to \cX'$. Here we need that $f$ has quasi-affine diagonal unless $X$ is quasi-excellent.

We now handle case \itemref{SetupI:deformations:exc}.  Let $\hat{X}$ be the completion of $X$ along $X_0$ and set $\hat{\cX} = \cX \times_X \hat{X}$ and $\hat{\cX}' = \cX' \times_X \hat{X}$.  Case \itemref{SetupI:deformations:complete} yields a section $\hat{s} \co \hat{\cX} \to 
\hat{\cX}'$ extending $s_0$.  The functor assigning an $X$-scheme $T$ to the set of sections $\Gamma(\cX' \times_X T / \cX \times_X T)$ is limit preserving, and we may apply Artin approximation (\Cref{T:artin-approximation}) to obtain a section of $s \co \cX' \to \cX$ restricting to $s_0$. 

Finally, we handle case \itemref{SetupI:deformations:fcpcn}. By
\Cref{C:excellent-approx-lin-fund}, we may write the linearly fundamental stack
$\cX$ as an inverse limit of linearly fundamental excellent stacks
$\cX_\lambda$, with affine transition maps. Since $\cX_0$ is the intersection
of finitely presented closed substacks \cite{rydh-2014}, we can also write the
pair $(\cX,\cX_0)$ as an inverse limit of pairs $(\cX_\lambda,\cX_{\lambda,0})$.
Since the good moduli space $(X,X_0)$ is a henselian pair, the induced map
$(X,X_0)\to (X_\lambda,X_{\lambda,0})$ on good moduli spaces factors through
the henselization $(X^h_\lambda,X_{\lambda,0})$. After replacing $\cX_\lambda$
with $\cX_\lambda\times_{X_\lambda} X^h_\lambda$, we may thus assume
that $(X_\lambda,X_{\lambda,0})$ is henselian for every $\lambda$.

Since $\cX' \to \cX$ is smooth (resp.~a smooth gerbe, resp.~\'etale) and
finitely presented, after possibly increasing $\lambda$ it descends to
$\cX'_\lambda \to \cX_\lambda$ and retains its properties of being smooth
(resp.~a smooth gerbe, resp.~\'etale) \cite[Prop.~B.3]{rydh-2009}.  If $s_0
\colon \cX_0 \to \cX'_0$ is a section of $\cX'_0\to \cX_0$, then for
sufficiently large $\lambda$, we have a section $s_{\lambda,0} \colon
\cX_{\lambda,0}\to \cX'_{\lambda,0}$. From case \itemref{SetupI:deformations:exc},
we obtain a section $s_\lambda \colon \cX_\lambda\to \cX'_\lambda$, hence
a section $s\colon \cX\to \cX'$ as requested.

The full and full faithfulness statements in \itemref{SetupI:deformations:exc} and  \itemref{SetupI:deformations:fcpcn} can be deduced using similar methods: the quasi-excellent case can be reduced to the complete case using Artin approximation, and the non-excellent case can be reduced to the excellent case using limits.
\end{proof}

\subsection{Deformation of morphisms}

A simple application of \Cref{P:deformation-sections} yields a deformation result of morphisms.

\begin{proposition}\label{P:deformation-morphisms}
Let $(\cX,\cX_0)$ be as in \Cref{Setup:deformations}.
Let $\cY \to X$ be a quasi-separated and smooth (resp.\ smooth gerbe, resp.\ \'etale) morphism with affine stabilizers. In case
\Cref{Setup:deformations} \itemref{SetupI:deformations:complete}, also assume that
$\cY\to X$ has quasi-affine diagonal.
Then any morphism $\cX_0 \to \cY$ can be extended (resp.\ extended uniquely up to non-unique 2-isomorphism, resp.\ extended uniquely up to unique 2-isomorphism) to a morphism $\cX \to \cY$.  In particular,
\begin{enumerate}
\item \label{PI:deformation-morphisms:fet} the natural functor $\FET(\cX)\to \FET(\cX_0)$ between the categories of finite \'etale covers is an equivalence;
\item \label{PI:deformation-morphisms:vect} the natural functor $\VB(\cX)\to \VB(\cX_0)$ between the categories of vector bundles is essentially
  surjective and full;
\item \label{PI:deformation-morphisms:quot} if $G\to X$ is an affine flat
  group scheme of finite presentation and $\cX_0=[\Spec A/G]$, then there is a $G$-equivariant closed immersion $\Spec A \hookrightarrow \Spec B$ over $X$ that induces $\cX_0 \hookrightarrow \cX=[\Spec B/G]$; and
\item \label{PI:deformation-morphisms:nice} if $\cX_0$ is nicely fundamental,
  then so is $\cX$.
\end{enumerate}
\end{proposition}
\begin{proof} For the main statement, apply \Cref{P:deformation-sections} with $\cX' = \cX \times_X \cY$.
For \itemref{PI:deformation-morphisms:fet}, apply the result to $\cY=\coprod_n BS_{n,X}$ noting
that $BS_n$ classifies finite \'etale covers of degree $n$.
Similarly, for \itemref{PI:deformation-morphisms:vect}, apply the result to
$\cY=\coprod_n B\GL_{n,X}$. For \itemref{PI:deformation-morphisms:quot}, apply the result to $\cY =
BG$ together with \Cref{P:refinement}\itemref{P:refinement:affine_diag}
to ensure that the induced
morphism $\cX \to BG$ is affine. For \itemref{PI:deformation-morphisms:nice},
note that, by definition, $\cX_0=[\Spec A/G_0]$ where $G_0\to S_0$ is nice and
embeddable. We next deform $G_0$ to a nice and embeddable group scheme $G\to S$
(\Cref{P:nice-deformations}) and then apply
\itemref{PI:deformation-morphisms:quot}.
\end{proof}

\subsection{Deformation of linearly fundamental stacks}
If $(S,S_0)$ is an affine complete noetherian pair and $\cX_0$ is a linearly fundamental stack with a syntomic morphism $\cX_0 \to S_0$ that is a good moduli space, \Cref{T:microlocalization} constructs a noetherian and linearly fundamental stack $\cX$ that is flat over $S$, such that $\cX_0 = \cX \times_S S_0$ and $\cX$ is coherently complete along $\cX_0$.  The following lemma shows that $\cX \to S$ is also a good moduli space. We also consider non-noetherian generalizations.

\begin{lemma} \label{L:deformations-fundamental}
  Let $\cX$ be a quasi-compact algebraic stack with affine diagonal and affine good moduli space $X$. Let $\pi \colon \cX \to S$ be a flat morphism. Let $S_0 \inj S$ be a closed immersion. Let $\cX_0 = \cX\times_S S_0$ and assume $\pi_0 \colon \cX_0 \to S_0$ is a good moduli space and $(\cX,\cX_0)$ is a local pair. In addition, assume that $(S,S_0)$ is an affine local pair and 
  \begin{enumerate}[label=(\alph*), ref=\alph*]
  \item \label{LC:deformations-fundamental:complete} $\cX$ is
    noetherian and $(S,S_0)$ is complete;
  \item \label{LC:deformations-fundamental:ft} $\cX$ is
    noetherian and $\pi$ is of finite type; or
  \item \label{LC:deformations-fundamental:fcpcn}
    $\pi$ is of finite presentation and $\cX_0$
    satisfies \ref{Cond:FC}, \ref{Cond:PC}, or \ref{Cond:N}.
  \end{enumerate}
  Then $\pi$ is a good moduli space morphism of finite presentation.  Moreover,
  \begin{enumerate}
    \item \label{LI:deformations-fundamental:syntomic}
    if $\pi_0$ is syntomic (resp.\ smooth, resp.\ \'etale), then so is
      $\pi$; and
    \item \label{LI:deformations-fundamental:gerbe}
    if $\pi_0$ is an fppf gerbe (resp.\ a smooth gerbe, resp.\ an \'etale
      gerbe), then so is $\pi$.
    \end{enumerate}
\end{lemma}
\begin{proof}
  We first show that $\pi$ is a good moduli space morphism of finite presentation.

  Let $S=\spec A$, $S_0 = \spec (A/I)$ and $X=\spec B$. Since
  $(\cX,\cX_0)$ is a local pair, it follows that
  $(\spec B,\spec B/IB)$ is a local pair. In particular, $IB$ is
  contained in the Jacobson radical of $B$. Note that if $\cX$ is
  noetherian, then $\cX \to \spec B$ is of finite
  type~\cite[Thm.~A.1]{luna-field}. Moreover, in the commuting diagram:
  \[
    \xymatrix{\cX_0 \ar@{_(->}[d] \ar[r] & \spec (B/IB) \ar@{_(->}[d] \ar[r] & S_0 \ar@{_(->}[d]
      \\ \cX \ar[r] & X \ar[r] & S,}
  \]
  the outer rectangle is cartesian, as is the right square, so it
  follows that the left square is cartesian. Since the formation of
  good moduli spaces is compatible with arbitrary base change, it
  follows that the morphism $A/I \to B/IB$ is an isomorphism.

  Case \itemref{LC:deformations-fundamental:complete}: let
  $A_n=A/I^{n+1}$ and $\cX_n=V(I^{n+1}\Orb_{\cX})$. Since $\cX_n$ is
  noetherian and $\pi_n \co \cX_n\to S_n:=\Spec A/I^{n+1}$ is flat, it
  follows that $B_n=\Gamma(\cX_n,\Orb_{\cX_n})=B/I^{n+1}B$ is a
  noetherian and flat $A_n=A/I^{n+1}$-algebra
  \cite[Thm.~4.16(ix)]{alper-good}. But $B_n/IB_n=A/I$ so $A_n\to B_n$
  is surjective and hence an isomorphism. Let $\hat{B}$ be the
  $IB$-adic completion of $B$; then the composition
  $A \to B \to \hat{B}$ is an isomorphism and $B \to \hat{B}$ is
  faithfully flat because $IB$ is contained in the Jacobson radical of
  $B$. It follows immediately that $A \to B$ is an isomorphism.

  Case \itemref{LC:deformations-fundamental:ft}: now the image
  of $\pi$ contains $S_0$ and by flatness is stable under
  generizations; it follows immediately that $\pi$ is faithfully
  flat. Since $\cX$ is noetherian, it follows that $S$ is
  noetherian.
  We may now base change everything along the
  faithfully flat morphism $\spec \hat{A} \to \spec A$, where
  $\hat{A}$ is the $I$-adic completion of $A$. By faithfully flat
  descent of good moduli spaces, we are now reduced to Case
  \itemref{LC:deformations-fundamental:complete}.

  Case \itemref{LC:deformations-fundamental:fcpcn}: the good moduli space $X\to S$ is also
  of finite presentation (\Cref{T:etale-local-gms}). The result then
  follows from \itemref{LC:deformations-fundamental:ft} using an approximation
  argument similar to that employed in the proof of \Cref{P:deformation-resprop}.

  Now claim \itemref{LI:deformations-fundamental:syntomic} follows because the conditions are open and all closed points of $\cX$ lie in $\cX_0$. For claim
  \itemref{LI:deformations-fundamental:gerbe}, since $\cX\to S$ and
  $\cX\times_S \cX\to S$ are flat and $\cX_0$ contains all closed
  points, the fiberwise criterion of flatness shows that
  $\Delta_{\cX/S}$ is flat if and only if $\Delta_{\cX_0/S_0}$ is
  flat.  It then follows that $\Delta_{\cX/S}$ is smooth (resp.\
  \'etale) if $\Delta_{\cX_0/S_0}$ is so.
\end{proof}

Combining \Cref{T:microlocalization}/\Cref{T:microlocalization:general} and \Cref{L:deformations-fundamental} with Artin approximation yields the following result. 

\begin{proposition}[Deformation of linearly fundamental stacks]\label{P:deformation-linearly-fundamental}
  Let $\pi_0 \colon \cX_0 \to S_0$ be a good moduli space,
  where $\cX_0$ is linearly fundamental. Let $(S,S_0)$ be an affine
  henselian pair and assume one of the following conditions:
\begin{enumerate}[label=(\alph*), ref=\alph*]
\item \label{PC:deformation-linearly-fundamental:complete} $(S,S_0)$ is a noetherian complete pair; 
\item \label{PC:deformation-linearly-fundamental:excellent}$S$ is quasi-excellent; or
\item \label{PC:deformation-linearly-fundamental:fcpcn} $\cX_0$
  satisfies \ref{Cond:FC}, \ref{Cond:PC}, or \ref{Cond:N}.
\end{enumerate}
If $\pi_0$ is syntomic, then there exists a syntomic morphism
$\pi\colon \cX \to S$ that is a good moduli space such that:
  \begin{enumerate}
  \item \label{PI:deformation-linearly-fundamental:lift} $\cX \times_{S} S_0 \cong \cX_0$;
  \item \label{PI:deformation-linearly-fundamental:linearly-fundamental}$\cX$ is linearly fundamental;
  \item \label{PI:deformation-linearly-fundamental:coh-complete} $\cX$ is coherently complete along $\cX_0$ if $(S,S_0)$ is a noetherian complete pair;
  \item \label{PI:deformation-linearly-fundamental:smooth/etale} $\pi$
    is smooth (resp.\ \'etale) if $\pi_0$ is smooth (resp.\ \'etale); and
  \item \label{PI:deformation-linearly-fundamental:gerbe} $\pi$
    is an fppf (resp.\ smooth, resp.\ \'etale) gerbe if $\pi_0$ is such a gerbe.
\end{enumerate}
Moreover, if $\pi_0$ is smooth (resp.\ a smooth gerbe, resp.\ \'etale), then $\pi$ is unique up to non-unique isomorphism (resp.\ non-unique
$2$-isomorphism, resp.\ unique $2$-isomorphism).
\end{proposition}

\begin{proof}
In case \itemref{PC:deformation-linearly-fundamental:complete}: the existence of a flat morphism $\cX\to S$ 
satisfying \itemref{PI:deformation-linearly-fundamental:lift}--\itemref{PI:deformation-linearly-fundamental:coh-complete} is immediate from \Cref{T:microlocalization:general}\itemref{TI:microlocalization:quaff}
  applied to $\cX_0 \to S_0 \to S$.
\Cref{L:deformations-fundamental}\itemref{LC:deformations-fundamental:complete} implies that $\cX \to S$ is syntomic, a good moduli space, and satisfies \itemref{PI:deformation-linearly-fundamental:smooth/etale}--\itemref{PI:deformation-linearly-fundamental:gerbe}.
If $\cX' \to S$ is another lift, the uniqueness statements follow by applying \Cref{P:deformation-morphisms} with $\cY = \cX'$.

In case \itemref{PC:deformation-linearly-fundamental:excellent}: consider the functor   
assigning an $S$-scheme $T$ to the set of isomorphism classes of fundamental 
stacks $\cY$ over $T$ such that $\pi \co \cY \to T$ is syntomic. This functor is limit preserving by \Cref{P:approximation-fundamental}\itemref{PI:approx:fundamental}, so we may use the construction in the complete case and Artin approximation (\Cref{T:artin-approximation}) to obtain a fundamental stack $\cX$ over $S$ such that $\cX \times_S S_0 = \cX_0$ and $\cX \to S$ is syntomic. An application of \Cref{L:deformations-fundamental}\itemref{LC:deformations-fundamental:ft} completes the argument again.

Case \itemref{PC:deformation-linearly-fundamental:fcpcn} follows from
case \itemref{PC:deformation-linearly-fundamental:excellent} by
approximation (similar to that used in the proof of
\Cref{P:deformation-resprop}).
\end{proof}

\subsection{Deformation of linearly reductive groups}

As a direct consequence of \Cref{P:deformation-linearly-fundamental}, we can prove the following result, cf.\ \Cref{P:nice-deformations}.

\begin{proposition}[Deformation of linearly reductive group schemes]\label{P:deformation-linearly-reductive}
Let $(S,S_0)$ be an affine henselian pair and
$G_0\to S_0$ a linearly reductive and embeddable group scheme. Assume one
of the following conditions:
\begin{enumerate}[label=(\alph*),ref=\alph*]
	\item $(S,S_0)$ is a noetherian complete pair; 
	\item $S$ is quasi-excellent; or
	\item $G_0$ has nice fibers at closed points or $S_0$ satisfies \ref{Cond:PC} or \ref{Cond:FC}.
\end{enumerate}
Then there exists a linearly reductive and embeddable group scheme $G\to S$ such that $G_0 = G \times_S S_0$.  If, in addition, $G_0 \to S_0$ is smooth (resp.\ \'etale), then $G \to S$ is smooth (resp.\ \'etale) and unique up to non-unique (resp.\ unique) isomorphism.
\end{proposition}

\begin{proof}
Applying \Cref{P:deformation-linearly-fundamental} to $BG_0 \to S_0$ yields a linearly fundamental fppf gerbe $\cX \to S$ such that $BG_0 = \cX \times_S S_0$. By \Cref{P:deformation-sections}, we may extend the canonical section $S_0 \to BG_0$ to a section $S \to \cX$ with the stated uniqueness property.
We conclude that $\cX$ is isomorphic to $BG$ for an fppf affine group scheme $G \to S$ extending $G_0$.  Since $BG$ is linearly fundamental, $G \to S$ is linearly reductive and embeddable (see \Cref{R:BG-fundamental}).  
\end{proof}

\begin{remark}
  When $G_0\to S_0$ is a split reductive group scheme, then the existence of
  $G\to S$ follows from the classification of reductive groups: $G_0\to S_0$ is
  the pull-back of a split reductive group over $\Spec \ZZ$~\cite[Exp.~XXV,
    Thm.~1.1, Cor.~1.2]{MR0274459}. Our methods require linear reductivity
  but also work for non-connected, non-split and non-smooth group schemes.
\end{remark}

\subsection{Extension over \'etale neighborhoods}
\label{SS:deformations:extensions-etale}
In this subsection, we consider the problem of extending objects
over \'etale neighborhoods. Recall that if $\pi\co \cX\to X$ is an adequate
moduli space, then a morphism $\cX'\to \cX$ is \emph{strongly \'etale} if
$\cX'=\cX\times_X X'$ for some \'etale morphism $X'\to X$
(\Cref{D:strongly-etale}).

\begin{proposition}[Extension of gerbes]\label{P:extension-gerbes}
Let $(S,S_0)$ be an affine pair. Let $\pi_0\co \cX_0\to S_0$ be an fppf gerbe (resp.\ smooth gerbe, resp.\ \'etale gerbe). Suppose that
$\cX_0$ is linearly
fundamental and satisfies \ref{Cond:PC}, \ref{Cond:N} or \ref{Cond:FC}.
Then, there exists an \'etale neighborhood $S'\to S$ of $S_0$ and a fundamental
fppf gerbe (resp.\ smooth gerbe, resp.\ \'etale gerbe) $\pi\co \cX'\to S'$
extending $\pi_0$.
\end{proposition}
\begin{proof}
The henselization $S^h$ of $(S,S_0)$ is the limit of the
affine \'etale neighborhoods $S'\to S$ of $S_0$ so the result follows
from \Cref{P:deformation-linearly-fundamental} and
\Cref{P:approximation-fundamental}\itemref{PI:approx:fundamental}.
\end{proof}

\begin{proposition}[Extension of groups]\label{P:extension-groups}
Let $(S,S_0)$ be an affine pair. Let $G_0\to S_0$ be a linearly
reductive and embeddable group scheme. Suppose that $G_0$ has nice fibers or
that $S_0$ satisfies \ref{Cond:PC} or \ref{Cond:FC}.
Then, there exists an \'etale neighborhood $S'\to S$ of $S_0$ and a geometrically
reductive embeddable group $G'\to S'$ extending~$G_0$.
\end{proposition}
\begin{proof}
Argue as before, using \Cref{P:deformation-linearly-reductive}
and \Cref{L:approximation-reductivity}.
\end{proof}

\begin{proposition}[Extension of the resolution property]\label{P:extension-resprop}
Let $X$ be an affine scheme and let $\cX\to X$ be an adequate moduli
space of finite presentation and affine diagonal. Let $\cX_0$ be a closed substack
which is linearly fundamental. Suppose that $X$ is quasi-excellent or that
$\cX_0$ satisfies \ref{Cond:PC}, \ref{Cond:N} or \ref{Cond:FC}.
Then there exists a strongly \'etale
neighborhood $\cX'\to \cX$ of $\cX_0$ such that $\cX'$ is fundamental.
\end{proposition}
\begin{proof}
Let $X^h$ be the henselization of $X$ along $X_0=\pi_0(\cX_0)$. Since $X_0\inj
X^h$ contains all closed points, it follows that $\cX\times_X X^h$ is linearly
fundamental (\Cref{C:adequate+lin-red=>good2} and
\Cref{P:deformation-resprop}). Since $X^h$ is the limit of all affine \'etale
neighborhoods of $X_0$ the result follows from
\Cref{P:approximation-fundamental}\itemref{PI:approx:fundamental}.
\end{proof}

\begin{proposition}[Extension of sections and morphisms]\label{P:extension-morphisms}
Let $(\cX,\cX_0)$ be a fundamental pair over an algebraic stack $S$. Suppose
that $\cX_0$ is linearly fundamental and satisfies \ref{Cond:PC}, \ref{Cond:N}
or \ref{Cond:FC}. Given one of the following:
\begin{enumerate}
\item \label{PI:extension-morphisms:section}
  a section $s_0\co \cX_0\to \cY_0$ of a smooth morphism $\cY\to \cX$
  that is quasi-separated with affine stabilizers;
\item an $S$-morphism $f_0\co \cX_0\to \cY$ where $\cY\to S$ is a smooth
  morphism that is quasi-separated with affine stabilizers;
\item an affine $S$-morphism $f_0\co \cX_0\to \cY$
  where $\cY\to S$ is a smooth morphism with affine diagonal;
\item a vector bundle $\cE_0$ on $\cX_0$; or
\item a finite \'etale morphism $\cW_0\to \cX_0$.
\end{enumerate}
Then there exists a
  strongly \'etale neighborhood $\cX'\to \cX$ of $\cX_0$ such that
the object over $\cX_0$ ($s_0$, $f_0$, $\cE_0$ or $\cW_0$) extends to a
corresponding object over $\cX'$.
\end{proposition}
\begin{proof}
Let $X$ be the adequate moduli space of $\cX$ and $X_0\subseteq X$ the image of
$\cX_0$. Then $(X,X_0)$ is an affine pair and its henselization $X^h$ is the
limit of \'etale neighborhoods $X'\to X$ of $X_0$. Since $X_0\inj X^h$ contains
all closed points, it follows that $\stX^h:=\stX\times_X X^h$ is linearly
fundamental by~\Cref{C:adequate+lin-red=>good:fundamental}. The result
follows from \Cref{P:deformation-sections,P:deformation-morphisms},
\Cref{P:refinement}\itemref{P:refinement:affine_diag} and standard limit
methods.
\end{proof}

\begin{proposition}[Extension of nicely fundamental]\label{P:extension-nicefund}
Let $(\cX,\cX_0)$ be a fundamental pair. If $\cX_0$ is nicely fundamental, then
there exists a strongly \'etale neighborhood $\cX'\to \cX$ of $\cX_0$ such that
$\cX'$ is nicely fundamental.
\end{proposition}
\begin{proof}
As in the previous proof, it follows that $\stX^h$ is linearly fundamental,
hence nicely fundamental by
\Cref{P:deformation-morphisms}\itemref{PI:deformation-morphisms:nice}.  By
\Cref{P:approximation-fundamental}\itemref{PI:approx:fundamental}, there exists
an \'etale neighborhood $X'\to X$ of $X_0$ such that $\cX':=\cX\times_X X'$ is
nicely fundamental.
\end{proof}

\begin{proposition}[Extension of linearly fundamental]\label{P:extension-linfund}
Let $(\cX,\cX_0)$ be a fundamental pair. Suppose that $\cX_0$ satisfies
\ref{Cond:PC}, or \ref{Cond:N}, or that $\cX$ satisfies \ref{Cond:FC} in an open
neighborhood of $\cX_0$. If $\cX_0$ is linearly fundamental, then there exists a
saturated open neighborhood $\cX' \subset \cX$
of $\cX_0$ such that $\cX'$ is linearly fundamental.
\end{proposition}

\begin{proof}
Let $X$ be the adequate moduli space of $\cX$ and $X_0$ the image of $\cX_0$.
The Zariskification $X^Z$ of $X$ is the limit of all affine open neighborhoods
$X'\to X$ of $X_0$. Since $X_0\inj X^Z$ contains all closed points, the stack
$\stX^Z:=\stX\times_X X^Z$ is linearly fundamental
(\Cref{C:adequate+lin-red=>good:fundamental}). By
\Cref{T:approximation-of-lin-fund:relative}, there exists an open neighborhood
$X'\to X$ of $X_0$ such that $\cX':=\cX\times_X X'$ is linearly fundamental.
\end{proof}

\begin{remark}\label{R:extension-non-closed-points}
Note that when $S_0$ is a single point, then \ref{Cond:FC} always holds for
$S_0$ and for objects over $S_0$.  In the results of this subsection, the
substacks $S_0\subseteq S$ and $\cX_0\subseteq \cX$ are by definition closed
substacks. The results readily generalize to the following situation:
$S_0=\{s\}$ is any point and $\cX_0=\cG_x$ is the residual gerbe of a point $x$
closed in its fiber over the adequate moduli space.
\end{remark}

\subsection{Universal property of good moduli space morphisms}

\begin{theorem}[Universal property]\label{T:universal-gms-variant}
Let $\pi\colon \cX\to \cY$ be a good moduli space morphism of finite presentation between
algebraic stacks. Let $\cZ$ be an algebraic stack with quasi-separated diagonal
and let $f\colon \cX\to \cZ$ be a morphism.
Then $f$ factors through $\pi$ if and only if the induced map on inertia, $I_\pi\to f^*I_\cZ$, factors through the identity section. Moreover, the factorization of $f$ through $\pi$ is unique up to unique $2$-isomorphism and if $\cY$ is quasi-compact and quasi-separated, then the condition is equivalent to:
\[
  \ker\bigl(\Aut_{\cX}(x)\to \Aut_{\cY}(\pi(x))\bigr) \;\subseteq\;
  \ker\bigl(\Aut_{\cX}(x)\to \Aut_{\cZ}({f(x)})\bigr)
\]
for every closed point $x\in |\cX|$.
\end{theorem}
We begin with
the uniqueness, which also holds for adequate moduli space morphisms.

\begin{lemma}\label{L:adequate-is-epimorphism}
Let $\pi\colon \cX\to \cY$ be an adequate moduli space morphism between
algebraic stacks.  Then $\pi$ is a categorical epimorphism, that is,
if $f,g\colon \cY\to \cZ$ are two morphisms then every $2$-isomorphism
$f\circ\pi \simeq g\circ \pi$ descends to a unique $2$-isomorphism $f\simeq g$.
\end{lemma}
\begin{proof}
Two morphisms $f,g\colon \cY\to \cZ$ gives rise to a morphism $(f,g)\colon
\cY\to \cZ\times \cZ$. Let $I:=\Isom(f,g)=\cZ\times_{\Delta,\cZ\times \cZ,(f,g)}
\cY$.  Then $I\to \cY$ is representable and its sections correspond to
$2$-isomorphisms between $f$ and $g$. Similarly, a $2$-isomorphism between
$f\circ\pi$ and $g\circ \pi$ corresponds to a $\cY$-morphism $\cX\to I$. That
every $\cY$-morphism $\cX\to I$ descends to a unique $\cY$-morphism $\cY\to I$
can be checked smooth-locally on $\cY$ and thus follows directly from
\Cref{T:universal}.
\end{proof}

\begin{proof}[Proof of \Cref{T:universal-gms-variant}]
The uniqueness is \Cref{L:adequate-is-epimorphism}. For the existence, we may
work smooth-locally on $\cY$ and assume that $X=\cY$ is an affine scheme.  Let
$p\colon U\to \cZ$ be a smooth presentation where $U$ is an algebraic space.
This gives a smooth, representable and quasi-separated morphism
$q\colon U\times_\cZ \cX \to \cX$.
Let $x\in |\cX|$ be a point, closed in its fiber over $X$.  The assumption on
inertia shows that $\stG_x\to \cZ$ factors through the structure morphism
$\stG_x\to \Spec \kappa(x)$. It follows that $q|_{\stG_x}$ has a section after
passing to a separable field extension of $\kappa(x)$ and that can be
accomplished \'etale-locally on $X$. By
\Cref{P:extension-resprop,P:extension-morphisms}\itemref{PI:extension-morphisms:section}
we obtain
a section of $q$ after replacing $X$ with an \'etale neighborhood of $\pi(x)$.
We can thus factor $f\colon \cX\to \cZ$ through $p$ and hence also through
$\pi$ by \Cref{T:universal}.
\end{proof}

\Cref{T:universal-gms-variant} generalizes \cite[Thm.~2.3.6]{MR4172700} from
tame stacks to good moduli space morphisms. The analogous result for adequate
moduli space morphisms does not hold. In fact, the result is false even if
$\stX$ is a wild Deligne--Mumford stack and $\stX \to \stY$ is its coarse
moduli space \cite[A.2.3]{MR4172700}.

\section{Refinements of local structure}\label{S:refinements}
In \Cref{T:base}, we have seen that for an algebraic stack $\cX$ 
satisfying mild hypotheses
and a point $x \in |\cX|$ with linearly reductive stabilizer with 
image $s \in |S|$ such that $\kappa(x)/\kappa(s)$ is finite, there exists 
an \'etale morphism
$$(\cW, w) \to (\cX, x),$$
where $\cW = [\Spec A/\GL_n]$ is a fundamental stack, 
$w \in |\cW|$ is closed in its fiber $\cW_s$, 
and the induced map $\cG_w \to \cG_x$ on residual gerbes is an isomorphism. 

We now prove two theorems providing \'etale neighborhoods $\cW' \to \cW$ 
of a point $w$ of a fundamental stack $\cW$ such that 
$\cW' = [\Spec A'/G]$ and the group scheme $G$ has a specific form.  
When applied to the 
output of \Cref{T:base}, these theorems yield refinements of the local 
structure theorem.  In \Cref{T:etale-local-gms-connected}, $\cW' \to \cW$ 
is even finite \'etale and 
$G \to \Spec \ZZ$ is split reductive such that the stabilizer 
of the action of $G$ at a point $u \in \Spec A'$ over $w$
is the connected component 
$G_w^0$.  On the other hand, in \Cref{T:refinement}, if 
the residual gerbe $\cG_w = BG_0$ is neutral, 
then  $G \to S'$ is a geometrically reductive 
group scheme defined over an \'etale neighborhood $S' \to S$ and  
is a deformation of $G_0$.  Moreover, under mild characteristic 
hypothesis, $G \to S'$ is linearly reductive.  
When the gerbe $\cG_w$ is not neutral,
then $\cW'$ can be arranged to be affine over a fundamental gerbe
$\cH' \to S'$ which is a deformation of $\cG_w$.

\subsection{Split local structure of fundamental stacks}

\begin{theorem}[Split local structure]\label{T:etale-local-gms-connected}
  Let $S$ be a quasi-separated algebraic space. Let $\cW$ be a
  fundamental stack of finite presentation over $S$.  Let
  $w\in |\cW|$ be a point with linearly reductive stabilizer and
  image $s\in |S|$ such that $w$ is closed in its fiber $\cW_s$. Then there exists a finite \'etale morphism
  $f\colon \cW' \to \cW$ such that:
  \begin{enumerate}
  \item $\cW'=[U/G]$ where $U$ is affine and
  $G\to \Spec \ZZ$ is a geometrically reductive embeddable group scheme;
  \item \label{TI:etale-local-gms-connected:0} there is a point $u\in |U|$ above $w$ fixed by $G$
    and $f$ identifies $G_u$ with the connected component of $\Aut_\cW(w)$;
  \item \label{TI:etale-local-gms-connected:red} if $\kar(\kappa(w))=0$, then $G\to \Spec \ZZ$ is split
    reductive; and
  \item \label{TI:etale-local-gms-connected:diag}if $\kar(\kappa(w))=p$, then $G\to \Spec \ZZ$ is
    diagonalizable.
  \end{enumerate}
\end{theorem}

In other words, there is a commutative diagram of adequate
moduli spaces
\[
\xymatrix{
  \mathllap{[U/G] = \;}\cW' \ar[r]^-{\smash{f}} \ar[d]  & \cW\mathrlap{\; = [\Spec A/\GL_n]} \ar[d]^{\pi} \\
  \mathllap{U\gitq G = \;}W' \ar[r]   & W\mathrlap{\; = \Spec A \gitq \GL_n}
}
\]
where $f$ is finite \'etale.
Note that $W'\to W$ is finite but not necessarily \'etale,
and that the diagram is not necessarily cartesian.
If $\kar(\kappa(w))=0$, then $G\to \Spec \ZZ$ is smooth with geometrically
connected fibers. If $\kar(\kappa(w))=p$, then $G\to \Spec \ZZ$ need neither be
smooth nor have connected fibers, e.g., $G=\Gmu_{p,\ZZ}$.

\begin{proof}[Proof of \Cref{T:etale-local-gms-connected}]
  Using standard limit methods, we may replace the adequate moduli space $W$ of $\cW$ with its henselization at $\pi(w)$.  In this case,  $w \in |\cW|$ is the unique closed point and $\cW$ is \emph{linearly} fundamental by \Cref{C:adequate+lin-red=>good:fundamental}.

The structure morphism of the residual gerbe $\cG_w\to \Spec \kappa(w)$ is
smooth. Thus, there exists a separable field extension $k/\kappa(w)$ that
neutralizes the gerbe.  Let $G_w$ be the stabilizer group of a lift $\Spec k\to
\cG_w$. Let $(G_w)^0$ be its connected component. If $\kar k>0$, then $(G_w)^0$
is of multiplicative type and if $\kar k=0$, then $(G_w)^0$ is smooth,
connected and reductive. In either case, we may thus find a further separable
field extension $k'/k$ such that $(G_w)^0$ becomes diagonalizable or a split
reductive group. That is, $(G_w^0)_{k'}=G_{k'}$ where $G$ is a group scheme
over $\ZZ$ which is either diagonalizable or split reductive. We now have a
group $G$ as in \itemref{TI:etale-local-gms-connected:red} or
\eqref{TI:etale-local-gms-connected:diag}.

We now apply \Cref{P:deformation-morphisms}\itemref{PI:deformation-morphisms:fet} to the finite \'etale morphism $\cW'_0:=B(G_w^0)_{k'} \to \cG_w$ and obtain a finite \'etale morphism $\cW' \to \cW$. Then $(\cW',\cW'_0)$ is a local linearly fundamental pair by construction
and its good moduli space $W'$ is finite over $W$ by
\Cref{T:etale-local-gms}\itemref{TI:etale-local-gms:fp2}, hence henselian.
By \Cref{P:deformation-morphisms} and \Cref{P:refinement}\itemref{P:refinement:affine_diag}, we may extend the affine morphism $\cW'_0 \to BG$ to an affine
morphism $\cW' \to BG$. That is, $\cW'=[U/G]$ for some affine
scheme $U$ and the unique point $u\in |U|$ above $\cW'_0$ is fixed by $G$.
\end{proof}

\subsection{Refinements on the local structure theorem}\label{SS:main-theorem-refinements}

\begin{theorem}[Local structure refinement]\label{T:refinement}
  Let $S$ be a quasi-separated algebraic space. Let $\cW$ be a
  fundamental stack of finite presentation over $S$.  Let
  $w\in |\cW|$ be a point with linearly reductive stabilizer and
  image $s\in |S|$ such that $w$ is closed in its fiber $\cW_s$.
  Then there exist a commutative diagram of algebraic stacks
  \[
    \xymatrix{%
    \mathllap{[\Spec A'/G] = \;}\cW' \ar[rrr]^{\smash{h}} \ar[d]_t &&& \cW \mathrlap{\; = [\Spec A/\GL_n]} \ar[d] \\
    \cH \ar[r]^-r & BG\ar[r] & S' \ar[r]^g	& S
    }%
    \vspace{-2mm}
  \]
  and a point $w' \in |\cW'|$ over $w \in |\cW|$ with image $s' \in |S'|$ where
  \begin{enumerate}
  \item  \label{T:refinement-h}
    $h \co (\cW',w') \to (\cW,w)$ is a strongly \'etale (see \Cref{D:strongly-etale}) neighborhood of
      $w$ such that $\cG_{w'} \to \cG_w$ is an isomorphism; 
  \item \label{T:refinement-g}
    $g \co (S',s')\to (S,s)$ is a smooth (\'etale if $\kappa(w)/\kappa(s)$
    is separable) morphism such that there is a $\kappa(s)$-isomorphism $\kappa(w)\cong \kappa(s')$;
  \item \label{T:refinement-G}
    $G\to S'$ is a geometrically reductive embeddable group scheme; 
  \item \label{T:refinement-cH}
    $\stk{H}\to S'$ is a gerbe such that $\stk{H}$ is fundamental and
    $\stk{H}_{s'}\cong \cG_w$;
  and
  \item \label{T:refinement-t}
    $t \co \cW' \to \cH$ and $r\co \cH \to BG$ are affine morphisms, so
    $\cW'=[\Spec A' / G]$.
  \end{enumerate}
  Moreover, we can arrange so that:
  \begin{enumerate}[resume]
    \item \label{T:refinement-neutral}
    if $\cG_w$ is neutral, then $\cH \to BG$ is an isomorphism;
    \item \label{T:refinement-lin-red-group} 
    if $s$ has an open
    neighborhood of characteristic zero, then
    $\stk{H}$ is linearly fundamental and $G$ is linearly reductive;
    \item \label{T:refinement-nice} 
      if $w$ has nice stabilizer (e.g., $\kar \kappa(s)>0$), then $\stk{H}$ is
      nicely fundamental and $G$ is nice; and
    \item \label{T:refinement-smooth} 
      if $\cW\to S$ is smooth at $w$ and $\kappa(w)/\kappa(s)$ is separable, then there exists a commutative diagram
    \[
    \xymatrix{%
    \VV(\cN_\sigma)\ar[dr] & \cW' \ar[l]_-{\smash{q}} \ar[d]_t \\
    & \cH\ar@/_/[u]_\sigma \ar@/^/[ul]^0
    }%
    \]
    where $q$ is strongly \'etale and $\sigma$ is a section of $t$ such that
  $\sigma(s')=w'$.
  \end{enumerate}
\end{theorem}

\begin{proof}
  We can replace $(S,s)$ with an \'etale neighborhood and assume that $S$
  is an affine scheme. For \itemref{T:refinement-g}, we note that we have a tower of
  finite extensions $\kappa(s) \subseteq \kappa(w)^s \subseteq \kappa(w)$, where the first extension
  is separable and the latter is purely inseparable. By \cite[IV.18.1.1]{EGA}, there is an \'etale
  neighborhood $(S',s') \to (S,s)$ such that $\kappa(s') \cong \kappa(w)^s$ over $\kappa(s)$. Hence,
  we are reduced to the situation where $\kappa(s) \subseteq \kappa(w)$ is purely inseparable; then
  $\kappa(w) = \kappa(s)[a_1^{p^{-r_1}},\dots,a_n^{p^{-r_n}}]$, where $p$ is the characteristic of
  $\kappa(s)$ and the $a_i \in \kappa(s)$. There is a 
  point $s'$ of $S'=\AA^n_{S}$ such that $\kappa(s')=\kappa(w)$
  and $(S',s') \to (S,s)$ is a smooth morphism as desired.
  In the sequel, we will repeatedly replace $(S',s')$ by further \'etale
  neighborhoods but without residue field extensions.

  For \itemref{T:refinement-cH}, we can by \Cref{P:extension-gerbes} (and
  \Cref{R:extension-non-closed-points}) extend the residual gerbe $\cG_w$ over
  $\kappa(w)=\kappa(s')$ to a fundamental gerbe $\cH\to S'$ after replacing
  $S'$ by an \'etale neighborhood of $s'$. If $s'$ has an open neighborhood of
  characteristic zero, then $\cH$ is linearly fundamental after restricting to
  this neighborhood.  If $w$ has nice stabilizer, then $\cH$ is nicely
  fundamental after replacing $S'$ with an \'etale neighborhood by
  \Cref{P:extension-nicefund}.

  If $\cG_w$ is neutral, that is, has a section $\sigma_0$, then after
  replacing $S'$ with an \'etale neighborhood, we obtain a section $\sigma$ of
  $\cH\to S'$
  (\Cref{P:extension-morphisms}\itemref{PI:extension-morphisms:section} with
  $\cX=S'$) and then $\cH=BG$ where $G=\Aut(\sigma)$. This gives
  \itemref{T:refinement-neutral}.

  If $\cG_w$ is not neutral, there exists, after replacing $S'$ with an \'etale
  neighborhood of $s'$, a finite \'etale surjective morphism $S''\to S'$ such
  that $\cH\times_{S'} S''\to S''$ has a section $\sigma'$.
  The group scheme $H'=\Aut(\sigma')\to S''$ is geometrically reductive and
  embeddable. We let $G$ be the Weil restriction of $H'$ along $S''\to S'$.
  It comes equipped with a morphism $\cH\to BG$ which is representable,
  hence affine by \cite[Cor.~4.3.2]{alper-adequate}.
  It can be seen that $G\to S'$ is geometrically reductive and embeddable and
  also linearly reductive (resp.\ nice) if $\cH$ is linearly fundamental
  (resp.\ nicely fundamental). This establishes \itemref{T:refinement-G},
  \itemref{T:refinement-lin-red-group} and \itemref{T:refinement-nice}.

  Since $\cH \to S'\to S$ is smooth, we may apply \Cref{P:extension-morphisms}
  to obtain a strongly \'etale neighborhood $h\co (\cW',w') \to (\cW,w)$ and an
  affine morphism $t\co \cW' \to \cH$. This establishes
  \itemref{T:refinement-h} and \itemref{T:refinement-t}.

  Finally, for \itemref{T:refinement-smooth}, if $\cW\to S$
  is smooth at $w$ and $\kappa(w)/\kappa(s)$ is separable, then $S'\to S$ is
  \'etale, so $\cW'\to S'$ is also smooth at $w'$. The result now follows
  from \Cref{P:smooth-refinement} below.
\end{proof}

\begin{proposition}[Smooth refinement]\label{P:smooth-refinement}
Let $S$ be a quasi-separated algebraic space. Let $\stk{H}\to S$ be an fppf gerbe such that $\stk{H}$ is fundamental and let $t\co\stk{W}\to \stk{H}$ be an affine morphism of finite presentation. 
Let $w\in |\stk{W}|$ be a point with
linearly reductive stabilizer and image $s\in |S|$ such that $w$ is closed in
its fiber $\stk{W}_s$.  Suppose that the induced map $\cG_w\to
\stk{H}_s$ is an isomorphism. If $\stk{W}\to S$ is smooth at $w$, then after replacing $S$ with an
\'etale neighborhood of $s$ (without residue field extension), there exist
\begin{enumerate}
\item a section $\sigma\co \stk{H}\to \stk{W}$ of $t$ such that
$\sigma(s)=w$; and
\item a morphism $q\co \stk{W}\to \VV(\cN_\sigma)$, where $\cN_{\sigma} = t_* (\cI/\cI^2)$ and $\cI$ is the sheaf of ideals in $\stk{W}$ defining $\sigma$, 
 which is strongly \'etale in
 a saturated open neighborhood of $\sigma$
  and such that $q\circ \sigma$ is the zero-section.
\end{enumerate}
\end{proposition}

\begin{proof}
Since $\stk{H}\to S$ is a gerbe, $\stk{W}\to S$ is smooth at $w$ and
$\cG_w\cong \stk{H}_s$, it follows that $t\colon \stk{W}\to \stk{H}$ is
smooth at $w$.
The existence of the section $\sigma$ thus follows from
\Cref{P:extension-morphisms}\itemref{PI:extension-morphisms:section}.
Note that since $t$ is affine and smooth, the
section $\sigma$ is a regular closed immersion.
An easy approximation argument allows us to replace $S$ by the henselization at
$s$. Then $\stk{H}$ is linearly fundamental
(\Cref{C:adequate+lin-red=>good:fundamental}). Let $\cI\subset \Orb_\cW$ be
the ideal sheaf defining $\sigma$. Since $\cN_\sigma=t_*(\cI/\cI^2)$ is locally free
and $\stk{H}$ is cohomologically affine, the surjection $t_*\cI\to \cN_\sigma$ of
$\Orb_{\stk{H}}$-modules admits a section. The composition $\cN_\sigma\to
t_*\cI\to t_*\Orb_\cW$ gives a morphism $q\co \cW\to \VV(\cN_\sigma)$. By
definition, $q$ maps $\sigma$ to the zero-section and induces an isomorphism of
normal spaces along $\sigma$, hence is \'etale along $\sigma$, hence
is strongly \'etale in a neighborhood by Luna's fundamental lemma
(\Cref{L:fundamental-lemma}).
\end{proof}

\section{The structure of linearly reductive groups} \label{S:structure}
Recall from \Cref{D:reductive/nice} that a linearly reductive (resp.\ geometrically reductive)
group scheme $G\to S$ is flat, affine and of finite presentation such that
$BG\to S$ is a good moduli space (resp.\ an adequate moduli space). In this
section we will show that a group algebraic space is linearly
reductive if and only if it is flat, separated, of finite presentation, has
linearly reductive fibers, and has
a finite component group (\Cref{T:char-of-lin-reductive-groups}).

\subsection{Extension of closed subgroups}
\begin{lemma}[Anantharaman]\label{L:group-over-DVR}
Let $S$ be the spectrum of a DVR. If $G\to S$ is a separated group algebraic space of finite
type, then $G$ is a scheme. If in addition $G\to S$ has affine fibers or
is flat with affine generic fiber $G_\eta$, then $G$ is affine.
\end{lemma}
\begin{proof}
The first statement is \cite[Thm.~4.B]{MR0335524}. For the second statement,
it is enough to show that the flat group scheme $\overline{G}=\overline{G_\eta}$
is affine. This is \cite[Prop.~2.3.1]{MR0335524}.
\end{proof}

\begin{proposition}\label{P:geomreductive-mono}
Let $H\to S$ be a geometrically reductive group scheme that is embeddable
fppf-locally on $S$, e.g., $H$ linearly
reductive (\Cref{C:lin-red-groups}).
\begin{enumerate}
\item \label{P:geomreductive:closed-qf-normal-are-finite}
  If $N\subset H$ is a closed normal subgroup such that $N\to S$ is
  quasi-finite, then $N\to S$ is finite.
\end{enumerate}
Let $G\to S$ be a separated group algebraic space of finite presentation
and let $u\colon H\to G$ be a homomorphism.
\begin{enumerate}[resume]
\item \label{P:geomreductive:mono-closed}
  If $u$ is a monomorphism, then $u$ is a closed immersion.
\item \label{P:geomreductive:closed-imm}
  If $u_s\colon H_s\to G_s$ is a monomorphism for a point $s\in S$, then
  $u_U\colon H_U\to G_U$ is a closed immersion for some open neighborhood
  $U$ of $s$.
\end{enumerate}
\end{proposition}
\begin{proof}
  The questions are local on $S$ so we can assume that $H$ is embeddable.
  For \itemref{P:geomreductive:closed-qf-normal-are-finite} we note that
  a normal closed subgroup $N\subset H$ gives rise to a closed subgroup $[N/H]$
  of the inertia stack $[H/H]=I_{BH}$ (where $H$ acts on itself via conjugation). 
  The result thus follows from
  \Cref{L:fundamental:quasi-finite-subgroup-of-inertia}.

  For \itemref{P:geomreductive:mono-closed}, it is
  enough to prove that $u$ is proper. After noetherian approximation, we can
  assume that $S$ is noetherian. By the valuative criterion for properness, we
  can further assume that $S$ is the spectrum of a DVR. We can also replace
  $G$ with the closure of $u(H_\eta)$. Then $G$ is an
  affine group scheme (\Cref{L:group-over-DVR}) so $G/H\to BH\to S$ is adequately
  affine.  As $G/H \to S$ is also representable, it follows from \cite[Thm.~4.3.1]{alper-adequate} that it is also affine. It follows that $u$ is a closed immersion.

  For \itemref{P:geomreductive:closed-imm}, we apply
  \itemref{P:geomreductive:closed-qf-normal-are-finite} to $\ker(u)$ which is
  quasi-finite, hence finite, in an open neighborhood of $s$. By Nakayama's lemma
  $u$ is thus a monomorphism in an open neighborhood and we conclude by
  \itemref{P:geomreductive:mono-closed}.
\end{proof}

\begin{remark}
  If $G\to S$ is flat, then \itemref{P:geomreductive:mono-closed} says
  that any representable morphism $BH\to BG$ is separated.
  When $H$ is of multiplicative type then \Cref{P:geomreductive-mono}
  is~\cite[Exp.~IX, Thm.~6.4 and Exp.~VIII, Rmq.~7.13b]{MR0274459}. When $H$ is
  reductive (i.e., smooth with connected reductive fibers) it is
  \cite[Exp.~XVI, Prop.~6.1 and Cor.~1.5a]{MR0274459}.
\end{remark}

\begin{proposition}\label{P:deformations-of-subgroups}
Let $(S,s)$ be a henselian local ring, let $G\to S$ be a flat group algebraic space of
finite presentation with affine fibers and let $u_s\co H_s\inj G_s$ be a closed
subgroup. If $H_s$ is linearly reductive and $G_s/H_s$ is smooth, then there exists a linearly reductive and embeddable group scheme $H\to S$ and a homomorphism
$u\co H\to G$ extending $u_s$.
\begin{enumerate}
\item\label{PI:deformation-of-subgroups:et}
  If $G_s/H_s$ is \'etale (i.e., if $u_s$ is open and closed),
  then $u$ is \'etale and the pair $(H,u)$ is unique: if $H'$ is a
  quasi-separated group algebraic space and $u'\colon H'\to G$ is an \'etale
  homomorphism extending $u_s$, then there exists a unique homomorphism $H\to
  H'$ over $G$.
\item\label{PI:deformation-of-subgroups:sep}
  If $G\to S$ is separated, then $u$ is a closed immersion.
\end{enumerate}
\end{proposition}
\begin{proof}
Note that since $(S,s)$ is local, condition \ref{Cond:FC} is satisfied.
By \Cref{P:deformation-linearly-fundamental}, the gerbe $BH_s$ extends
to a unique linearly fundamental gerbe $\cH\to S$.
Since $BG\to S$ is smooth, we can extend the morphism $\varphi_s\co BH_s\to
BG_s$ to a morphism $\varphi\co \cH\to BG$ (\Cref{P:deformation-morphisms}).
The morphism $\varphi$ is flat and the special
fiber $\varphi_s$ is smooth since $G_s/H_s$ is smooth. Thus $\varphi$ is
smooth. Similarly, if $G_s/H_s$ is \'etale, then $\varphi$ is \'etale.

Let $p\co S\to BG$ and $f_s\co \Spec \kappa(s)\to BH_s\cong \cH_s$ denote the
tautological sections. Since $\varphi$ is smooth and $\varphi_s \circ f_s=p_s$,
we obtain a section $f\co S\to \cH$ such that $f|_s=f_s$ and $\varphi\circ f=p$
by \Cref{P:deformation-sections} (applied to $\stX=S$ and $\stX'=\cH\times_{BG}
S$).  We let $H=\Aut(f)$ and let $u\co H\to G=\Aut(\varphi\circ f)$ be the
induced morphism, extending $u_s$.

For \itemref{PI:deformation-of-subgroups:et}, we have already seen that if
$G_s/H_s$ is \'etale, then $\varphi$ is \'etale so that $u$ is \'etale. Let
$u'\co H'\to G$ be another \'etale homomorphism extending $u_s$. We have a
unique map $\psi\co BH\to BH'$ over $BG$ extending the isomorphism $BH_s\to
BH'_s$ (apply \Cref{P:deformation-sections} to $\stX=BH$ and
$\stX'=BH'\times_{BG} BH$). The tautological section of $BH$ is then mapped
by $\psi$ to the tautological section of $BH'$ (apply
\Cref{P:deformation-sections} to $\stX=S$ and $\stX'=BH'\times_{BG} S$). That
is, $\psi$ is induced by a unique homomorphism $H\to H'$ over $G$.

For \itemref{PI:deformation-of-subgroups:sep}, if $G$ is separated,
then $u$ is a closed immersion by \Cref{P:geomreductive-mono}.
\end{proof}

\begin{remark}
If $G_s/H_s$ is not smooth, then the tautological section of $BH_s$
still extends to a section of $\cH\to S$ so $\cH=BH$ where
$H$ is an extension of $H_s$. But $\varphi$ merely induces a
homomorphism $H\to \widetilde{G}$ where $\widetilde{G}$ is a
twisted form of $G$.
\end{remark}

\subsection{The smooth identity component of linearly reductive groups}\label{S:sm-id-cmpt-lr}
Recall that if $G\to S$ is a smooth group scheme, then there is an open
characteristic
subgroup $G^0\subseteq G$ such that $G^0\to S$ is smooth with connected
fibers~\cite[Exp.~6B, Thm.~3.10]{MR0274459}. This is also true when $G\to S$ is
a smooth group algebraic space, cf~\cite[6.8]{lmb}. The subgroup $G^0$ is
not always closed, not even when $G$ is affine~\cite[VII, \S3]{MR0260758}.
For a (not necessarily smooth) group scheme of finite type over a field, the
identity component $G^0$ exists and is open and closed.

When $(S,s)$ is henselian and $(G_s)^0$ is linearly reductive
but not necessarily smooth, then
\Cref{P:deformations-of-subgroups} gives the existence of a unique
linearly reductive group scheme $G^0_\loc$ together with an \'etale
homomorphism $u\co G^0_\loc\to G$ extending $u_s\co (G_s)^0\inj G_s$. There are
at least three subtleties:
\begin{enumerate}[label=(\alph*), ref=\alph*]
  \item\label{I:u-not-inj} $u$ need not be injective, even if $G$ is smooth and of
    characteristic zero.
  \item\label{I:G^0_loc-not-cn} $G^0_\loc$ need not have connected fibers.
  \item\label{I:G^0_loc-not-normal} Even if $u$ is injective, $G^0_\loc$ need not be a normal subgroup.
\end{enumerate}
As we will see, the first problem only happens when $G$ is not separated and
the latter two only in mixed characteristic when $G$ is not smooth.

\begin{example}
We give two examples in mixed characteristic and one in equal characteristic:
\begin{enumerate}
\item Let $G=\Gmu_{p,\ZZ_p}\to \Spec \ZZ_p$ which is a finite linearly
  reductive group scheme. Then $G^0_\loc=G$ but the generic geometric fiber is
  not connected, illustrating~\itemref{I:G^0_loc-not-cn}. If we let $G'$ be the
  gluing of $G$ and a finite group over $\QQ_p$ containing $\Gmu_p$ as a
  non-normal subgroup, then $G'^0_\loc=G^0_\loc\subseteq G'$ is not normal,
  illustrating~\itemref{I:G^0_loc-not-normal}.
\item Let $G$ be as in the previous example and consider the \'etale group
  scheme $H\to \Spec \ZZ_p$ given as extension by zero from $\Gmu_{p,\QQ_p}\to
  \Spec \QQ_p$. Then we have a bijective monomorphism $H\to G$ which is not an
  immersion and $G'=G/H$ is a quasi-finite group algebraic space with connected
  fibers which is not locally separated. Note that
  $(G')^0_\loc=G=\Gmu_{p,\ZZ_p}$ and the \'etale morphism $(G')^0_\loc\to G'$
  is not injective, illustrating~\itemref{I:u-not-inj}.
\item Let $G=\Gm\times S\to S=\Spec k\llbracket t \rrbracket$ and let $H\to S$
  be $\Gmu_{r,k(\!(t)\!)}$ extended by zero for some invertible $r>1$.  Let
  $G'=G/H$. Then $G'$ is a smooth locally separated algebraic space,
  $G'^0_\loc=G$ and $G'^0_\loc\to G'$ is not injective,
  illustrating~\itemref{I:u-not-inj}.
\end{enumerate}
\end{example}

From now on, we only consider separated group schemes. Then $G^0_\loc\to G$ is
an open and closed subgroup by \Cref{P:deformations-of-subgroups}
so~\itemref{I:u-not-inj} does not occur.
The subgroup $G^0_\loc$ exists over the henselization but not globally in mixed
characteristic due to problem~\itemref{I:G^0_loc-not-cn}. We remedy this by
considering a slightly smaller subgroup $G^0_\sm$ which
is closed but not open.

\begin{lemma}[Identity component: nice case]\label{L:existence-of-G^0_sm}
Let $S$ be an algebraic space and let $G\to S$ be a flat and separated group algebraic
space of finite presentation with affine fibers.
\begin{enumerate}
\item\label{LI:connected-nice-open} The locus of $s\in S$ such that $(G_s)^0$ is nice is open in $S$.
\end{enumerate}
Now assume that $(G_s)^0$ is nice for all $s\in S$.
\begin{enumerate}[resume]
\item There exist a unique characteristic  closed subgroup $G^0_\sm\inj G$ smooth over $S$
  that restricts to $(G_s)^0_\red$ on fibers.
\item $G^0_\sm \to S$ is a torus, 
  $G / G^0_\sm \to S$ is quasi-finite and separated, and $G \to S$ is quasi-affine.
\end{enumerate}
Now assume in addition that $S$ has equal characteristic.
\begin{enumerate}[resume]
\item There exist a unique characteristic open and closed subgroup $G^0\inj G$
  that restricts to $(G_s)^0$ on fibers.
\item $G^0 \to S$ is of multiplicative type with connected fibers and $G / G^0 \to S$ is
  \'etale and separated.
\end{enumerate}
\end{lemma}
\begin{proof}
The questions are \'etale-local on $S$.  For \itemref{LI:connected-nice-open},
if $(G_s)^0$ is nice, i.e.,
of multiplicative type, then over the henselization at $s$ we can find
an open and closed subgroup $G^0_\loc\subseteq G$ such that $G^0_\loc$ is of
multiplicative type (\Cref{P:deformations-of-subgroups}). After replacing $S$
with an \'etale neighborhood of $s$, we can thus find an open and closed subgroup
$H\subseteq G$ where $H$ is of multiplicative type. It follows
that $(G_s)^0$ is of multiplicative type for all $s$ in $S$.

For an $H$ as above, we have a characteristic closed subgroup $H_\sm\inj H$ such
that $H_\sm$ is a torus and $H/H_\sm$ is finite. Indeed, the Cartier dual of
$H$ is an \'etale sheaf of abelian groups and its torsion is a characteristic
subgroup. It follows that $G/H_\sm$ is quasi-finite and separated and that $G$
is quasi-affine.

It remains to prove that $H_\sm$ is characteristic and independent on the
choice of $H$ so that it glues to a characteristic subgroup $G^0_\sm$. This can
be checked after base change to henselian local schemes.  If $(S,s)$ is
henselian, then $G^0_\loc\subseteq H$ and since these are group schemes of
multiplicative type of the same dimension, it follows that
$(G^0_\loc)_\sm=H_\sm$. This shows that $H_\sm$ is independent on the choice
of $H$.
Since any automorphism of $G$ leaves $G^0_\loc$ fixed, any automorphism leaves
$(G^0_\loc)_\sm$ fixed as well. This shows that $H_\sm$ is characteristic.

If $S$ has equal characteristic, then $H$ is an open and closed subgroup with
connected fibers, hence clearly unique.
\end{proof}

\begin{lemma}[Identity component: smooth case]\label{L:smooth-lin-red}
Let $S$ be an algebraic space and let $G\to S$ be a flat and separated group algebraic
space of finite presentation with affine fibers. Suppose that $G\to S$ is
smooth and that $(G_s)^0$ is linearly reductive for all $s$.
\begin{enumerate}
\item The open subgroup $G^0\subseteq G$ is also closed and linearly reductive (and in particular affine).
\item $G / G^0 \to S$ is \'etale and separated and $G \to S$ is quasi-affine.
\end{enumerate}
\end{lemma}
\begin{proof}
This follows immediately from \Cref{P:deformations-of-subgroups} since in the
henselian case $G^0_\loc$ is the smallest open subscheme containing $(G_s)^0$,
hence equal to $G^0$.
\end{proof}

\begin{theorem}[Identity component]\label{T:char-of-lin-reductive-groups}
Let $S$ be an algebraic space and let $G\to S$ be a flat and separated group algebraic
space of finite presentation with affine fibers. Suppose that $(G_s)^0$ is
linearly reductive for every $s\in S$.
\begin{enumerate}
\item\label{TI:lin-red:smooth-identity} There exist a unique  linearly reductive and 
characteristic closed subgroup $G^0_\sm\inj G$
smooth over $S$
  that restricts to $(G_s)^0_\red$ on fibers, and $G / G^0_\sm \to S$ is quasi-finite
  and separated.
\item If $S$ is of equal characteristic, then there exists a unique linearly
  reductive characteristic open and closed subgroup $G^0\inj G$ that restricts
  to $(G_s)^0$ on fibers, and $G / G^0 \to S$ is \'etale and separated.
\item $G \to S$ is quasi-affine.
\end{enumerate}
The following are equivalent:
\begin{enumerate}[resume]
\item\label{TI:lin-red:lin-red} $G \to S$ is linearly reductive (in particular affine).
\item\label{TI:lin-red:fin-tame} $G/G^0_\sm \to S$ is finite and tame.
\item (if $S$ of equal characteristic) $G/G^0 \to S$ is finite and tame.
\end{enumerate}
In particular, if $G \to S$ is linearly reductive and $S$ is of equal characteristic
$p>0$, then $G \to S$ is nice.
\end{theorem}
\begin{proof}
Let $S_1\subseteq S$ be the open locus where $(G_s)^0$ is nice and let
$S_2\subseteq S$ be the open locus where $G_s$ is smooth. Then $S=S_1\cup S_2$.
Over $S_1$, we define $G^0_\sm$ as in \Cref{L:existence-of-G^0_sm}.  Over
$S_2$, we define $G^0_\sm=G^0$ as in \Cref{L:smooth-lin-red}. In equal
characteristic, we define $G^0$ as in
\Cref{L:existence-of-G^0_sm,L:smooth-lin-red}. The first three
statements follow.

Since $G^0_\sm \to S$ is linearly reductive, it follows that
$BG\to S$ is cohomologically affine if and only if $B(G/G^0_\sm)\to S$ is
cohomologically affine \cite[Prop.~12.17]{alper-good}.
If $B(G/G^0_\sm)\to S$ is cohomologically affine, then $G/G^0_\sm$ is finite
\cite[Thm.~8.3.2]{alper-adequate}. Conversely, if $G/G^0_\sm$ is finite and tame
then $BG\to S$ is cohomologically affine and $G\to S$ is affine.
\end{proof}

\begin{corollary}\label{C:lin-red-embeddable-over-normal}
If $S$ is a normal noetherian scheme with the resolution property (e.g., $S$ is
regular and separated, or $S$ is quasi-projective) and $G\to S$ is linearly reductive,
then $G$ is embeddable.
\end{corollary}
\begin{proof}
The stack $BG^0_\sm$ has the resolution property~\cite[Cor.~3.2]{thomason}.
Since $BG^0_\sm\to BG$ is finite and faithfully flat, it follows that $BG$
has the resolution property~\cite{gross-resolution}, hence that $G$ is embeddable.
\end{proof}

\begin{corollary}
Let $S$ be an algebraic space and let $G\to S$ be a flat and separated group
algebraic space of finite presentation with affine, connected and linearly reductive
fibers. Then $G\to S$ is affine and linearly reductive.
\end{corollary}
\begin{proof}
By \Cref{T:char-of-lin-reductive-groups}\itemref{TI:lin-red:smooth-identity},
there is a closed (characteristic) smooth linearly reductive group scheme
$G^0_\sm$ and $G/G^0_\sm\to S$ is a quasi-finite flat separated morphism with
connected fibers since $G\to S$ has connected fibers. Such a morphism is
finite, as for example can be seen by passing to henselizations. Moreover,
$G/G^0_\sm\to S$ is tame since this can be checked on fibers and $G_s$ is
linearly reductive.  Thus $G\to S$ is linearly reductive by the equivalence
of \itemref{TI:lin-red:lin-red} and \itemref{TI:lin-red:fin-tame} in
\Cref{T:char-of-lin-reductive-groups}.
\end{proof}

\begin{remark}
Let $G\to S$ be as in \Cref{T:char-of-lin-reductive-groups}.
When $G/G^0_\sm$ is merely finite, then $G\to S$ is geometrically reductive.
This happens precisely when $G\to S$ is \emph{pure} in the sense of Raynaud--Gruson~\cite[D\'efn.~3.3.3]{MR0308104}. In particular, $G\to S$ is geometrically
reductive if and only if $\pi\co G\to S$ is affine and $\pi_*\Orb_G$ is a locally
projective $\Orb_S$-module~\cite[Thm.~3.3.5]{MR0308104}.
\end{remark}

\section{Further applications} \label{S:further-applications}
In this section we give generalizations of Sumihiro's theorem and Luna's \'etale
slice theorem in equivariant geometry. These are obtained by applying
the local structure theorem to $\stX = [X/G]$. We also show that
the henselization $\cX^h_x$ exists if $\cX$ has affine stabilizers and $x$ is
a closed point with linearly reductive stabilizer.
Finally, we deduce that several stacks, including $\Cohstk_X(\stX)$,
$\Hilbshf_{\stX/X}$ and $\Homstk_X(\stX,\stY)$, are algebraic if $\stX\to X$ is
a good moduli space.

\subsection{Generalization of Sumihiro's theorem on torus actions}\label{SS:sumihiro}
Sumihiro's theorem on torus actions in the relative case is the following.  
Let $S$ be a noetherian scheme and $X \to S$ a morphism of schemes satisfying Sumihiro's condition \ref{Cond:N}, that is, $X \to S$ is flat and of finite type, $X_s$ is geometrically normal for all generic points $s \in S$ and $X_s$ is geometrically integral for all codimension 1 points $s \in S$ (which by a result of Raynaud implies that $X$ is normal); see \cite[Defn.~3.4 and Rem.~3.5]{sumihiro2}.  If $S$ is normal and $T \to S$ is a smooth and Zariski-locally diagonalizable group scheme acting on $X$ over $S$, then 
 there exists a $T$-equivariant affine open neighborhood of any point of $X$  \cite[Cor.~3.11]{sumihiro2}.
We provide the following generalization of this result which simultaneously generalizes \cite[Thm.~4.4]{luna-field} to the relative case.

\begin{theorem} \label{T:sumihiro}
Let $S$ be a quasi-separated algebraic space.  Let $G$ be an affine and flat group scheme over $S$ of finite presentation.  Let $X$ be a quasi-separated algebraic space locally of finite presentation over $S$ with an action of $G$.  Let $x \in X$ be a point with image $s\in S$ such that $\kappa(x)/\kappa(s)$ is finite.  Assume that $x$ has linearly reductive stabilizer.  Then there exists a $G$-equivariant \'etale neighborhood $(\Spec A, w) \to (X,x)$ that induces an isomorphism of residue fields and stabilizer groups at $w$.
\end{theorem}

\begin{proof}
By applying \Cref{T:base} to $\cX = [X/G]$ with $\cW_0 = \cG_x$ (the residual gerbe of $x$), we obtain an \'etale morphism $h \co (\cW,w)\to (\cX,x)$ with $\cW$ fundamental and $h|_{\cG_x}$ an isomorphism. By applying \Cref{P:refinement}\itemref{P:refinement:affine_diag} to the composition $\stk{W} \to \cX \to BG$, we may shrink $\stk{W}$ around $w$ so that $\stk{W} \to BG$ is affine.  It follows that $W:=\cW \times_{\cX} X$ is affine and $W \to X$ is $G$-equivariant. If we let $w' \in W$ be the unique preimage of $x$, then
$(W,w')\to (X,x)$ is the desired \'etale neighborhood.
\end{proof}

\begin{corollary} \label{C:sumihiro-torus}
Let $S$ be a quasi-separated algebraic space, $T \to S$ be a group scheme of multiplicative type over $S$ (e.g., a torus), and $X$ be a quasi-separated algebraic space locally of finite presentation over $S$ with an action of $T$.  Let $x \in X$ be a point with image $s\in S$ such that $\kappa(x)/\kappa(s)$ is finite. Then there exists a $T$-equivariant \'etale morphism $(\Spec A, w) \to (X,x)$ that induces an isomorphism of residue fields and stabilizer groups at $w$.
\end{corollary}

\begin{proof} This follows immediately from \Cref{T:sumihiro} as any subgroup of a fiber of $T \to S$ is linearly reductive.
\end{proof}

\begin{remark}
In \cite{brion-linearization}, Brion establishes several powerful structure results for actions of connected algebraic groups on varieties.  In particular, \cite[Thm.~4.8]{brion-linearization}  recovers the result above when $S$ is the spectrum of a field, $T$ is a torus and $X$ is quasi-projective without the final conclusion regarding residue fields and stabilizer groups.
\end{remark}

\subsection{Relative version of Luna's \'etale slice theorem} \label{SS:luna}

We provide the following generalization of Luna's \'etale slice theorem \cite{luna} (see also \cite[Thm.~4.5]{luna-field}) to the relative case.

\begin{theorem} \label{T:luna}
Let $S$ be a quasi-separated algebraic space.  Let $G \to S$ be a smooth, affine group scheme.
Let $X$ be a quasi-separated algebraic space locally of finite presentation over $S$ with an action of $G$.  Let $x \in X$ be a point with image $s \in S$ such that $\kappa(x)/\kappa(s)$ is a finite separable extension.  Assume that $x$ has linearly reductive stabilizer $G_x$. Then there exist
\begin{enumerate}
	\item \label{T:luna1} an \'etale morphism $(S',s') \to (S,s)$ and a $\kappa(s)$-isomorphism $\kappa(s') \cong \kappa(x)$;
	\item  \label{T:luna2} a geometrically reductive (linearly reductive if $\kar \kappa(s)>0$ or $s$ has an open neighborhood of characteristic zero) closed subgroup $H \subset G' := G \times_S S'$ over $S'$ such that $H_{s'} \cong G_x$; and
	\item  \label{T:luna3} an unramified $H$-equivariant $S'$-morphism $(W, w) \to (X',x')$ of finite presentation with $W$ affine and $\kappa(w) \cong \kappa(x')$ such that
	$W \times^H G' \to X'$ is \'etale.  Here $x' \in X':=X\times_S S'$ is the unique $\kappa(x)$-point over $x \in X$ and $s' \in S'$.
\end{enumerate}
Moreover, it can be arranged that
\begin{enumerate}
\setcounter{enumi}{3}
\item  \label{T:luna4} if $X \to S$ is smooth at $x$, then $W\to S'$ is smooth and there exists an $H$-equivariant section $\sigma \co S' \to W$ such that $\sigma(s')=w$, and 
there exists a strongly \'etale $H$-equivariant morphism $W \to \VV(\cN_\sigma)$;
\item  \label{T:luna5} if $X$ admits an adequate GIT quotient by $G$ (e.g., $X$ is affine over $S$ and $G$ is geometrically reductive over $S$), and $Gx$ is closed in $X_s$, then $W \times^H G' \to X'$ is strongly \'etale; and
\item  \label{T:luna6} if, \'etale-locally on $S$, either 
	\begin{enumerate}
		\item \label{T:luna6a} $BG$ has the resolution property and $X \to S$ is affine;
		\item \label{T:luna6b} $BG$ has the resolution property, $G \to S$ has connected fibers, $S$ is a normal noetherian scheme, $X$ is a scheme, and $X \to S$ is flat of finite type with geometrically normal fibers; or
		\item  \label{T:luna6c} there exists a $G$-equivariant locally closed immersion $X \hookrightarrow \PP(V)$ where $V$ is a locally free $\oh_S$-module of finite rank with a $G$-action;
	\end{enumerate}
    then $W \to X'$ is a locally closed immersion.  
\end{enumerate}
\end{theorem}

In the statement above, $W \times^H G'$ denotes the quotient $(W \times G')/H$ which inherits a natural action of $G'$, and $\cN_{\sigma}$ is the conormal bundle $\cI/\cI^2$ (where $\cI$ is the sheaf of ideals in $W$ defining $\sigma$) which inherits an action of $H$.
If $H \to S$ is a flat and affine group scheme of finite presentation over an algebraic space $S$, and $X$ and $Y$ are algebraic spaces over $S$ with an action of $H$ which admit adequate GIT quotients (i.e.\ $[X/H]$ and $[Y/H]$ admit adequate moduli spaces), then an $H$-equivariant morphism $f \co X \to Y$ is called {\it strongly \'etale} if $[X/H] \to [Y/H]$ is.  

The section $\sigma\co S'\to W$ of \itemref{T:luna4} induces an $H$-equivariant
section $\tilde{\sigma}\co S'\to X'$. This factors as $S'\to G'/H\to W\times^H
G'\to X'$.  Since the last map is \'etale, we have that
$L_{(G'/H)/X'}=\cN_{\sigma}[1]$.  The map $G'/H\to X'$ is unramified and its
image is the orbit of $\tilde{\sigma}$.  We can thus think of $\cN_{\sigma}$ as
the \emph{conormal bundle for the orbit of $\tilde{\sigma}$}. We also have an
exact sequence:
\[
0\to \cN_{\sigma}\to \cN_{\tilde{\sigma}}\to \cN_e\to 0
\]
where $e\colon S'\to G'/H$ is the unit section.

\begin{remark} A considerably weaker variant of this theorem had been established in \cite[Thm.~2]{alper-quotient}, which assumed the existence of a section $\sigma \co S \to X$ such that $X \to S$ is smooth along $\sigma$, the stabilizer group scheme $G_{\sigma}$ of $\sigma$ is smooth, and  the induced map $G/G_{\sigma} \to X$ is a closed immersion.
\end{remark}

\begin{proof}[Proof of \Cref{T:luna}]
We start by picking an \'etale morphism $(S',s')\to (S,s)$ realizing
\itemref{T:luna1} with $S'$ affine.
After replacing $S'$ with an \'etale neighborhood,
\Cref{P:deformations-of-subgroups} yields a geometrically reductive closed
subgroup scheme $H \subset G'$ such that $H_{s'} \cong G_x$. This can be made
linearly reductive if $\kar \kappa(s)>0$ or $s$ has an open neighborhood of
characteristic zero (\Cref{P:extension-linfund}). This settles \itemref{T:luna2}.

We apply the main theorem (\Cref{T:base}) to $([X'/G'],x')$ and
$h_0\co \cW_0=BG_x\cong \cG_{x'}$ where $x'$ also denotes the
image of $x'$ in $[X'/G']$. This gives us a fundamental stack $\cW$ and an
\'etale morphism $h\co (\cW,w)\to ([X'/G'],x')$ such that
$\cG_{w}=BG_x$.

Since $G\to S$ is smooth, so is $G'/H \to S'$ and $[X'/H]\to [X'/G']$. The point
$x'\in X'$ gives a canonical lift of $\cG_{w}=BG_x\to [X'/G']$ to
$\cG_{w}=BG_x\to [X'/H]$.  After replacing $S'$ with an \'etale
neighborhood, we can thus lift $h$ to a map $q\co (\cW,w)\to
([X'/H],x')$ (\Cref{P:extension-morphisms}). This map is unramified
since $h$ is \'etale and $[X'/H]\to [X'/G']$ is representable. After replacing
$\cW$ with an open neighborhood, we can also assume that
$\cW\to [X'/H]\to BH$ is affine
by \Cref{P:refinement}\itemref{P:refinement:affine_diag}. Thus $\cW=[W/H]$
where $W$ is affine and $q$ corresponds to an $H$-equivariant unramified map
$W\to X'$.  Note that since $w\in |\cW|$ has stabilizer $H_{s'}$,
there is a unique point $w\in |W|$ above $w\in |\cW|$. This establishes
\itemref{T:luna1}--\itemref{T:luna3}.

If $X\to S$ is smooth at $x$, then $\cW\to S'$ is smooth at $w$
and \itemref{T:luna4} follows from
\Cref{P:smooth-refinement} applied to $\cW\to BH\to S'$. Note that unless $H$
is smooth it is a priori not clear that $W\to S'$ is smooth. But the section
$\sigma\co S'\to W$ is a regular closed immersion since it is a pull-back
of the regular closed immersion $BH\inj \cW$ given by \Cref{P:smooth-refinement}.
It follows that $W$ is smooth in a neighborhood of $\sigma$.

If $[X/G]$ has an adequate moduli space, then $\cW\to [X/G]$ becomes strongly
\'etale after replacing $\cW$ with a saturated open neighborhood by Luna's
fundamental lemma (\Cref{L:fundamental-lemma}). This establishes
\itemref{T:luna5}.

Finally, for \itemref{T:luna6} we will construct a new $W$, not relying
on the main theorem via~\itemref{T:luna3}.
By a limit argument we may assume that $(S,s)$ is
henselian local. In particular, $H$ is linearly reductive. If \itemref{T:luna6b}
holds, then there exists a $G$-quasi-projective $G$-invariant open neighborhood
$U\subseteq X$ of $x$~\cite[Thms.~3.9 and 2.5]{sumihiro2}. Thus, cases \eqref{T:luna6a} and \eqref{T:luna6b} both reduce to case \eqref{T:luna6c} after resolving a coherent sheaf by a vector bundle.

As $H$ is linearly reductive, there exists an $H$-semi-invariant function $f
\in V=\Gamma(\PP(V), \oh(1))$ not vanishing at $x$.
Then $\PP(V)_f$ is an
$H$-invariant affine open neighborhood. Applying \Cref{P:smooth-refinement} to
$[\PP(V)_f/H]\to BH\to S$ gives an affine open $H$-invariant neighborhood
$U\subseteq \PP(V)_f$,
a section $\tilde{\sigma}\co BH\to [U/H]$ and a strongly \'etale
morphism $U\to \VV(\cN_{\tilde{\sigma}})$.  We now consider the composition
$\sigma\co BH\to [U/H]\to [\PP(V)/G]$ which is unramified since $\sigma_s$
is a closed immersion and $S$ is local. This gives
the exact sequence
\[
0\to \cN_\sigma\to \cN_{\tilde{\sigma}}\to \Omega_{BH/BG}\to 0.
\]
Since $H$ is linearly reductive, this sequence splits. After choosing a
splitting, we obtain an $H$-equivariant closed subscheme $\VV(\cN_\sigma)\inj
\VV(\cN_{\tilde{\sigma}})$ and by pull back, an $H$-equivariant closed subscheme
$W\inj U$. By construction $[W/H]\to [U/H]\to [\PP(V)/G]$ is \'etale at
$x$. Finally, we replace $W$ with an affine open $H$-saturated neighborhood of
$x$ in the quasi-affine scheme $W\cap X$.
\end{proof}

\subsection{Existence of henselizations}\label{SS:henselizations}
Let $\cX$ be an algebraic stack with affine stabilizers and let $x\in |\cX|$ be
a point with linearly reductive stabilizer.  We have already seen that the
completion $\widehat{\cX}_x$ exists if $\cX$ is noetherian
(\Cref{C:existence-completions:general}). In this section we will prove that
there also is a henselization $\cX^h_x$ when $\cX$ is of finite presentation over
an algebraic space $S$ and $\kappa(x)/\kappa(s)$ is finite.

We say that an algebraic stack $\stG$ is a {\it one-point gerbe} if $\stG$ is noetherian
and an fppf-gerbe over the spectrum of a field $k$, or, equivalently, if 
 $\stG$ is reduced, noetherian and $|\stG|$ is a one-point
space.
A morphism $\cX\to \stY$ of algebraic stacks is called \emph{pro-\'etale} if $\cX$ is the inverse limit
of a system of quasi-separated \'etale morphisms $\cX_\lambda\to \stY$ such that $\cX_\mu\to
\cX_\lambda$ is affine for all sufficiently large $\lambda$ and all $\mu\geq
\lambda$.

Let $\cX$ be an algebraic stack and let $x\in |\cX|$ be a point. Consider the inclusion $i \co \stG_x \inj \cX$ of the residual gerbe of $x$.  Let $\nu\colon \stG\to \stG_x$ be a
pro-\'etale morphism of one-point gerbes.  The {\it henselization of $\cX$ at $\nu$} is by definition an initial object in the 
$2$-category of $2$-commutative diagrams
\begin{equation} \label{E:henselization}
\vcenter{\xymatrix{%
\stG\ar[r]\ar[rd]_{\nu} & \cX'\ar[d]^f \\
& \cX
}}%
\end{equation}
where $f$ is pro-\'etale (but not necessarily representable even if $\nu$ is representable). 
If $\nu \co \cG_x \to \cG_x$ is the identity, we say that $\cX^h_x := \cX^h_{\nu}$ 
 is the {\it henselization at $x$}.

\begin{proposition}[Henselizations for stacks with good moduli spaces]\label{P:henselizations-gms}
Let $\cX$ be an algebraic stack with affine diagonal and good moduli space $\pi \co \cX \to X$ of finite presentation (e.g., $\cX$ noetherian). If $x \in |\cX|$ is a point such that $x \in |\cX_{\pi(x)}|$ is closed, then the henselization $\cX^h_x$ of $\cX$ at $x$ exists. Moreover
\begin{enumerate}
\item $\cX^h_x=\cX \times_X \Spec \oh_{X, \pi(x)}^h$;
\item $\cX^h_x$ is linearly fundamental; and
\item $(\cX^h_x,\cG_x)$ is a henselian pair.
\end{enumerate}
\end{proposition}

\begin{proof}
Let $\cX^h_x:=\cX \times_X \Spec \oh_{X, \pi(x)}^h$ which has good moduli space $\Spec \oh_{X, \pi(x)}^h$.
The pair $(\cX^h_x,\cG_x)$ is henselian (\Cref{T:henselian-pair-gms})
and linearly fundamental (\Cref{T:etale-local-gms}). It thus satisfies the hypotheses of \Cref{Setup:deformations}\itemref{SetupI:deformations:fcpcn}. To see that it is the henselization,
we note that \Cref{P:deformation-sections}
trivially extends to pro-\'etale morphisms $\cX'\to \cX$ and implies that a section $\cG_x \to \cX' \times_{\cX} \cG_x$ extends to a unique $\cX$-morphism $\cX_x^h \to \cX'$.
\end{proof}

\begin{remark}
Recall that if $\cX$ has merely separated diagonal, then it has affine diagonal
(\Cref{T:etale-local-gms}). If $\cX$ does not have separated diagonal, it
is still true that $(\cX \times_X \Spec \oh_{X, \pi(x)}^h,\cG_x)$ is a
henselian pair but it need not be the henselization.
In \Cref{E:non-sep-counter-example} the pair $(\cY,B\ZZ/2\ZZ)$ is henselian
with non-separated diagonal and the henselization map $\cX\to \cY$ is
non-representable.
\end{remark}

\begin{theorem}[Existence of henselizations] \label{T:henselizations-general}
Let $S$ be a quasi-separated algebraic space. Let $\cX$ be an algebraic stack, 
locally of finite presentation and quasi-separated over $S$, 
with affine stabilizers. 
Let $x\in |\cX|$ be a point such that the residue field extension
$\kappa(x)/\kappa(s)$ is finite and let $\nu\colon \stG\to
\stG_x$ be a pro-\'etale morphism such that $\stG$ is a one-point
gerbe with linearly reductive stabilizer. Then the henselization $\cX^h_\nu$ of
$\cX$ at $\nu$ exists. Moreover, $\cX^h_\nu$ is a linearly fundamental algebraic stack
and $(\cX^h_\nu,\stG)$ is a henselian pair.
\end{theorem}

\begin{remark}
If $x \in |\cX|$ has linearly reductive stabilizer, the theorem above shows that the henselization $\cX_x^h$ of $\cX$ at $x$ exists and moreover that $\cX_x^h$ is linearly fundamental and $(\cX_x^h,\cG_x)$ is a henselian pair.
\end{remark}

\begin{proof}[Proof of \Cref{T:henselizations-general}]
By definition, we can factor $\nu$ as
$\stG\to \stG_1\to \stG_x$ where $\nu_1\colon \stG\to \stG_1$ is pro-\'etale and affine
and $\stG_1\to \stG_x$ is \'etale. We can also arrange so that
$\stG_1$ is a one-point gerbe and $\stG\to \stG_1$
is stabilizer-preserving. Then $\stG=\stG_1\times_{k_1} \Spec k$ where $k/k_1$
is a separable algebraic field extension and $\stG_1$ has linearly reductive stabilizer.

By \Cref{T:base} we can find a fundamental stack $\cW$, a closed point $w\in
|\cW|$ and an \'etale morphism $(\cW,w)\to (\cX,x)$ such that $\stG_w=\stG_1$.  Then
$\stW^h_w=\cW \times_W \Spec \oh_{W, \pi(w)}^h$, where $\pi\colon \cW\to W$
is the adequate moduli space. Indeed, $\cW \times_W \Spec \oh_{W, \pi(w)}^h$
is linearly fundamental (\Cref{C:adequate+lin-red=>good:fundamental}) so
\Cref{P:henselizations-gms} applies. Finally, we obtain
$\stX^h_\nu=\stW^h_{\nu_1}$ by base changing $\stW^h_w\to W^h_{\pi(w)}$ along a
pro-\'etale morphism $W'\to W^h_{\pi(w)}$ extending $k/k_1$.
\end{proof}

\subsection{\'Etale-local equivalences} \label{SS:etale-local-equivalences}

\begin{theorem} \label{T:etale-local-equivalences}
  Let $S$ be a quasi-separated algebraic space.  Let $\cX$ and $\cY$ be algebraic stacks, locally of finite presentation and quasi-separated over $S$, with affine stabilizers.  Suppose $x \in |\cX|$ and $y \in |\cY|$ are points with  linearly reductive stabilizers above a point $s \in |S|$ such that $\kappa(x)/\kappa(s)$ and $\kappa(y)/\kappa(s)$ are finite.  Then the following are equivalent:
  \begin{enumerate}[label=(\arabic*),ref=\arabic*]
  	\item\label{TI:etale:henselization}
	  There exists an isomorphism $\cX^{h}_x \to \cY^{h}_y$ of henselizations.
  	\item\label{TI:etale:etale}
	  There exists a diagram of  \'etale pointed morphisms
$$
	\xymatrix@R=3mm{
			& ([\Spec A/ \GL_n], w) \ar[ld]_-f \ar[rd]^-g \\
		(\cX,x) &  & (\cY,y) }
$$
	such that both $f$ and $g$ induce isomorphisms of residual gerbes at $w$.
  \end{enumerate}
If $S$ is quasi-excellent, then the conditions above are also equivalent to:
\begin{enumerate}[label=(\arabic*$'$),ref=\arabic*$'$]
  	\item \label{TI:etale:completion}
	  There exists an isomorphism $\hat{\cX}_x \to \hat{\cY}_y$ of  completions.
\end{enumerate}
\end{theorem}

\begin{proof}
  The implications \itemref{TI:etale:etale} $\Longrightarrow$ \itemref{TI:etale:henselization} and  \itemref{TI:etale:etale} $\Longrightarrow$ \itemref{TI:etale:completion} are clear.
  For the converses, we may assume that $S$, $\cX$ and $\cY$ are quasi-compact.
  For \itemref{TI:etale:henselization} $\Longrightarrow$ \itemref{TI:etale:etale}, since $\cY\to S$ is locally of finite presentation, we obtain a factorization $\cX^h_x\to \cW\to \cX$, where the second map is \'etale, such that $\cX^h_x\cong\cY^h_y\to \cY$ factors via $\cW$. The induced map $\cW\to \cY$ is flat and unramified at the image $w$ of $x$, hence \'etale at $w$. After passing to a further \'etale neighborhood, $\cW$ is fundamental and \itemref{TI:etale:etale} follows.
  For \itemref{TI:etale:completion} $\Longrightarrow$ \itemref{TI:etale:etale}, the argument of \cite[Thm.~4.19]{luna-field} is valid if one applies \Cref{T:base} instead of \cite[Thm.~1.1]{luna-field}.
\end{proof}

\subsection{Algebraicity results} \label{SS:algebraicity-results} Here
we generalize the algebraicity results of \cite[\S5.3]{luna-field} to
the setting of mixed characteristic. We will do this using the
formulation of Artin's criterion in \cite[Thm.~A]{MR3589351}. This
requires us to prove that certain deformation and obstruction functors
are coherent, in the sense of \cite{MR0212070} (cf.\
\cite{hallj_coho_bc,MR1656482,MR3589351}), a definition that we
briefly recall. Let $X$ be an affine scheme; then an additive functor
$F \colon \QCOH(X) \to \AB$ is \emph{coherent} if there exists a
morphism of quasi-coherent $\Orb_X$-modules $\phi \colon M \to N$ such
that
\[
  F(-) \simeq \coker(\Hom_{\Orb_X}(N,-) \to \Hom_{\Orb_X}(M,-)).
\]
The following result generalizes \cite[Prop~5.14]{luna-field} to the setting of mixed characteristic.
\begin{proposition}\label{P:coherence}
  Let $\cX$ be a noetherian algebraic stack with affine diagonal and affine
  good moduli space $X$.
  If $\cplx{F} \in \DQCOH(\cX)$ and $\cplx{G} \in \DCAT^b_{\COH}(\cX)$, then the functor
  \[
    \Hom_{\Orb_{\cX}}(\cplx{F},\cplx{G} \otimes^{\LDERF}_{\Orb_{\cX}} \LDERF \pi^*(-)) \colon \QCOH(X) \to \QCOH(X)
  \]
  is coherent.
\end{proposition}
\begin{proof}
  The proof is identical to \cite[Prop.~5.14]{luna-field}: by
  \Cref{P:compact-generation}, $\DQCOH(\cX)$ is compactly
  generated. Also, the restriction of
  $\RDERF (f_{\qcsubscript})_* \colon \DQCOH(\cX) \to \DQCOH(X)$ to
  $\DCAT_{\COH}^+(\cX)$ factors through $\DCAT^+_{\COH}(X)$
  \cite[Thm.~4.16(x)]{alper-good}. By
  \cite[Cor.~4.19]{perfect_complexes_stacks}, the result follows.
\end{proof}
For this subsection, we will now assume that we are in the following situation:
\begin{setup} \label{setup:algebraicity}
  Fix an algebraic space $X$ and an algebraic stack $\cX$ with affine diagonal
  over $X$, such that $\cX \to X$ is a good moduli space of finite
  presentation. Assume that $X$ is quasi-excellent or $\cX$ satisfies one of the
  conditions \ref{Cond:FC}, \ref{Cond:PC} or \ref{Cond:N}. Note that if $\cX$ is
  noetherian, then $\cX \to X$ is automatically of finite type
  \cite[Thm.~A.1]{luna-field}.
\end{setup}

The following corollary is a mixed characteristic variant of
\cite[Cor.~5.15]{luna-field}.
\begin{corollary}[Hom scheme]\label{C:affine-hom}
  Let $\cX\to X$ be a good moduli space as in \Cref{setup:algebraicity}.
  Let $\cF$ be a quasi-coherent $\Orb_{\cX}$-module. Let $\cG$ be a finitely presented $\Orb_{\cX}$-module. If $\cG$ is flat over $X$, then the $X$-presheaf $\Homshf_{\Orb_{\cX}/X}(\cF,\cG)$, whose objects over $T\xrightarrow{\tau} X$ are homomorphisms $\tau_{\cX}^*\cF \to \tau_{\cX}^*\cG$ of $\Orb_{\cX\times_X T}$-modules (where $\tau_{\cX} \colon \cX \times_X T \to \cX$ is the projection), is representable by an affine $X$-scheme.
\end{corollary}
\begin{proof}
  The question is \'etale-local on $X$, so we may assume that $X$ is an affine scheme.
  Since $\cX\to X$ is of finite presentation, we can write $\cF$ as a filtered colimit
  of quasi-coherent modules $\cF_\lambda$ of finite presentation. Then
  $\Homshf_{\Orb_{\cX}/X}(\cF,\cG)=\varprojlim_\lambda \Homshf_{\Orb_{\cX}/X}(\cF_\lambda,\cG)$
  so we may assume that $\cF$ is of finite presentation.
  After a standard limit argument, using \Cref{C:gms-approximation}, we can assume
  that $X$ is noetherian. The result now follows directly from
  the coherence (\Cref{P:coherence}) and left-exactness of the functor
  $\Hom(\cF,\cG\otimes_{\Orb_{\cX}} \pi^*(-))\colon \QCOH(X)\to \QCOH(X)$,
  cf.\ \cite[Cor.~5.15]{luna-field} or \cite[Thm.~D]{hallj_coho_bc}.
\end{proof}

\begin{theorem}[Stacks of coherent sheaves]\label{T:coh}
  Let $\cX\to X$ be a good moduli space as in \Cref{setup:algebraicity}.
  The $X$-stack $\Cohstk_{\cX/X}$, whose objects over $T \to X$ are finitely presented quasi-coherent sheaves on $\cX \times_X T$ flat over $T$, is an algebraic stack, locally of finite presentation over $X$, with affine diagonal over $X$.
\end{theorem}

\begin{proof}
  After approximating to the quasi-excellent situation using \Cref{C:gms-approximation}, the proof is identical to \cite[Thm.~5.7]{luna-field}, which is a small modification of \cite[Thm.~8.1]{MR3589351}: the formal GAGA statement of \Cref{C:formal-gaga} implies that formally versal deformations are effective and \Cref{P:coherence} implies that the automorphism, deformation and obstruction functors are coherent. Therefore, Artin's criterion (as formulated in \cite[Thm.~A]{MR3589351}) is satisfied and the result follows. \Cref{C:affine-hom} implies that the diagonal is affine.
\end{proof}
Just as in \cite{luna-field}, the following corollaries follow immediately from \Cref{T:coh} appealing to the observation that \Cref{C:affine-hom} implies that $\Quotshf_{\cX/X}(\cF) \to \Cohstk_{\cX/X}$ is quasi-affine.

\begin{corollary}[Quot schemes]\label{C:quot}
  Let $\cX\to X$ be a good moduli space as in \Cref{setup:algebraicity}.
 If $\cF$ is a quasi-coherent $\oh_{\cX}$-module, then the $X$-sheaf $\Quotshf_{\cX/X}(\cF)$, whose objects over $T \xrightarrow{\tau} X$ are quotients $\tau_X^* \cF \to \cG$ (where $\tau_X \co \cX \times_X T \to \cX$ is the projection) such that $\cG$ is a finitely presented $\oh_{\cX \times_X T}$-module that is flat over $T$, is a separated algebraic space over $X$. If $\cF$ is finitely presented, then $\Quotshf_{\cX/X}(\cF)$ is locally of finite presentation over $X$. \epf
\end{corollary}
\begin{corollary}[Hilbert schemes]\label{C:hilb}
  Let $\cX\to X$ be a good moduli space as in \Cref{setup:algebraicity}.
 The $X$-sheaf $\Hilbshf_{\cX/X}$, whose objects over $T \to X$ are closed substacks $\cZ \subseteq \cX \times_X T$ such that $\cZ$ is flat and of finite presentation over $T$, is a separated algebraic space locally of finite presentation over $X$. \epf
\end{corollary}

We now establish algebraicity of Hom stacks.  Related results were established in \cite{hlp} under other hypotheses.

\begin{theorem}[Hom stacks] \label{T:hom}
  Let $\cX\to X$ be a good moduli space as in \Cref{setup:algebraicity}.
  Let $\cY$ be an algebraic stack, quasi-separated and locally of finite presentation over $X$ with affine stabilizers.
If $\cX \to X$ is flat, then the $X$-stack $\Homstk_X(\cX, \cY)$, whose objects are pairs consisting of a morphism $T \to X$ of algebraic spaces and a morphism $\cX \times_X T \to \cY$ of algebraic stacks over $X$, is an algebraic stack, locally of finite presentation over $X$ with quasi-separated diagonal.  If $\cY \to X$ has affine (resp.\ quasi-affine, resp.\ separated) diagonal, then the same is true for $\Homstk_X(\cX, \cY) \to X$.
\end{theorem}

\begin{proof}
  As before we may first assume that $X$ is affine and then reduce to the situation
  where $\cY$ is quasi-compact.
  After approximating to the quasi-excellent case using \Cref{C:gms-approximation}, the proof
  becomes identical to the proof of \cite[Thm.~5.10]{luna-field},
  which is a variant of \cite[Thm.~1.2]{hr-tannaka}.
\end{proof}

\begin{remark}[$G$-equivariant Hom stacks]
Let $G\to S$ be a group scheme acting on algebraic spaces or stacks $X$ and $Y$
over $S$. Then $G$-equivariant morphisms $X\to Y$ are equivalent to morphisms
of stacks $[X/G]\to [Y/G]$ over $BG$. It follows that the $G$-equivariant
Hom-stack $\Homstk_S^G(X,Y)$ fits into a cartesian square
\[
\xymatrix{
  \Homstk_S^G(X,Y)\ar[r]\ar[d] & \Homstk_S([X/G], [Y/G])\ar[d]\\
  S\ar[r] & \Homstk_S([X/G],BG).\ar@{}[ul]|\square
}
\]
Thus, if $G$ is linearly reductive and $X\to X\gitq G \cong S$ is a flat good
quotient, then we obtain algebraicity results for $\Homstk_S^G(X,Y)$, cf.\
\cite[Cor.~5.11]{luna-field}.
\end{remark}

\appendix
\section{Counterexamples in mixed characteristic}\label{A:mixed-char-counterexamples}

 We first recall 
the following conditions on an algebraic stack $\cW$ introduced in \Cref{S:approximation-deformation}.
\begin{enumerate}
\myitem{FC}\label{Cond:FC2} There is only a finite number of
  different characteristics in $\cW$.
\myitem{PC}\label{Cond:PC2} Every closed point of $\cW$ has
  positive characteristic.
\myitem{N}\label{Cond:N2} Every closed point of $\cW$ has nice
  stabilizer.
\end{enumerate}
We also introduce the following condition which is implied by \ref{Cond:FC2} or \ref{Cond:PC2}. 
\begin{enumerate}
\myitem{$\QQ_{\mathrm{open}}$}\label{Cond:QO}
Every closed point of $\cW$ that is of
characteristic zero has a neighborhood of characteristic zero.
\end{enumerate}

In this appendix we will give examples of schemes and linearly fundamental
stacks in mixed characteristic with various bad behavior.
\begin{enumerate}
\item A noetherian linearly fundamental stack $\cX$ with good moduli space
  $\cX\to X$ such that $X$ does not satisfy condition \ref{Cond:QO} and
  we cannot write $\cX=[\Spec(B)/G]$ with $G$ linearly
  reductive \'etale-locally on $X$ or \'etale-locally on $\cX$ (\Cref{SS:noetherian-counterexample}).
  In particular, condition \ref{Cond:QO}
  is necessary in \Cref{T:etale-local-gms} and the similar condition
  is necessary in \Cref{T:refinement}\itemref{T:refinement-lin-red-group}.
\item A non-noetherian linearly fundamental stack $\cX$ that cannot be written
  as an inverse limit of noetherian linearly fundamental stacks (\Cref{SS:nonnoetherian-counterexample1,SS:nonnoetherian-counterexample2}).
\item A noetherian scheme satisfying \ref{Cond:QO} but neither \ref{Cond:FC2} nor \ref{Cond:PC2} (\Cref{SS:examples-Cond:QO}).
\end{enumerate}
Such counterexamples must have infinitely many different characteristics
and closed points of characteristic zero.

\bigskip
Throughout this appendix, we work over the base scheme $\Spec \ZZ[\frac{1}{2}]$. Let $\SL_2$ act on
$\lasl_2$ by conjugation. Then $\cY=[\lasl_2/\SL_2]$ is a fundamental stack
with adequate moduli space $\cY \to Y:=\lasl_2\gitq \SL_2=\Spec
\ZZ[\frac{1}{2},t]$ given by the determinant. Indeed, this
follows from Zariski's main theorem and the following description of the
orbits over algebraically closed fields. For $t\neq 0$, there is a unique orbit with Jordan normal form
\[
\begin{bmatrix}
\sqrt{-t} & 0\\
0 & -\sqrt{-t}
\end{bmatrix}
\]
and stabilizer $\Gm$. For $t=0$, there are two orbits, one closed and one open,
with Jordan normal forms
\[
\begin{bmatrix}
0 & 0\\
0 & 0
\end{bmatrix}
\quad\text{and}\quad
\begin{bmatrix}
0 & 1\\
0 & 0
\end{bmatrix}
\]
and stabilizers $\SL_2$ and $\Gmu_2 \times \Ga$ respectively. The nice
locus is $Y_\nicelocus=\{t\neq 0\}$. The linearly reductive locus is
$\{t\neq 0\}\cup \AA^1_\QQ$.

\subsection{A noetherian example}\label{SS:noetherian-counterexample}
Let $A=\ZZ[\frac{1}{2},t,\frac{1}{t+p}\,:\, p\in P]\subset \QQ[t]$, where $p$
ranges over the set $P$ of all odd primes.
\begin{itemize}
\item $A$ is a noetherian integral domain: the localization of
$\ZZ[\frac{1}{2},t]$ in the multiplicative submonoid generated by $(t+p)$ for all $p \in P$.
\item $A/(t)=\QQ$.
\end{itemize}
We let $X=\Spec A$, let $X\to Y$ be the natural map (a flat monomorphism) and
let $\cX=\cY\times_Y X$. Then $\cX$ is linearly fundamental with good moduli
space $X$.

The nice locus of $X$ is $\{t\neq 0\}$ and the complement consists of a single
closed point $x$ of characteristic zero. Any neighborhood of this point
contains points of positive characteristic. It is thus impossible to write
$\cX=[\Spec B/G]$, with a linearly reductive group $G$, after restricting to
any \'etale neighborhood of $x\in X$, or more generally, after restricting to
any \'etale neighborhood in $\cX$ of the unique closed point above $x$.

\subsection{A non-noetherian example} \label{SS:nonnoetherian-counterexample1}
Let $A=\ZZ[\frac{1}{2},t,\frac{t-1}{p}\,:\, p\in P]\subset \QQ[t]$, where $p$
ranges over the set $P$ of all odd primes. Note that
\begin{itemize}
\item $A$ is a non-noetherian integral domain,
\item $A=\ZZ[\frac{1}{2},t,(x_p)_{p\in P}]/(px_p-t+1)_{p\in P}$,
\item $A\otimes_\ZZ \ZZ_{(p)}=\ZZ_{(p)}[x_p]$ is regular, and thus
  noetherian, for every $p\in P$,
\item $A\otimes_\ZZ \QQ = \QQ[t]$,
\item $A/(t)=\QQ$,
\item $A/(t-1)=\ZZ[\frac{1}{2},(x_p)_{p\in P}]/(px_p)_{p\in P}$ has infinitely
  many irreducible components: the spectrum is the union of $\Spec
  \ZZ[\frac{1}{2}]$ and $\AA^1_{\FF_p}$ for every $p\in P$, and
\item $\Spec A\to \Spec \ZZ[\frac{1}{2}]$ admits a section: $t=1$, $x_p=0$ for all $p\in P$.
\end{itemize}
We let $X=\Spec A$, let $X\to Y$ be the natural map and let $\cX=\cY\times_Y
X$. Then $\cX$ is linearly fundamental with good moduli space $X$. Note that
$\cX\to X$ is of finite presentation as it is a pull-back of $\cY\to Y$.

\begin{proposition}
There does not exist any noetherian linearly fundamental stack $\cX_\alpha$
with an affine morphism $\cX\to \cX_\alpha$.
\end{proposition}
\begin{proof}
Suppose that such an $\cX_\alpha$ exists. Then we may write $\cX=\varprojlim
\cX_\lambda$ where the $\cX_\lambda$ are affine and of finite presentation over
$\cX_\alpha$. Let $\cX_\lambda\to X_\lambda$ denote the good moduli space which
is of finite type~\cite[Thm.~A.1]{luna-field}. Thus, $\cX\to
\cX_\lambda\times_{X_\lambda} X$ is affine and of finite presentation. For all
sufficiently large $\lambda$ we can thus find an affine finitely presented
morphism $\cX'_\lambda\to \cX_\lambda$ such that $\cX\to
\cX'_\lambda\times_{X_\lambda} X$ is an isomorphism. Since also $\cX\to
\cY\times_Y X$ is an isomorphism and $Y\to \Spec \ZZ$ is of finite presentation,
it follows that there is an isomorphism
$\cX'_\lambda\to \cY\times_Y X_\lambda$ for all sufficiently large $\lambda$.

To prove the proposition, it is thus enough to show that there does not exist a
factorization $X\to X_\lambda\to Y$ with $X_\lambda$ noetherian and affine such
that $\cY\times_Y X_\lambda$ is linearly fundamental. This follows from the
following lemma.
\end{proof}

\begin{lemma}
Let $Z$ be an integral affine scheme together with
a morphism $f\colon Z\to Y=\Spec \ZZ[\frac{1}{2},t]$ such that
\begin{enumerate}
\item\label{LI:counter-ex:iso-Q}
  $f_\QQ\colon Z_\QQ\to \Spec \QQ[t]$ is an isomorphism;
\item\label{LI:counter-ex:0-Q}
  $f^{-1}(0)$ is of pure characteristic zero; and
\item\label{LI:counter-ex:1-section}
  $f^{-1}(1)$ admits a section $s$.
\end{enumerate}
Then $Z$ is not noetherian.
\end{lemma}
\begin{proof}
For $a\in \ZZ$ and $p\in P$, let $a_p$ (resp.\ $a_\QQ$) denote the point
in $Y$ corresponding to the prime ideal $(p,t-a)$
(resp.\ $(t-a)$). Similarly, let $\eta_p$ (resp.\ $\eta$) denote the points
corresponding to the prime ideals $(p)$ (resp.\ $0$).
Let $W=\Spec \ZZ[\frac{1}{2}]\inj Z$ be the image of the section $s$ and
let $1_p\in Z$ also denote the unique point of characteristic $p$ on $W$.

Suppose that the local rings of $Z$ are noetherian. We will prove that
$f^{-1}(1)$ then has infinitely many irreducible components. Since $f^{-1}(1)$
is the union of the closed subschemes $W$ and $W_p:=f^{-1}(1_p)$, $p\in
P$, it is enough to prove that $W_p$ has (at least) dimension $1$ for every $p$.

Note that $\Orb_{Z,1_p}$ is (at least) $2$-dimensional since there is a
chain $1_p\leq 1_{\QQ} \leq \eta$ of length $2$ (here we use
\itemref{LI:counter-ex:iso-Q} and \itemref{LI:counter-ex:1-section}). By
Krull's Hauptidealsatz, $\Orb_{Z,1_p}/(p)$ has (at least) dimension $1$
(here we use that the local ring is noetherian). The complement of $\Spec
\Orb_{W_p,1_p}\inj \Spec \Orb_{Z,1_p}/(p)$ maps to $\eta_p$.
It is thus enough to prove that $f^{-1}(\eta_p)=\emptyset$.

Consider the local ring $\Orb_{Y,0_p}$. This is a regular local ring of
dimension $2$. Since $f_\QQ$ is an isomorphism, $Z\times_Y \Spec
\Orb_{Y,\eta_p}\to \Spec \Orb_{Y,\eta_p}$ is a birational affine
morphism to the spectrum of a DVR. Thus, either $f^{-1}(\eta_p)=\emptyset$
or $Z\times_Y \Spec \Orb_{Y,\eta_p}\to \Spec \Orb_{Y,\eta_p}$
is an isomorphism. In the latter case,
$f^{-1}(\Spec \Orb_{Y,0_p})=f^{-1}(\Spec \Orb_{Y,0_p}\smallsetminus 0_p)\cong \Spec \Orb_{Y,0_p}\smallsetminus 0_p$ (here we use \itemref{LI:counter-ex:iso-Q} and \itemref{LI:counter-ex:0-Q}) which
contradicts that $f$ is affine.
\end{proof}

\subsection{A variant of the non-noetherian example}  \label{SS:nonnoetherian-counterexample2}
Let $A=\ZZ[\frac{1}{2},t,\frac{t-1}{p^a}\,:\, p\in P,\, a\geq 1]\subset \QQ[t]$, where $p$ ranges over the set $P$ of all odd primes. Note that
\begin{itemize}
\item $A$ is a non-noetherian integral domain,
\item $A=\ZZ[\frac{1}{2},t,(x_{p,a})_{p\in P,a\geq 1}]/(px_{p,1}-t+1,px_{p,a+1}-x_{p,a})_{p\in P,a\geq 1}$,
\item $A\otimes_\ZZ \ZZ_{(p)}=\ZZ_{(p)}[(x_{p,a})_{a\geq 1}]/(px_{p,a+1}-x_{p,a})_{a\geq 1}$ is two-dimensional and integral but not noetherian, for every $p\in P$,
\item $A\otimes_\ZZ \QQ = \QQ[t]$,
\item $A/(t)=\QQ$,
\item $A/(t-1)=\ZZ[\frac{1}{2},(x_{p,a})_{p\in P,a\geq 1}]/(px_{p,1},px_{p,a+1}-x_{p,a})_{p\in P,a\geq 1}$ is non-reduced with one irreducible component:
  the nil-radical is $(x_{p,a})_{p\in P,a\geq 1}$.
\item $\Spec A\to \Spec \ZZ[\frac{1}{2}]$ admits a section: $t=1$, $x_{p,a}=0$ for all $p\in P$, $a\geq 1$.
\end{itemize}
As in the previous subsection, the fiber product $\Spec A \times_{Y} \cY$ (where
$\cY=[\lasl_2/\SL_2]$ and $Y=\lasl_2\gitq \SL_2 =\Spec \ZZ[\frac{1}{2},t]$)
provides an example of a linearly fundamental stack that cannot be written as an
inverse limit of noetherian linearly fundamental stacks.

\subsection{Condition \ref{Cond:QO}}  \label{SS:examples-Cond:QO}
We provide examples illustrating that condition \ref{Cond:QO}, introduced in
the beginning of the appendix, is slightly weaker than conditions 
\ref{Cond:FC2} and \ref{Cond:PC2} even in the noetherian case.
A non-connected example is given by $S=\Spec (\ZZ \times \QQ)$ which has infinitely
many different characteristics and a closed point of characteristic zero.
A connected counterexample is given by the push-out $S=\Spec\bigl(
\ZZ\times_{\FF_p} (\ZZ_{(p)}[x]) \bigr)$ for any choice of prime number
$p$. Note that the irreducible component $\Spec( \ZZ_{(p)}[x] )$ has closed
points of characteristic zero, e.g., the prime ideal $(px-1)$. The push-out
is noetherian by Eakin--Nagata's theorem.

For an irreducible noetherian scheme, condition \ref{Cond:QO} implies
  \ref{Cond:FC2} or \ref{Cond:PC2}. That is,
  an irreducible noetherian scheme with a dense open of equal characteristic
  zero, has only a finite number of characteristics. This follows
  from Krull's Hauptidealsatz.
We also note that for a scheme of finite type over $\Spec \ZZ$, there are no closed
  points of characteristic zero so \ref{Cond:QO} and \ref{Cond:PC2} hold trivially.

\bibliography{references}
\bibliographystyle{bibstyle}
\end{document}